\documentclass[reqno, 12pt]{amsart}

\usepackage{mathrsfs}
\usepackage{amscd}
\usepackage{amsmath}
\usepackage{latexsym}
\usepackage{amsfonts}
\usepackage{amssymb}
\usepackage{amsthm}
\usepackage{graphicx}
\usepackage{hyperref}
\usepackage{makecell}
\usepackage{array}
\usepackage{booktabs}
\usepackage{multirow}

\newtheorem{theorem}{Theorem}[section]
\newtheorem{proposition}[theorem]{Proposition}
\newtheorem{corollary}[theorem]{Corollary}
\newtheorem{lemma}[theorem]{Lemma}
\newtheorem{definition}[theorem]{Definition}

\newtheorem{remark}[theorem]{Remark}

\numberwithin{equation}{section}

\title[Decay estimates for higher order elliptic operators ]
{Decay estimates for higher order elliptic operators}

\author{Hongliang Feng, \ Avy Soffer, \ Zhao Wu \ and Xiaohua Yao}

\address {Hongliang Feng, School of Mathematics, Sun Yat-sen University, Guangzhou, 510275, P.R. China}
\email{fenghongliang@aliyun.com}
\address{Avy Soffer, School of Mathematics and Statistics, Central China
Normal University, Wuhan, 430079, P.R. China\\
On leave from Rutgers University}
\email{soffer@math.rutgers.edu}
\address{Zhao Wu, School of Mathematics and Statistics, Central China Normal University, Wuhan, 430079, P.R. China}
\email{wuzhao218@yahoo.com}
\address{Xiaohua Yao, Department of Mathematics and  Hubei Province Key Laboratory of Mathematical Physics, Central China Normal University, Wuhan, 430079, P.R. China}
\email{yaoxiaohua@mail.ccnu.edu.cn}
\date{\today}

\begin{document}

	\begin{abstract}\baselineskip=13pt
		This paper is mainly devoted to study time decay estimates of the higher-order Schr\"{o}dinger type operator $H=(-\Delta)^{m}+V(x)$ in $\mathbf{R}^{n}$ for $n>2m$ and $m\in\mathbf{N}$. For certain decay potentials $V(x)$, we first derive the  asymptotic expansions of  resolvent $R_{V}(z)$ near zero threshold with the presence of zero resonance or zero eigenvalue, as well identify the resonance space for each kind of zero resonance which displays different effects on time decay rate.  Then we establish   Kato-Jensen type estimates and local decay estimates for higher order Schr\"odinger propagator $e^{-itH}$  in  the presence of zero resonance or zero eigenvalue. As a consequence,  the endpoint Strichartz estimate and $L^{p}$-decay estimates can also be obtained.  Finally, by a virial argument, a criterion on the absence of positive embedding eigenvalues is given for  $(-\Delta)^{m}+V(x)$ with a repulsive potential.
	\end{abstract}
		\maketitle

	\baselineskip=15pt
	\section{Introduction}
	
	\subsection{Backgrounds and problems}
	
	Consider the higher-order Schr\"odinger type operators in $\mathbf{R}^{n}$:
	$$H=(-\Delta)^{m}+V(x),\,\, H_{0}=(-\Delta)^{m},$$
where $n>2m$, $m\in\mathbf{N}$ and $V(x)$ is a real-valued function  satisfying $|V(x)|\lesssim (1+|x|)^{-\beta}$ for some $\beta>0$.
It was well-known that the higher-order elliptic operator $P(D)+V$ has been extensively studied as general Hamiltonian operator by many people in different contexts. For instance, one can see Schechter \cite{Schechter} for spectral theory, Kuroda \cite{Kur}, Agmon \cite {Agmon}, H\"ormander \cite{H2} for scattering theory, Davies \cite{Davies}, Davies and Hinz \cite{DaHi}, Deng et al \cite{DDY} for semigroup theory,  and  as well \cite{Herbst-Skibsted-Adv-2015, Herbst-Skibsted-Adv-2017, BS, Mourre, SYY} for many other interesting studies.

 In this paper, we are interested in establishing some dispersive estimates for the higher-order Schr\"{o}dinger type operator $H=(-\Delta)^{m}+V$  with $m\ge 2$, among which, including {\it Kato-Jensen type estimates, local decay estimates, Strichartz estimate and $L^{p}$-decay estimates}. We also show that the presence of zero resonance or eigenvalue of $H$ will affect the time decay rate of $e^{-itH}$. For classical Schr\"odinger operator $-\Delta+V$ (i.e. $m=1$), recall that  in the last thirty years, dispersive estimates of Schr\"odinger operator have been one of the key topics, which were applied broadly to  nonlinear Schr\"odinger equations, see e.g. \cite{Cazenave, TT2, JSS, KeelTao, Schlag, Simon-Review-1, Simon-Review-2} and  references therein.

In the sequel, the basic estimate is the Kato-Jensen type estimate for $H=(-\Delta)^m+V$.  The key idea is motivated by Kato and Jensen \cite{JK}. We first the splitting trick
\begin{equation*}
R_{0}(z)=\big[(-\Delta)^{m}-z\big]^{-1}=\frac{1}{mz}\sum_{\ell=0}^{m-1}z_\ell\big(-\Delta-z_\ell\big)^{-1},\ \ \ z_\ell=z^{\frac{1}{m}}e^{i\frac{2\ell\pi}{m}}
\end{equation*}
to obtain  informations of $R_{0}(z)$ by the known results of Laplacian, and then use them to study the three types behaviors of resolvent $R_{V}(z)=\big((-\Delta)^{m}+V-z\big)^{-1}$: {\it Resolvent asymptotic expansions at zero energy,  Higher energy resolvent estimates and Limiting absorption principle of $R_{V}(z)$}.  Among these studies, the spectral analysis of the operators $(-\Delta)^{m}+V$ is an indispensable part. In particular, the classifications of zero energy (eigenvalue/resonance) play distinguished roles in Kato-Jensen decay estimates. Local decay estimate can be established as a corollary of the uniformly estimate for resolvent $R_{V}(z)=\big((-\Delta)^{m}+V-z\big)^{-1}$  on the weighted $L^2$ spaces $B(L^2_{s}, L^{2}_{-s})$ with suitable $s>0$. Finally, with the help of Kato-Jensen type estimate, we also establish the endpoint Strichartz estimates, which are very useful to nonlinear dispersive problems.

To further understand our results in this paper, let us first review the famous work \cite{JK} for Schr\"odinger operator $-\Delta+V$ in $\mathbf{R}^3$, where Jensen and Kato established the time asymptotic expansion:
	\begin{equation}\label{asmpexp}
	e^{-it(-\Delta+V)}=\sum_{j=0}^{N}e^{-it\lambda_{j}}P_{j}+t^{-1/2}C_{-1}+t^{-3/2}C_{0}+\cdots
	\end{equation}
	as $t\rightarrow\infty$  on the polynomially weighted $L^2$ spaces $B\big(L^{2}_{s}, L^{2}_{-s'}\big )$ with suitable $s, s'>0$. Here the $\lambda_{j}$ are negative eigenvalues of $-\Delta+V$ with associated eigen-projection $P_{j}$.  In fact, Rauch in \cite{Rauch-CMP-1978} had already obtained the expansion for $-\Delta+V$  on the exponentially weighted $L^2$ spaces.  The spectral property at zero threshold of $-\Delta+V$ affects the leading term of the asymptotic expansion.  In \cite{JK}, they pointed out that $C_{-1}=0$  if $0$ is a regular point, and the operator $C_{-1}$ does not vanish  if $0$ is purely a resonance of $-\Delta+V$. Here, resonance means that there exists a solution to $(-\Delta+V)\psi=0$ in the distributional sense with $\psi\in L^{2}_{-s}\setminus L^{2}$ for some $s>0$.
	In particular, the following four cases were discussed: {\it zero is a regular point; zero is purely a resonance; zero is purely an eigenvalue; zero is both resonance and eigenvalue}. Zero is a  regular point of $-\Delta+V$ means zero is neither an eigenvalue nor a resonance.

To obtain \eqref{asmpexp}, their fundamental idea in \cite{JK} is to deduce the lower energy expansion and higher energy decay estimate of resolvent $R(-\Delta+V; z)$.
The lower energy expansion means the asymptotic expansion of perturbed resolvent $R(-\Delta+V; z)$ as $z$ near zero in $B\big(L^{2}_{s}, L^{2}_{-s'}\big )$. Higher energy decay estimate means the decay estimate of $R(-\Delta+V; z)$ as $z\rightarrow \infty$ in $B\big(L^{2}_{s}, L^{2}_{-s'}\big )$.
As a consequence, the following pointwise decay estimates can be obtained:
	\begin{equation}\label{Ka-Jen}
	\Big\|(1+|x|)^{-\sigma} e^{-it(-\Delta+V)}P_{ac}(-\Delta+V)	(1+|x|)^{-\sigma}\Big\|_{L^2(\mathbf{R}^3)\rightarrow L^2(\mathbf{R}^3)}\lesssim (1+|t|)^{-3/2}
	\end{equation}
provided that zero is a regular point of $-\Delta+V$.   The formula \eqref{asmpexp} also shows that, in the presence of zero resonance or eigenvalue, the time decay rate of Kato-Jensen decay estimates is slower than \eqref{Ka-Jen} in the regular case. Indeed,  if zero is purely a resonance of $-\Delta+V$, then the time decay rate is $(1+|t|)^{-1/2}$ by \eqref{asmpexp}.   For the similar results of $n=4$ and $n\geq 5$, see Jensen's works \cite{J1} and \cite{J} respectively. Also see \cite{JN} for $n=1$ and $n= 2$.

Soon afterward, Murata in \cite{MM}  generalized Kato and Jensen's  work \cite{JK} to a certain class of $P(D)+V$. In \cite{MM}, Murata required that $P(D)$ satisfies
	\begin{equation}\label{nondegenerate}
	\big(\nabla P\big)(\xi_0)=0, \ \ \ \  \det\big[\partial_{i}\partial_{j}P(\xi)\big]\Big|_{\xi_0}\neq0.
	\end{equation}
	However, the polyharmonic operators $H_0=(-\Delta)^m$ do not satisfy the  nondegenerate condition \eqref{nondegenerate} at zero but the case $m=1$.  For the fourth order Schr\"odinger operator $(-\Delta)^2+V$, i.e. the case $m=2$,  the first two authors and the last author in \cite{FSY} established the Kato-Jensen decay estimates with the time decay rate is $(1+|t|)^{-n/4}$ for $n\geq5$ and $(1+|t|)^{-5/4}$ for $n=3$ in the regular case. Recently, the authors in \cite{FWY} further studied Kato-Jensen decay estimates in non-regular case for $d\ge5$.

Therefore, in the following, we mainly deal with the higher order Schr\"odinger operator $(-\Delta)^m+V$ for all $m\geq3$. For instance, when $3m-1\leq n\leq 4m$ and  $n$ is odd , we have established the complete asymptotic formulas even in the presence of zero resonance or eigenvalue as $t\rightarrow \infty$ (see Subsection \ref{main result} below):
	\begin{equation}\label{Higher order}
	e^{-itH}P_{ac}(H)=\begin{cases}
	|t|^{-n/2m}\bar{B}+o( |t|^{-n/2m}),\,\,\,  & 0\,\, \text{ is a regular point};\\
	|t|^{-(2-\frac{n}{2m}-\frac{j-1}{m})}\bar{B}_{j}+o\big(|t|^{-(2-\frac{n}{2m}-\frac{j-1}{m})}\big),\,\,\,  & 0\,\, \text{ is the} \ $j-$\text{th kind resonance};\\
	|t|^{-1/2m}\bar{B}_{k+1}+o( |t|^{-1/2m}),\,\,\,  & 0\,\, \text{ is an eigenvalue},
	\end{cases}
	\end{equation}
where $\bar{B},~ \bar{B}_{j},~ \bar{B}_{k+1} \in B(s, -s')$ for suitable $s, s'>0$ and $\bar{B}_{j},~ \bar{B}_{k+1}$ are finite rank.

In order to prove the estimate \eqref{Higher order},  we will use the following spectral formula:
$$e^{-itH}P_{ac}(H)=\int_{0}^{\infty}e^{-it\lambda}E'(\lambda) d\lambda$$
where $E(\lambda)$ is the spectral projection of higher order operator $H=(-\Delta)^m+V$.  The key step is to get the asymptotic expansion of spectral density $E'(\lambda)$ at zero.  Note that by Stone's formula, we have (formally)
 $$E'(\lambda)=\frac{1}{2\pi i}\big(R_{V}(\lambda+i0)-R_{V}(\lambda-i0)\big),$$
which will make us to study the resolvent asymptotic expansion of $R_V(z)=\big((-\Delta)^m+V-z\big)^{-1}$ at zero by the following symmetric identity:
	\begin{equation*}
	R_{V}(z)=R_{0}(z)-R_{0}(z)v\big(U+vR_{0}(z)v\big)^{-1}vR_{0}(z).
	\end{equation*}
	 Here $v(x)=|V(x)|^{1/2}$ and $U={\rm sign}\big(V(x)\big)$.

Among these processes above,   the difficult point is how to establish the asymptotic expansion of  the inverse operator $\big(U+vR_{0}(z)v\big)^{-1}$ with the presence of zero resonance or eigenvalue.
Due to the degenerate of $H_{0}=(-\Delta)^m$ at zero, the classifications of zero resonance for $H=(-\Delta)^m+V$ are more complex than  $-\Delta+V$ with spatial dimension $n\geq3$. We need to iterate a few more steps than Schr\"odinger operator. The additional iterative steps lead us to split the resonance space into several subspaces. Heuristically, we need to split the whole resonance space into serval subspaces depending on the process of  the inverse expansion of $U+vR_{0}(z)v$ as $z$ approaches zero. {\it In this paper, we will devote the last Section\ref{proof} to give the proof on the zero asymptotic expansion and zero resonance classifications, which is long and complex.}  Recall that,  for $-\Delta+V$ in 3 and 4 dimensional  do not need to split the resonance space, see \cite{JK, J} respectively. However, for $-\Delta+V$ in 2-dimensional, Jensen and Nenciu in \cite{JN} need to  split the resonance space.

Besides zero resonance or eigenvalue,  in the decay estimates \eqref{Higher order}, we also assume that  $H=(-\Delta)^{m}+V$ does not exist any positive eigenvalue embedded into the absolutely continuous spectrum.
For Schr\"odinger operator $-\Delta+V$, Kato in the famous work \cite{K} first showed that $-\Delta+V$ has no positive eigenvalues if the potential $V(x)$ decays fast enough at infinity (e.g. $o(|x|^{-1})$). The result was  further generalized by Agmon \cite{Agmon1}, Simon \cite{Simon3}, Froese, Herbst, M. Hoffmann and T. Hoffmann \cite{FHHH} et al.
In particular, Ionescu and Jerison in \cite{IJ}  showed  such  criterion on absence of positive eigenvalues of  Schr\"{o}dinger operator with integrable potentials $V\in L^{n/2}(\mathbf{R}^{n})$.  Koch and Tataru \cite{KoTa} proved the same result for  $V\in L^{(n+1)/2}(\mathbf{R}^{n})$, where $(n+1)/2$ is the highest possible integrable exponent due to the counterexample in \cite{IJ}.

For higher order operator $P(D)+V$, the situations are much more complicate than the second order operator. Since there exit some examples even with compactly supported smooth potentials such that the positive eigenvalues appear, so it would be interesting and useful to establish some effective criterion for $P(D)+V$ on absence of positive embedded eigenvalue.  In the sequel, a simple criterion can be given on absence of positive eigenvalues for the higher order operators by a virial argument, which works for repulsive potentials. Besides, we also notice that for a general selfadjoint operator $\mathcal{H}$ on $L^{2}(\mathbf{R}^n)$, even if $\mathcal{H}$ has a simple embedded eigenvalue $\lambda_{0}$, Costin and Soffer in \cite{CoSo} have proved that $\mathcal{H}+\epsilon W$ can kick off the eigenvalue in a small interval around $\lambda_{0}$ under certain small perturbation of potential.

	\subsection{Main results}\label{main result}
	In the subsection, we  will state our main results. To the end, some notations are listed as follows:
For $a\in\mathbf{R}$, we write $a\pm$ to mean $a\pm\epsilon$ for any small $\epsilon>0$. $[a]$ denotes the largest integer less than or equal to $a$.  For $s, s'\in\mathbf{R}$, $B(s, s')$ denote the space of the bounded operators from $L^{2}_{s}(\mathbf{R}^n)$ to $L^{2}_{s'}(\mathbf{R}^n)$,  where $L^{2}_{s}(\mathbf{R}^n)$ is the weighted $L^2$ space:
	$$L^{2}_{s}(\mathbf{R}^n)=\Big\{f\in \mathscr{S}^{\prime}(\mathbf{R}^n): (1+|\cdot|)^{s}f\in L^{2}(\mathbf{R}^n)\Big\}.$$

Let us first give Kato-Jensen estimates for $H=(-\Delta)^m+V(x)$  assuming that zero is neither a resonance nor an eigenvalue.
\begin{theorem}\label{Kato-Jensen decay regular}
Let $n>2m$ and $H=(-\Delta)^m+V(x)$ with $|V(x)|\lesssim (1+|x|)^{-\beta}$ for some $\beta>n$.  Assume that $H$ has no positive embedded eigenvalue and zero is a regular point. $P_{ac}(H)$ denotes the projection onto the absolutely continuous spectrum space of $H$. Then for any $s, s'>n/2$, we have
\begin{equation*}
  \big\|e^{-itH}P_{ac}(H)\big\|_{B(s,-s')}\lesssim (1+|t|)^{-\frac{n}{2m}}, \,\,  t\in\mathbf{R}.
\end{equation*}
\end{theorem}	
	
It was quite known that the presence of zero resonance or zero eigenvalue affect the time decay rate of Kato-Jensen decay estimates. Recall that Schr\"odinger operator $-\Delta+V$ does not exist resonance at zero for $n>4$.	Similarly, if $n>4m$ and $m\ge2$,  then $H=(-\Delta)^m+V$ do not exist zero resonance (see Remark \ref{no resonance} below), and if $n\leq 4m$, then $H=(-\Delta)^m+V$ may have zero resonances. However, since the polyhamonic operators $(-\Delta)^m$ are degenerate at zero for $m\ge 2$,
	the resonance space for $H=(-\Delta)^m+V$ is more complex than Schr\"odinger operator $-\Delta+V$.

In the sequel, we will define resonance space of $(-\Delta)^m+V$  for $n\le 4m$ and then give different resonance types, which lead to different decay rates of time. For $\sigma\in\mathbf{R}$,  let $W_{\sigma}(\mathbf{R}^n)$ denotes the intersection  space
	$$ W_{\sigma}(\mathbf{R}^n)= \bigcap_{s>\sigma} L^{2}_{-s}(\mathbf{R}^n).$$
	Note that, $W_{\sigma_1}(\mathbf{R}^n)\supset W_{\sigma_2}(\mathbf{R}^n)$ if $\sigma_1>\sigma_2$.  Especially, $W_{0}(\mathbf{R}^n)\supset L^{2}(\mathbf{R}^n)$.

		\begin{definition}\label{resdef}
		For $n\le 4m$ and $H=(-\Delta)^m+V$ with $|V(x)|\lesssim(1+|x|)^{-\beta}$ for some $\beta>0$. If there exists some  $\psi(x)\in W_{2m-\frac{n}{2}}(\mathbf{R}^{n})\setminus L^2(\mathbf{R}^{n}) $ such that $H\psi(x)=0$ in the distributional sense,  then we say zero is  a resonance  of $H$. If $H\psi(x)=0$ for some nonzero $\psi\in L^2(\mathbf{R}^{n})$,  then we say zero is  an eigenvalue  of $H$.
%If $n$ is odd and $3m-1\leq n\leq 4m$, let $n=4m+1-2k$ with $1\le k\le [\frac{m}{2}]+1$. If $n$ is even with $3m\leq n\leq 4m$, let $n=4m+2-2k$ with $1\le k\le [\frac{m}{2}]+1$.
			\end{definition}
		
We can further define the $k$ different types of resonance, where $k$ is a positive integer by
\begin{equation}\label{type number}
		k=k(n,m)=\begin{cases}
			\frac{4m-n+1}{2},\ \ &\ \ n\le4m\ \ n\text{ odd}; \\
			\frac{4m-n+2}{2},\ \ &\ \ n\le4m\ \ n\text{ even}.
		\end{cases}
\end{equation}
For $l\in \mathbf{N}$, if  $$\psi(x)\in W_{2m-\frac{n}{2}-(l-1)}(\mathbf{R}^{n})\setminus W_{2m-\frac{n}{2}-l}(\mathbf{R}^{n}), \ 1\leq l\leq k-1, $$ satisfies $H\psi(x)=0$ in the distributional sense,  then we say zero is {\it the $l$-th kind of resonance }  of $H$. If $\psi(x)\in W_{1/2}(\mathbf{R}^{n})\setminus L^2(\mathbf{R}^{n})$ ($n$ odd) or $\psi(x)\in W_{0}(\mathbf{R}^{n})\setminus L^2(\mathbf{R}^{n})$ ($n$ even), then we say zero is {\it the $k$-th kind of resonance} of $H$.

\begin{remark}
	Note that $H\psi=0$ in the distributional sense if and only if $ (1+R_{0}(0)V)\psi=0$ with $R_{0}(0)=(-\Delta)^{-m}$. Denote
	$$\mathfrak{M}_s=\Big\{\psi\in L^{2}_{-s}: \big(1+R_{0}(0)V\big)\psi=0\Big\},\,\, \mathfrak{N}_s=\Big\{\psi\in L^{2}_{s}: \big(1+VR_{0}(0)\big)\psi=0\Big\}.$$
	%For $V(x)$ satisfies $|V(x)|\lesssim(1+|x|)^{-\beta}$ and $\psi\in  L^{2}_{-s}(\mathbf{R}^n)$, then $V\psi\in  L^{2}_{\beta-s}(\mathbf{R}^n)$. While $ (1+R_{0}(0)V)\psi=0$ implies $\psi=-R_{0}(0)V\psi$.  By the boundness of $R_{0}(0)$ in $B(s, -s')$ (see \cite[Lemma 2.3]{J} or Lemma \ref{Riesz-potential-boundedness} below), then the identity $\psi=-R_{0}(0)V\psi$ limits $s\in (2m-\frac{n}{2},\,\, \beta+\frac{n}{2}-2m)$. Furthermore, a priori that  $\mathfrak{M}$ and $\mathfrak{N}$ are depend on $s$, but they are monotone in $s$ in the opposite direction. By duality, we have $\dim(\mathfrak{M})=\dim(\mathfrak{N})<\infty$.
Since $\mathfrak{M}_s$ and $\mathfrak{N}_s$ are independent of $s\in (2m-\frac{n}{2},\,\, \beta+\frac{n}{2}-2m)$ due to the compactness of $R_0(0)V$, the zero resonance function must belong to the intersection space $W_{\sigma}(\mathbf{R}^n)=\cap_{s>\sigma} L^{2}_{-s}(\mathbf{R}^n)$ for  $\sigma=2m-n/2$.  Recall that, for Schr\"odinger operator $-\Delta+V$, $\sigma=1/2$ when $n=3$ and $\sigma=0$ when $n=4$. Also, $-\Delta+V$ has only one resonance type (i.e. $k=1$) for $n=3, 4$. When $n\geq5$, $-\Delta+V$ does not exist resonance at zero.  See  \cite{JK, J, J1}.
\end{remark}

\begin{remark}\label{no resonance}
	For $n>4m$, $H=(-\Delta)^m+V$ does not exist resonance  at zero.  Indeed, let $\psi\in L^{2}_{-s}(\mathbf{R}^n)$ with some $s>0$ satisfying $H\psi=0$.  Since the Riesz potential $R_{0}(0)=(-\Delta)^{-m}$ is bounded from $L^2_{s}(\mathbf{R}^n)$ to $L^2_{-s'}(\mathbf{R}^n)$ with $s+s'\geq 2m$ and $s, s'\geq0$, then $R_{0}(0)$ will pull $V\psi\in L^{2}_{\beta-s}(\mathbf{R}^n)$ into $L^{2}(\mathbf{R}^n)$ if $\beta-s\geq2m$ by the identity $\psi=-R_{0}(0)V\psi$.  Hence $\psi\in L^{2}(\mathbf{R}^n)$ and  $H=(-\Delta)^m+V$ does not exist zero resonance if $n>4m$.
\end{remark}

    %\begin{remark}
%    For fixed order $m$, the kind number of zero resonances depends on the dimension $n$. The number of resonance kind is increasing if the dimension $n$ close to $3m-1$ or $3m$ with odd $n$ and even $n$ respectively. Indeed, the number of terms with the power of $\mu$ less than $2m$ in the asymptotic expansion of $R_{0}(\mu^{2m})$ highly depends on dimension $n$. Furthermore, the number of such terms decide the steps we needed to iterate.
%    \end{remark}

Now, we state the Kato-Jensen decay estimates of $H=(-\Delta)^m+V$ with the presence of zero resonance or eigenvalue.
	\begin{theorem}\label{Kato-Jensen decay}
 Let $n>2m$ and $H=(-\Delta)^m+V(x)$ with $|V(x)|\lesssim (1+|x|)^{-\beta}$ for some $\beta>0$.  Assume that $H$ has no positive embedded eigenvalue. $P_{ac}(H)$ denotes the projection onto the absolutely continuous spectrum space of $H$.
	
	(I) For $n>4m$.  Let  $\beta>n+4$ and $s, s'>\frac{n}{2}+2$. If zero is an eigenvalue of $H$,  then
		\begin{equation*}\label{kato-4m-eigenvalue}
		\big\|e^{-itH}P_{ac}(H)\big\|_{B(s,-s')}\lesssim (1+|t|)^{2-\frac{n}{2m}},\ \ \ t\in\mathbf{R}.
		\end{equation*}

	(II) For $n\le 4m$ and $n$ is odd. Let $\beta>n+4k$, $s, s'>\frac{n}{2}+2k$,  where $2k=4m-n+1$ and $1\le k \le [\frac{m}{2}]+1$. Then
	
	\begin{itemize}	
		%\item If $1\le k\le m$, and zero is a regular point of $H$, then
%		\begin{equation*}\label{kato-2m-odd-regular}
%		\|e^{-itH}P_{ac}(H)\|_{B(s,-s')}\lesssim (1+|t|)^{-\frac{n}{2m}},\ \ \ t\in\mathbf{R}.
%		\end{equation*}
			
		\item If  zero is   the $j$-th kind of resonance of $H$ with $1\le j\le k$, then
		\begin{equation*}\label{kato-2m-odd-j}
		\big\|e^{-itH}P_{ac}(H)\big\|_{B(s, -s')}\lesssim (1+|t|)^{-\frac{4m+2-n-2j}{2m}},\ \ \ t\in\mathbf{R}.
		\end{equation*}
		
		\item If  zero is an eigenvalue of $H$, then
		\begin{equation*}\label{kato-2m-odd-k+1}
		\big\|e^{-itH}P_{ac}(H)\big\|_{B(s,-s')}\lesssim (1+|t|)^{-\frac{1}{2m}},\ \ \ t\in\mathbf{R}.
		\end{equation*}
	\end{itemize}
	
	(III) For $n\le 4m$ and  $n$ is even.   Let $\beta>n+4k+2$, $s, s'>\frac{n}{2}+2k+1$, where $2k=4m-n+2$ and $1\le k \le [\frac{m}{2}]+1$. Then
	\begin{itemize}
		%\item If $1\le k\le m$, and  zero is a regular point of $H$,  then
%		\begin{equation*}\label{kato-2m-even-regular}
%		\|e^{-itH}P_{ac}(H)\|_{B(s,-s')}\lesssim (1+|t|)^{-\frac{n}{2m}},\ \ \ t\in\mathbf{R}.
%		\end{equation*}
		
		\item If zero is  the $j$-th kind of resonance of $H$ and  $1\le j\le k-1$, then
		\begin{equation*}\label{kato-2m-even-j}
		\big\|e^{-itH}P_{ac}(H)\big\|_{B(s,-s')}\lesssim (1+|t|)^{-\frac{4m+2-n-2j}{m}},\ \ \ t\in\mathbf{R}.
		\end{equation*}

		\item If zero is  the $k$-th kind of resonance or an eigenvalue of $H$, then
		\begin{equation*}\label{kato-2m-even-k}
		\big\|e^{-itH}P_{ac}(H)\big\|_{B(s,-s')}\lesssim \big(1+\ln(1+|t|)\big)^{-1},\ \ \  t\in\mathbf{R}.
		\end{equation*}
	\end{itemize}
\end{theorem}
\vskip0.3cm
\begin{remark}
In the following Table 1, we list the known results of Kato-Jensen decay estimates for $H=(-\Delta)^m+V$. Moreover, if zero is a regular point of $H$, for any order $2\leq m\in\mathbf{N}$ and all dimension $n\geq1$, one can obtain the Kato-Jensen decay estimate for $e^{-itH}P_{ac}(H)$ by the similar process as in this paper. In the non-regular cases, i.e. zero is a resonance or an eigenvalue $H$, for $m\geq2$ and $n$ in the remainder range, the Kato-Jensen decay estimates are still open at present.
\end{remark}

\begin{table}[h]
\renewcommand\arraystretch{2}
\centering
%\centerline{\bf Known results of Kato-Jensen decay estimates}
\begin{tabular}{|p{2cm}<{\centering}|p{3cm}<{\centering}|p{7cm}<{\centering}|}
  \specialrule{0.03em}{0pt}{0pt}
  {\bf Order}      & {\bf Dimension}     & {\bf Results} \\
  \specialrule{0.03em}{0pt}{0pt}
  $m=1$      & $n\geq1$      & \makecell{ $n=1, 2$: \cite{JN}; $n=3$: \cite{JK};  \\ $n=4$: \cite{J1}; $n\geq5$: \cite{J}} \\
  \specialrule{0.03em}{0pt}{0pt}
  $m=2$      & $n\geq3$      & \makecell{$n=3, 4$: \cite{Erdogan-Green-Toprak, Green-Toprak};\\ $n\geq5$: \cite{FSY, FWY}} \\
  \specialrule{0.03em}{0pt}{0pt}
  $m$ odd    & $n\geq3m$   & \makecell{Regular case: Theorem \ref{Kato-Jensen decay regular}; \\ Non-regular cases: Theorem \ref{Kato-Jensen decay} } \\
  \specialrule{0.03em}{0pt}{0pt}
  $m$ even   & $n\geq3m-1$     & \makecell{Regular case: Theorem \ref{Kato-Jensen decay regular}; \\ Non-regular cases: Theorem \ref{Kato-Jensen decay} } \\
  \specialrule{0.03em}{0pt}{0pt}
\end{tabular}
\vspace{0.2cm}
\caption{Known results of Kato-Jensen decay estimates}
\end{table}

	%\begin{remark}
%	The decay rate $\beta$ of potential here we choose is to ensure the conclusions hold if there exists the last kind of resonance. For other kind of resonances, one can choose the sharp $\beta$ smaller than the last kind resonance we choose. Indeed, we only need to expand the free resolvent to the needed terms.
%   \end{remark}	

Next, we state the local decay estimates which is related to Kato's $\mathcal{H}$-smooth perturbation theory.  Recall that for a general selfadjoint operator $\mathcal{H}$, a closed operator $A$ is {\it $\mathcal{H}$-smooth }if and only if
	\begin{equation*}
	\sup_{z\notin\mathbf{R}; \phi\in\mathcal{D}(A^*), \|\phi\|=1}\big|(A^*\phi, {\rm Im}[R_{\mathcal{H}}(z)]A^*\phi)\big|<\infty,
	\end{equation*}
where $R_{\mathcal{H}}(z)=(\mathcal{H}-z)^{-1}$. There are many equivalent characterizations of $A$ is $\mathcal{H}$-smooth in \cite{RS2}. Moreover, if the above uniformly bound holds without taking the imaginary part, then $A$ is called to be {\it $\mathcal{H}$-supersmooth}.  Kato's $\mathcal{H}$-smooth perturbation theory is not only used in the spectral analysis, but also widely applied in the nonlinear dispersive problems. See e.g. \cite{RS2, DAnFane, DAncona, Simon-Review-1, Simon-Review-2} and references therein. 	As a direct consequence of the resolvent estimates, we will  prove that {\it $(1+|x|)^{-\sigma}P_{ac}(H)$ is $H$-supersmooth}.  Then it follows immediately that the following local decay estimate of $H=(-\Delta)^m+V$ holds ,  which can  be applied to establish the endpoint Strichartz estimates (see Section \ref{endpoint-strichartz} below).
\begin{theorem}\label{theo-local decay}
	For $n>2m$ and  $H=(-\Delta)^m+V$ with $|V(x)|\lesssim (1+|x|)^{-\beta}$ for some $\beta>n$. Assume zero is a regular point of $H$ and $H$ has no positive embedded eigenvalue. 	Then for $\phi\in L^{2}(\mathbf{R}^{n})$ and $\sigma>n/2$, we have
	\begin{equation}\label{local decay}
	\Big\| \langle x\rangle^{-\sigma}e^{-itH}P_{ac}(H)\phi\Big\|_{L^{2}_{t}L^{2}_{x}(\mathbf{R}^{n+1})}\lesssim\|\phi\|_{L^{2}(\mathbf{R}^{n})}.
	\end{equation}
where $\langle x\rangle=(1+|x|^2)^{1/2}$.
\end{theorem}

\begin{remark}
	We remark that the above estimate \eqref{local decay} also holds even if zero is an eigenvalue of $H$. Indeed,  $R_{V}(z)$ is the Laplace transform of  $e^{-itH}$, i.e.
	$$\Big|\langle x\rangle^{-\sigma}P_{ac}(H)R_{V}(z)\langle x\rangle^{-\sigma}\Big|=\Big|\int_{0}^{\infty}\langle x\rangle^{-\sigma}P_{ac}(H)e^{-itH}\langle x\rangle^{-\sigma}e^{itz}dt\Big|, \ \ {\rm Im} z>0.$$
The integral is uniformly bounded provided $n>6m$  by (I) of Theorem \ref{Kato-Jensen decay},  which implies estimate \eqref{local decay}.
\end{remark}

\begin{remark}
	Note that, we obtain the Kato-Jensen decay estimates and local decay estimates under the assumption that $H=(-\Delta)^m+V$ does not exists positive embedded eigenvalue. In fact, if $H$ exists one positive eigenvalue $\lambda_{0}>0$, one can also obtain the Kato-Jensen decay estimates and local decay estimates for $\bar{H}$ by the positive commutator methods and Mourre's theory. Where $\bar{H}=\bar{P}H\bar{P}$ and $\bar{P}=1-P_{\lambda_{0}}$. $P_{\lambda_{0}}$ is the orthogonal projection onto eigen-subspace related to $\lambda_{0}$. In \cite{LS}, Larenas and Soffer established the Kato-Jensen type decay estimates for general Hamiltonian $\mathcal{H}$ using the commutator methods. In \cite{GLS}, they applied the approach in \cite{LS} to Schr\"odinger operator. They start from the local decay estimate while the Mourre estimates implies this estimate, please see \cite{Mourre, ABG}.  In order to ensure Mourre's theory can be applied to $\bar{H}=\bar{P}H\bar{P}$, one need $\bar{H}\in C^{2}(\mathcal{A})$ where conjugator $\mathcal{A}=-\frac{i}{2}(x\cdot\nabla+\nabla\cdot x)$, see \cite{ABG, LS}. Note that with the projection $\bar{P}$, operator $\bar{H}=\bar{P}H\bar{P}$ only has continuous spectrum in the interval $(\lambda_{0}-\epsilon, \, \lambda_{0})\cup (\lambda_{0},\, \lambda_{0}+\epsilon)$. Hence, the Mourre estimate can hold for $\bar{H}$ in the neighborhood of $\lambda_{0}$. Moreover, the Mourre estimate implies the local decay estimate and the limiting absorption principle hold for $\bar{H}$ in the neighborhood of $\lambda_{0}$. Please see \cite{Mourre, MoWe, LS} and Amrein, Boutet~de Monvel and Georgescu's book \cite{ABG}.  Furthermore, we will discuss that $H=(-\Delta)^m+V$ does not exist positive embedded eigenvalue for certain class of potential in Section \ref{eigenvalue-problem} below.
\end{remark}

Note that the following  $L^{1}-L^{\infty}$ estimates of  free propagator $e^{-it(-\Delta)^m}$ always hold:
  \begin{equation*}
			\big\|e^{-it(-\Delta)^{m}}u\big\|_{L^{\infty}(\mathbf{R}^{n})}\lesssim |t|^{-\frac{n}{2m}}\|u\|_{L^{1}(\mathbf{R}^{n})}, \ \ t\neq0.
		\end{equation*}
Hence it would be a natural problem to establish  the $L^{1}-L^{\infty}$ estimates for higher order Schr\"odinger operators with some potentials. For $H=(-\Delta)^2+V$, in \cite{FSY}, the first two authors and the last author applied the Kato-Jensen decay estimates and local decay estimate to  obtain the $L^{1}(\mathbf{R}^3)- L^{\infty}(\mathbf{R}^3)$ estimate, which is not optimal. Green and Toprak  in \cite{Green-Toprak} further studied the $L^{1}- L^{\infty}$  estimate for $e^{-itH}$ in 4-dimension in the regular case and the cases of zero resonance. Recently, Erdo\u{g}an, Green and Toprak in \cite{Erdogan-Green-Toprak} give the 3-dimensional results.  For Schr\"odinger operators $-\Delta+V$, Journ\'e, Soffer and Sogge in \cite{JSS} first established the $L^{1}-L^{\infty}$ estimate in the regular case when $n\geq3$. For $n\le 3$, one can see Schlag and Goldberg \cite{Goldberg-Schlag}, Schlag \cite{Schlag-CMP}, Rodnianski and Schlag \cite{RodSchl} and so on.  Yajima \cite{Yajima-JMSJ-95} also proved the $L^{1}-L^{\infty}$  estimates for $-\Delta+V$ by wave operator method.
%As an application of the lower energy resolvent asymptotic expansions of $-\Delta+V$ in the weighted Lebesgue space $L^{2}_{s}(\mathbf{R}^{d})$, The key point of the spectral analysis methods is to  bridge the weighted space $L^{2}_{s}(\mathbf{R}^{d})$ to the space $L^{1}(\mathbf{R}^d)$. In detail, how to obtain the properties as $\lambda$ goes to zero of spectral density $dE(\lambda)$ in the $B(L^{1},  L^{\infty})$-topology through the weighted $B(L^{2}_{s},  L^{2}_{-s})$-topology.
For many further studies, one can refer to \cite{ES1, ES2, Schlag, GV,  Goldberg-Green-1, Goldberg-Green-2} and references therein.

In the sequel, as a consequence of the Kato-Jensen decay estimates, we can establish the following $L^1\cap L^2- L^{\infty}+L^2$-decay estimate (Ginibre argument) in the presence of zero resonance or eigenvalue.

	\begin{theorem}\label{Ginibre-argument}
		For $H=(-\Delta)^m+V$ with $|V(x)|\lesssim (1+|x|)^{-\beta}$ for some $\beta>0$. Assume that $H$ has no positive embedded eigenvalue.
		
		(I) For $n>2m$, let $\beta>n$.  If zero is a regular point of $H$, then
		\begin{equation*}
		\big\|e^{-itH}P_{ac}(H)u\big\|_{L^2+L^\infty(\mathbf{R}^n)}\lesssim(1+|t|)^{-\frac{n}{2m}}\|u\|_{L^2\cap L^1(\mathbf{R}^n)},\ \ \ t\in\mathbf{R}.
		\end{equation*}
		
		(II) For $n>4m$, let $\beta>n+4$.  If zero is an eigenvalue of $H$, then
			\begin{equation*}
			\big\|e^{-itH}P_{ac}(H)u\big\|_{L^2+L^\infty(\mathbf{R}^n)}\lesssim(1+|t|)^{2-\frac{n}{2m}}\|u\|_{L^2\cap L^1(\mathbf{R}^n)},\ \ \ t\in\mathbf{R}.
			\end{equation*}
		
		(III) For $2m<n\leq 4m$ and $n$ is odd. Let $\beta>n+4k$, where $2k=4m-n+1$ and $1\le k\le [\frac{m}{2}]+1$.
		\begin{itemize}
					
			\item If zero is  the $j$-th kind  resonance of $H$ with $1\le j \le k-1$, then
			\begin{equation*}
			\big\|e^{-itH}P_{ac}(H)u\big\|_{L^2+L^\infty(\mathbf{R}^n)}\lesssim\big(1+|t|\big)^{-2+\frac{n}{2m}+\frac{j-1}{m}}\|u\|_{L^2\cap L^1(\mathbf{R}^n)},\ \ \ t\in\mathbf{R}.
			\end{equation*}
			
			\item If zero is the $k$-th  kind of resonance or an eigenvalue of $H$, then
			\begin{equation*}
			\big\|e^{-itH}P_{ac}(H)u\big\|_{L^2+L^\infty(\mathbf{R}^n)}\lesssim(1+|t|)^{-\frac{1}{2m}}\|u\|_{L^2\cap L^1(\mathbf{R}^n)},\ \ \ t\in\mathbf{R}.
			\end{equation*}
		\end{itemize}
		
		(IV) For $2m<n\leq 4m$ and $n$ is even. Let $\beta>n+4k+2$, where $2k=4m-n+2$ and $1\le k\le [\frac{m}{2}]+1$.
		\begin{itemize}
			\item If zero is the $j$-th kind of resonance of $H$ with $1\le j \le k-1$, then
			\begin{equation*}
			\big\|e^{-itH}P_{ac}(H)u\big\|_{L^2+L^\infty(\mathbf{R}^n)}\lesssim\big(1+|t|\big)^{-2+\frac{n}{2m}+\frac{j-1}{m}}\|u\|_{L^2\cap L^1(\mathbf{R}^n)},\ \ \ t\in\mathbf{R}.
			\end{equation*}
			
			\item If zero is the $k$-th  kind of resonance or an eigenvalue of $H$,  then
			\begin{equation*}
			\big\|e^{-itH}P_{ac}(H)u\big\|_{L^2+L^\infty(\mathbf{R}^n)}\lesssim(1+\ln |t|)^{-1}\|u\|_{L^2\cap L^1(\mathbf{R}^n)},\ \ \  t\in\mathbf{R}.
			\end{equation*}
	\end{itemize}
	\end{theorem}

Finaly,  we will give some  results on the  absence of positive embedded eigenvalue for $h(D)+V$, where $h\geq0$ is a homogeneous real-valued function of order $\varrho$.
%For any multi-index $\alpha\in \mathbf{N}^n$, write $P^{(\alpha)}(\xi):=\partial_{\xi}^{\alpha}P(\xi)$. Then $P(D)$ is a selfadjoint operator  on $\mathcal{D}\big(P(D)\big)$:
%$$\mathcal{D}\big(P(D)\big)=\Big\{u: P^{(\alpha)}(D)u\in L^{2}(\mathbf{R}^n), \forall 0\leq |\alpha|\leq\tau\Big\}.$$
	
	\begin{theorem}\label{absence of positive eigenvalue}
		Let  $H_{h}=h(D)+V$, where $h$ is a nonnegative real-valued function satisfying $h(\gamma\xi)=\gamma^{\varrho}h(\xi)$ with $\varrho>0$ and $\gamma>0$. $V(x)$ is a real-valued function on $\mathbf{R}^{n}$. Suppose that $V$ is $h(D)$-bounded with relative bound less than one. Then we have the following conclusions:
		\par (i) If $V(x)$ is repulsive, that is $V(\gamma x)\leq V(x)$ for all $\gamma>1$ and $x\in\mathbf{R}^{n}$, then $H_{h}$ has no bound states, i.e. the point spectrum $\sigma_{p}(H_{h})\cap\mathbf{R}=\emptyset$.
		\par (ii) If $V(x)$ is homogeneous of degree $-\nu$ with $0<\nu<\varrho$, that is $V(\gamma x)=\gamma^{-\nu}V(x)$, then $H_{h}$ has no nonnegative eigenvalues, i.e. the point spectrum $\sigma_{p}(H_{h})\cap[0, +\infty)=\emptyset$.
		\par (iii) If there exists a multiplication operator $\mathcal{V}$ on $L^{2}(\mathbf{R}^{n})$ with $\mathcal{D}(\mathcal{V})\supset \mathcal{D}\big(h(D)\big)$, such that for all $\phi\in \mathcal{D}\big(h(D)\big)$,
		\begin{equation}\label{V-limit}
			s-\lim_{\theta\rightarrow 1}(\theta-1)^{-1}\big(V(\theta\cdot)-V\big)\phi(x)=\mathcal{V}\phi(x),
		\end{equation}
		where $\mathcal{D}\big(h(D)\big)$ and $\mathcal{D}(\mathcal{V})$ are the self-adjoint domain of $h(D)$ and $\mathcal{V}$ respectively. Moreover, if  for some $a>0$, $V$ satisfies,
		\begin{equation}\label{V condition 3}
			h(D)-\frac{1}{\varrho}(1+a)\mathcal{V}-aV\geq0,
		\end{equation}
		then $H_{h}$ has no strictly positive eigenvalues, i.e. $\sigma_{p}(H_{h})\cap(0, +\infty)=\emptyset$.
	\end{theorem}
	
	\begin{remark} The conclusion (i) can be applied to the potential $V(x)$ satisfying $x\cdot\nabla V(x)\le 0$.
		The conclusion (iii) can be applied to certain homogeneous potentials $V(x)$ of degree $-\varrho$. Indeed, if $V(x)$ does, then $\mathcal{V}=-\varrho V$. Thus,
		\begin{equation*}
			h(D)-\frac{1}{\varrho}(1+a)\mathcal{V}-aV=h(D)+V\geq0.
		\end{equation*}
		For instance, if consider the higher order polyhamonic operators $(-\Delta)^m-c|x|^{-2m}$ on $L^{2}(\mathbf{R}^{n})$, by Rellich's type inequality (see e.g. \cite{Davies-Hinz-MathZ}): if $1<p<\infty$ and $n>2mp$, then
 \begin{equation*}
 \big\||x|^{-2m}u(x)\big\|_{L^p(\mathbf{R}^n)}\leq c(m, p, n)\big\|\Delta^m\big\|_{L^{p}(\mathbf{R}^n)},
 \end{equation*}
 where $c(m, p, n)=p^{2m}\Big(\prod_{l=1}^{m}(n-2lp)\big(2(l-1)p+(p-1)n\big)\Big)^{-1}$.
 Thus  $-c|x|^{-2m}$  is the $(-\Delta)^m$-bounded with relative bound less than one if $n>4m$ and $0<c<C(m, 2, n)$. Hence Theorem \ref{absence of positive eigenvalue} (iii) can apply to the higher order operator $(-\Delta)^m-c|x|^{-2m}$.
	\end{remark}

	The paper is organized as follows. In Section \ref{free-resolvent}, we state the asymptotic expansions of $R_{V}(z)$ as $z\rightarrow0$ with presence of zero resonance or eigenvalue and the classification of resonance spaces. In Section \ref{high energy}, we show the higher energy decay estimates.   In Section \ref{Kato-Jensen-estimate},  we give the proof of local decay estimate and derive the asymptotic expansions in time for $e^{-itH}P_{ac}(H)$  with presence of zero resonance. As a consequence,  we get the  Kato-Jensen type decay estimate. In Section \ref{Lp} and Section \ref{endpoint-strichartz}, as applications of Kato-Jensen type decay estimate and local decay estimate, we prove the $L^{p}$-decay estimates and the endpoint Strichartz estimates for  $H=(-\Delta)^m+V$. In Section \ref{eigenvalue-problem}, we construct examples to show that $H_{h}=h(D)+V$ exists positive eigenvalue and give the proof of Theorem \ref{absence of positive eigenvalue}. In the last section, we show the iterated processes of how to derive the expansion of $R_{V}(z)$ and the proof of  classification of resonance spaces.

	\section{Resolvent asymptotic expansions at zero threshold}\label{free-resolvent}
	In this section, we derive the asymptotic expansions of  $R_{V}(z)$ as $z\rightarrow0$.
	
	\subsection{Free resolvent asymptotic expansions}
	For $z\in\mathbf{C}\setminus\mathbf{R}^{+}$, denote
	\begin{equation*}
		R_{0}(z)=(H_{0}-z)^{-1}=\big[(-\Delta)^m-z\big]^{-1}, \,\, R_{V}(z)=(H-z)^{-1}=\big[(-\Delta)^m+V-z\big]^{-1}.
	\end{equation*}
	For free resolvent $R_{0}(z)$ with $z\in\mathbf{C}\setminus\mathbf{R}^{+}$ and $0<\arg(z)<2\pi$ , we have the following decomposition identity
	\begin{equation}\label{free-resolvent-identity}
		R_{0}(z)=\big[(-\Delta)^{m}-z\big]^{-1}=\frac{1}{mz}\sum_{\ell=0}^{m-1}z_\ell\big(-\Delta-z_\ell\big)^{-1},\ \ \ z_\ell=z^{\frac{1}{m}}e^{i\frac{2\ell\pi}{m}}.
	\end{equation}
	Identity  \eqref{free-resolvent-identity} shows that the free resolvent $R_{0}(z)$ of $(-\Delta)^m$ is a linear combination of the free resolvent of Laplacian. Thus we can obtain the asymptotic expansions of $R_0(z)$ around zero from the asymptotic expansions of  the free resolvent of Laplacian. Note that $z_{\ell}$ satisfy ${\rm Im}\, z_{\ell}^{1/2}>0$ for $0\leq \ell\leq m-1$ with chosen $0<\arg(z)<2\pi$. Recall  the  asymptotic expansions of  the free resolvent of Laplacian $R(-\Delta; \zeta):=(-\Delta-\zeta)^{-1}$ as follows, see \cite{JK, J, J1}.
	
	\begin{lemma}\label{laplacian}
		For $\zeta\in \mathbf{C}\setminus[0, \infty)$ and ${\rm Im}\, \zeta^{1/2}>0$.
		
		(i) If $n\ge3$ and $n$ is odd, then
		\begin{equation*}\label{lap odd}
			R(-\Delta; \zeta)=\sum_{j=0}^{\infty}(i\zeta^{1/2})^{j}G_{j}^{odd}
		\end{equation*}
		where the operators $G_j^{odd}$ are given by the following integral kernels:
		\begin{equation*}
			G_j^{odd}(x, y)=\frac{(-1)^{(n-3)/2}}{2(2\pi)^{(n-1)/2}}n_j|x-y|^{j+2-n}
		\end{equation*}
		with $\displaystyle n_j=\sum_{\ell=0, \ell\ge(n-3)/2-j}^{(n-3)/2}\frac{\big((n-3)/2+\ell\big)!}{\ell!\big((n-3)/2-\ell\big)!}\frac{(-2)^{-\ell}}{\big(\ell+j-(n-3)/2\big)!}$.
		
		(ii) If $n\ge 4$ and $n$ is even, then
		\begin{equation*}\label{lap even}
			R(-\Delta; \zeta)=\sum_{j=0}^{\infty}\sum_{\ell=0}^{1}\zeta^{j}(\ln \zeta)^{\ell}G_{j}^{\ell, e}
		\end{equation*}
		For $0\le j\le \frac{n}{2}-2,$  the operators $G_j^{\ell, e}$ are given by the following integral kernels:
		\begin{equation*}
			G_j^{0, e}(x, y)=\pi^{-\frac{n}{2}}\frac{(\frac{n}{2}-j-2)!}{j!}4^{-j-1}|x-y|^{2j+2-n};\,\,\, G_j^{1, e}(x, y)=0.
		\end{equation*}
		For $j\ge\frac{n}{2}-1,$  the operators $G_j^{k, e}$ are given by the following integral kernels:
		\begin{equation*}
			\begin{split}
				G_j^{0, e}(x, y) =&(4\pi)^{-n/2}[\varphi(j+1)+\varphi(j+2-n/2)]\frac{(-1/4)^{j+1-n/2}}{j!(n/2-1+j)!}|x-y|^{2j+2-n}\\
				&-2(4\pi)^{-n/2}(-1/4)^{j+1-n/2}\frac{\ln(|x-y|/2)}{j!(n/2-1+j)!}|x-y|^{2j+2-n}\\
				&+\frac{i}{4}(4\pi)^{-n/2+1}\frac{1}{j!(n/2-1+j)!}(-1/4)^{j+1-n/2}|x-y|^{2j+2-n}.
			\end{split}
		\end{equation*}
		\begin{equation*}
			G_j^{1, e}(x, y)=-(4\pi)^{-n/2}\frac{(-1/4)^{j+1-n/2}}{j!(n/2-1+j)!}|x-y|^{2j+2-n}.
		\end{equation*}
		where $\varphi(1)=-1,\ \varphi(\ell)=\sum\limits_{j=1}^{\ell-1}\frac{1}{j}-\kappa$ and $\kappa$ is the Euler's constant.
	\end{lemma}
	\begin{remark}\label{rem2}
		 If $j$ is odd and $0<j<n-2$, then $n_j=0$, see \cite[Lemma 3.3]{J}.  For any $j=0, 1, 2,\cdots$, the operators $G_j^{odd},\ G_j^{\ell, e}\in B(s, -s')$ with $s, s'$ depend on $j$ and $n$.
	\end{remark}
	
	Based on the expansions of $R(-\Delta;\zeta)$ and the resolvent identity \eqref{free-resolvent-identity}, we obtain the formally expansions of $R_0(z)$ directly.
	
	\begin{lemma}\label{lemma-free-expansions}
		For $z\in\mathbf{C}\setminus[0, +\infty)$ and $0<\arg(z)<2\pi$, we obtain the following formally expansions of $R_0(z)$:
		
		(i) If $n\ge3$ and $n$ is odd, then
		\begin{equation}\label{eq-odd-free}
			R_0(z)=\sum_{j=0}^\infty z^{\frac{j+2-2m}{2m}}g_jG_j^{odd}
		\end{equation}
		with $g_j=i^j$ for $j\in2m\mathbf{N}-2$ and $g_j=\frac{1}{m}\frac{i^j-(-i)^j}{1-e^{i\frac{(j+2)\pi}{m}}}$ for $j\notin2m\mathbf{N}-2$.
		
		(ii) If $n\ge4$ and $n$ is even, then
		\begin{equation}\label{eq-even-free}
			R_0(z)=\sum_{j=0}^\infty\sum_{\ell=0}^1(\ln z^{\frac{1}{m}})^\ell z^{\frac{j+1-m}{m}}\widetilde{G}_j^{\ell, e},
		\end{equation}
		where $\widetilde{G}_j^{0, e}=g_j^0C_j^{0, e}+g_j^{1, 0}G_{j}^{1, e}$ and $\widetilde{G}_j^{1, e}=g_j^{1, 1}G_j^{1, e}$. Here,  $$ g_j^0=g_j^{1, 1}=1,\,\, g_j^{1, 0}=i\frac{(m-1)\pi}{m},\,\,\,\, j\in m\mathbf{N}-1;$$ $$g_j^0=g_j^{1, 1}=0,\,\, g_j^{1, 0}=\frac{-2i\pi}{1-e^{i\frac{(j+1)2\pi}{m}}},\,\,\,\, j\notin m\mathbf{N}-1.$$
	\end{lemma}
	
	Notice that there are factors $i^j-(-i)^j$ and $0$ in the expansions of $R_0(z)$,  then so many terms can be cancelled. According to Remark \ref{rem2}, we obtain the above formal expansions in the following sense. For simplifying the notation, we let $z=\mu^{2m}$ with $0<\arg(\mu)<\pi/m$. Denote $R_{0}(\mu^{2m}; x, y)$ be the convolution kernel of $R_{0}(\mu^{2m})$.
	\begin{proposition}\label{free-expansions}
		For $\mu\in\mathbf{C}\setminus[0,+\infty)$ and $0<\arg(\mu)<\pi/m$, we obtain the following expansions for $0<|\mu|\ll1$ in $B(s, -s')$ with $s, s'>\frac{n}{2}+2$:
		
		(i) If $n>4m$ and $n$ is odd, denotes $\aleph=[\frac{n}{2m}]$. Then we have
		\begin{equation}\label{4m-odd-free}
			\begin{split}
				R_0(\mu^{2m}; x, y)=&a_0|x-y|^{2m-n}+\sum_{j=1}^{\aleph-1}a_j\mu^{2mj}|x-y|^{2m(j+1)-n}
				+a_{\aleph}\mu^{n-2m}|x-y|^0\\
				&\ \ +E_1(\mu; x,y)
			\end{split}
		\end{equation}
		where $a_{\aleph}\in \mathbf{C}\setminus\mathbf{R}$ and $a_j\in\mathbf{R}\setminus \{0\}$ for $0\le j\le \aleph-1$.  Furthermore, for $s, s'>\frac{n}{2}+2$,
		\[E_{1}(\mu; x, y)\in B(s, -s')\ \ \ \text{and}\ \ \ \|E_{1}(\mu; x, y)\|_{B(s,-s')}=O(\mu^{n-2m+1}).\]
		
		(ii) If $n>4m$ and $n$ is even, denotes $\aleph=[\frac{n}{2m}]$. Then we have
		\begin{equation}\label{4m-even-free}
			\begin{split}
				R_0(\mu^{2m}; x, y)=&b_0|x-y|^{2m-n}+\sum_{j=1}^{\aleph-1}b_j\mu^{2mj}|x-y|^{2m(j+1)-n}
				+b_{\aleph}\mu^{n-2m}|x-y|^0\\
				&\ \ +E_2(\mu; x, y)
			\end{split}
		\end{equation}
		where $b_\aleph\in \mathbf{C}\setminus\mathbf{R}$ and $b_j\in\mathbf{R}\setminus \{0\}$ for $0\le j\le \aleph-1$.  Furthermore, for $s, s'>\frac{n}{2}+2$,
		\[E_2(\mu; x, y)\in B(s, -s')\ \ \ \text{and}\ \ \ \|E_2(\mu; x, y)\|_{B(s,-s')}=O(\mu^{n-2m}+2).\]
		
		(iii) If $2m<n\le 4m$ and $n$ is odd, let $n=4m+1-2k$ with $1\le k\le m$.  Then we have
		\begin{equation}\label{2m-odd-free}
			\begin{split}
				R_0(\mu^{2m}; x, y)=&c_0|x-y|^{2m-n}+\sum_{j=1}^{k}c_j\mu^{2(m+j-k)-1}|x-y|^{2j-2}\\
				&\ \ +c_{k+1}\mu^{2m}|x-y|^{2k-1}+E_3(\mu; x,   y)
			\end{split}
		\end{equation}
		where $c_j\in\mathbf{C}\setminus \mathbf{R}$ for $1\le j\le k$ and $c_0,\ c_{k+1}\in\mathbf{R}\setminus \{0\}$.  Furthermore, for $s, s'>\frac{n}{2}+2k$,
		\[E_3(\mu; x, y)\in B(s,-s')\ \ \ \text{and}\ \ \ \|E_3(\mu; x, y)\|_{B(s, -s')}=O(\mu^{2m+1}).\]
		
		(iv) If $2m<n\le 4m$ and $n$ is even, let $n=4m+2-2k$ with $1\le k\le m$.  Then we have
		\begin{equation}\label{2m-even-free}
			\begin{split}
				R_0(\mu^{2m}; x, y)=&d_0|x-y|^{2m-n}+\sum_{j=1}^{k-1}d_j\mu^{2(m+j-k)}|x-y|^{2j-2}
				+\mu^{2m}g(\mu)|x-y|^{2k-2}\\
				&+d_{k+1}\mu^{2m}|x-y|^{2k-2}\ln(|x-y|)+ E_4(\mu; x, y)
			\end{split}
		\end{equation}
		where $g(\mu)=d_k\ln(\mu)+c_k$, $d_j\in\mathbf{C}\setminus \mathbf{R}$ for $1\le j\le k-1$ and $d_0, d_k, d_{k+1}\in\mathbf{R}\setminus\{0\}$. $c_{k}\in\mathbf{C}\setminus\mathbf{R}$ that can be calculated by Lemma \ref{lemma-free-expansions} and identity \eqref{free-resolvent-identity}.  Furthermore, for $s, s'>\frac{n}{2}+2k$,
		\[E_4(\mu; x, y)\in B(s, -s')\ \ \ \text{and}\ \ \ \|E_4(\mu; x, y)\|_{B(s,-s')}=O(\mu^{2m+2}).\]
	\end{proposition}
	\begin{proof}
		The proof follows from \cite[Lemma 2.3]{J} and \cite[Lemma 3.5, Lemma 3.9]{J1}, since we derive the expansions  of $R_0(\mu^{2m})$ by the expansions of $R(-\Delta;\zeta)$.
	\end{proof}

	\subsection{Asymptotic expansions of $R_{V}(z)$ around zero }
	
	For $R_{V}(\mu^{2m})$, we apply the following symmetric resolvent identity to derive the asymptotic expansions as $\mu \rightarrow 0$.
	\begin{equation}\label{symmetric-resolvent-idnetity}
		R_{V}(\mu^{2m})=R_{0}(\mu^{2m})-R_{0}(\mu^{2m})vM(\mu)^{-1}vR_{0}(\mu^{2m}),
	\end{equation}
	where $v(x)=|V(x)|^{1/2}$,
	\begin{equation*}
		M(\mu)=U+vR_{0}(\mu^{2m})v
	\end{equation*}
	and $U(x)=\rm{sign}\big (V(x)\big)$.  Furthermore, let $w(x)=U(x)v(x)$, then
	\begin{equation}\label{weighted-resolvent-identity}
		wR_{V}(\mu^{2m})w=U-M(\mu)^{-1}.
	\end{equation}
Note that, for $0<|\mu|\ll1$, the identity \eqref{weighted-resolvent-identity} shows $R_{V}(\mu^{2m})$ in $B(s, -s)$ with suitable $s>0$ has the same asymptotic behaviors of $M(\mu)^{-1}$ in $B(0, 0)$ if $w(x)(1+|x|)^{s}\in L^{\infty}(\mathbf{R}^{n})$.
	
	Since the number of  expansion terms depends on dimension $n$ for a fixed order $m$, we divide the dimension into three intervals to derive the asymptotic expansion of $R_{V}(\mu^{2m})$.

	\textbf{Case 1: $n>4m$.}
	By Proposition \ref{free-expansions} (i) and (ii), we have the following expansions of $R_0(\mu^{2m}; x, y)$ in $B(s, -s')$ with $s, s'>\frac{n}{2}+2$:
	\begin{equation}\label{4m-free}
		\begin{split}
			R_{0}(\mu^{2m}; x, y)=&G_{0}(x, y)+\sum_{j=1}^{\aleph-1}\mu^{2mj}G_j(x, y)+\mu^{n-2m}G_\aleph(x, y)\\
			&\ \ +\mu^{n-2m+2}G_{\aleph+1}(x,y)+E_0(\mu; x, y)
		\end{split}
	\end{equation}
	where
	\begin{align*}
		&G_0(x, y)=\begin{cases}
			a_0|x-y|^{2m-n},\ \ &\ \ n>4m\ \ \text{odd};\\
			b_0|x-y|^{2m-n},\ \ &\ \ n>4m\ \ \text{even};
		\end{cases}\\
		&G_j(x, y)=\begin{cases}
			a_j|x-y|^{2m(j+1)-n},\ \ &\ \ n>4m\ \ \text{odd};\\
			b_j|x-y|^{2m(j+1)-n},\ \ &\ \ n>4m\ \ \text{even};
		\end{cases}\ \ \text{for}\ \ 1\le j\le \aleph-1,\\
		&G_\aleph(x, y)=\begin{cases}
			a_\aleph|x-y|^0,\ \ & \ \ n>4m\ \ \text{odd};\\
			b_\aleph|x-y|^0,\ \ & \ \ n>4m\ \ \text{even};\\
		\end{cases}\\
		&G_{\aleph+1}(x, y)=\begin{cases}
			a_{\aleph+1}|x-y|^2,\ \ &\ \ n>4m\ \ \text{odd};\\
			b_{\aleph+1}|x-y|^2,\ \ &\ \ n>4m\ \ \text{even};
		\end{cases}\\
		&E_0(\mu; x, y)=\begin{cases}
			E_1(\mu; x, y),\ \ &\ \ n>4m\ \ \text{odd};\\
			E_2(\mu; x, y),\ \ &\ \ n>4m\ \ \text{even}.
		\end{cases}
	\end{align*}
	Substituting \eqref{4m-free} into the symmetric resolvent identity \eqref{symmetric-resolvent-idnetity}, we have
	\begin{equation}\label{M-4m}
		M(\mu)=U+vG_0v+\sum_{j=1}^{\aleph-1}\mu^{2mj}vG_jv+\mu^{n-2m}vG_\aleph v+vE_0(\mu)v.
	\end{equation}
	Recall that $a_{\aleph}\in\mathbf{C}\setminus\mathbf{R}$. Note that we need to expand to the order $\mu^{n-2m}$ here. The reason is that $G_{j} (0\leq j\leq \aleph-1)$ are selfadjoint operators which has no contribution to the spectral density.
	
	Denotes $T_0=U+vG_0v$.  If $T_0$ is invertible on $L^2(\mathbf{R}^n)$, then we say zero is {\it a regular point of $H=(-\Delta)^m+V$.} Otherwise, if $T_{0}$ is not invertible,  let $S_1$ to be the Riesz projection onto $\ker(T_{0})$ as an operator on $L^2(\mathbf{R}^n)$. Then  $T_0+S_1$ is invertible on $L^2(\mathbf{R}^n)$. Accordingly, let
	$D_0=(T_0+S_1)^{-1}$.  If $T_{0}$ is not invertible and $T_{1}=S_{1}vG_{1}vS_{1}$ is invertible, then we say zero is {\it a ``resonance" of $H$}.  In the case $n>4m$, the subspace $S_{1}L^{2}(\mathbf{R}^n)$ is actually the zero eigenspace of $H$.  We will show $T_{1}$ is invertible on $L^2(\mathbf{R}^n)$, see  Theorem \ref{RV-4m} and Proposition \ref{classification-4m}  below.  Hence, $H=(-\Delta)^m+V$ does not exist zero resonance if $n>4m$.
	
	Next, we give the expansions of $R_{V}(\mu^{2m})$ in $B(s, -s')$ with $s, s'>\frac{n}{2}+2$.
	\begin{theorem}\label{RV-4m}
		For $n>4m$, let $|V(x)|\lesssim (1+|x|)^{-\beta}$ with some $\beta>n+4$. Then for $0<|\mu|\ll1$, we have the following expansions of $R_V(\mu^{2m})$ in $B(s, -s')$:
		
		(i) If zero is a regular point of $H$, then
		\begin{equation*}\label{Rv-regular-4m}
			R_V(\mu^{2m})=A_0+\sum_{\ell=1}^{\aleph-1}\mu^{2m\ell}A_\ell+\alpha_\aleph\mu^{n-2m} A_{\aleph}+O(\mu^{n-2m+2})
		\end{equation*}
		where  $A_\ell\in B(s, -s')$ are selfadjoint operators and $\alpha_\aleph\in\mathbf{C}\setminus\mathbf{R}$. Furthermore,  $A_0=G_0-G_0v(T_0)^{-1}vG_0$.
		
		(ii) If zero is an eigenvalue of $H$, then
		\begin{equation*}\label{Rv-eigenvalue-4m}
			R_V(\mu^{2m})=\frac{A_1^e}{\mu^{2m}}+\sum_{\ell=2}^{\aleph-1}\mu^{2m(\ell-2)}A_\ell^e+\alpha_{\aleph}^e\mu^{n-6m}A_\aleph^e+O(\mu^{n-6m+2})
		\end{equation*}
		where $A_\ell^e\in B(s, -s')$ are selfadjoint operators and $\alpha_\aleph^e\in\mathbf{C}\setminus\mathbf{R}$. Furthermore,  $A_1^e=-G_0vS_{1}(S_1vG_1vS_1)^{-1}S_{1}vG_0$ is finite rank.
	\end{theorem}

	\textbf{Case 2: $n=4m+1-2k$ with $k=1, 2, \cdots, m$. }(i.e. $2m<n\le 4m$ and $n$ odd).
	In this case, by Proposition \ref{free-expansions} (iii),  we have the following expansions of $R_0(\mu^{2m}; x, y)$ in $B(s, -s')$ with $s, s'>n/2+2k$:
	\begin{equation}\label{2m-odd-free-2}
		\begin{split}
			R_{0}(\mu^{2m}; x, y)=&G_{0}(x, y)+c_{1} \mu^{2(m-k)+1} I+\sum_{j=2}^kc_j \mu^{2(m+j-k)-1}G_j(x, y)\\
			&\ \ +\mu^{2m}G_{k+1}(x, y)+E_3(\mu; x, y)
		\end{split}
	\end{equation}
	where
	\begin{align*}
		&G_{0}(x, y)=c_{0}|x-y|^{2m-n};\\
		&G_{j}(x, y)=|x-y|^{2j-2}, \,\,\,\, 2\le j\le k;\\
		&G_{k+1}(x, y)=c_{k+1}|x-y|^{2k-1}.
	\end{align*}
	
	Let $P$ be the projection onto the span of $v$, i.e. $P=\|v\|_{L^{2}(\mathbf{R}^{n})}^{-1}v\langle v, \cdot\rangle$. Thus $P$ is a self-adjoint operator with  $\dim(P)=1$. Substituting \eqref{2m-odd-free-2} into  \eqref{symmetric-resolvent-idnetity}, we have
	\begin{equation}\label{M-2m-odd}
		M(\mu)=U+vG_0v+\tilde{c}_1\mu^{2(m-k)+1}P+\sum_{j=2}^kc_j\mu^{2(m+j-k)-1}vG_jv+\mu^{2m}vG_{k+1}v+vE_3(\mu)v.
	\end{equation}
	Depending on the inverse processes,  now we give  the equivalent definition of each kinds of zero resonance of $H$.
	\begin{definition} \label{resonance}
		Let $T_{0}=U+vG_{0}v$.
		\begin{enumerate}
			\item If $T_{0}$ is invertible on $L^{2}(\mathbf{R}^n)$, then we call zero is a regular point of $H$.
			\item Assume that $T_0$ is not invertible on $L^2(\mathbf{R}^n)$. Let $S_1$ be the Riesz projection onto  $\ker (T_0)$, then $T_0+S_1$ is invertible on $L^2(\mathbf{R}^n)$.  If $T_0$ is not invertible and $T_1:=S_1PS_1$ is invertible on $S_1L^2(\mathbf{R}^n)$, then we call zero is the first kind of resonance of $H$.
			\item Assume that $T_1$ is not invertible on $S_1L^2(\mathbf{R}^n)$.  Let $S_2$ be the Riesz projection onto  $\ker(T_1)$, then $T_1+S_2$ is invertible on $S_1L^2(\mathbf{R}^n)$.  If $T_1$ is not invertible and $T_2:=S_2vG_2vS_2$ is invertible on $S_2L^2(\mathbf{R}^n)$, then we call zero is the second kind of resonance of $H$.
			\item For $3\le j\le k$ and $1\le k\le [\frac{m}{2}]+1$. Assume that $T_{j-1}:=S_{j-1}vG_{j-1}vS_{j-1}$ is not invertible on $S_{j-1}L^2(\mathbf{R}^n)$.  Let $S_{j}$ be the Riesz projection onto $\ker(T_{j-1})$, then $T_{j-1}+S_{j}$ is invertible on $S_{j-1}L^{2}(\mathbf{R}^n)$. If $T_{j-1}$ is not invertible and $T_j:=S_jvG_jvS_j$ is invertible on $S_jL^2(\mathbf{R}^n)$, then we call zero is the $j$-th kind of resonance of $H$.
			\item Assume that $T_k$ is not invertible on $S_kL^2(\mathbf{R}^n)$. Let $S_{k+1}$ be the Riesz projection onto $\ker(T_k)$, then $T_k+S_{k+1}$ is invertible on $S_kL^2(\mathbf{R}^n)$.  In this case, the operator $T_{k+1}:=S_{k+1}vG_{k+1}vS_{k+1}$ is always invertible, then we say there is a $(k+1)$-th kind of ``resonance" at zero.
		\end{enumerate}	
	\end{definition}
	
	\begin{remark}\label{properties-of-S-T}
		(1) Definition \ref{resonance} is equivalent to the Definition \ref{resdef} actually. Indeed, we will identify all the subspaces $S_{j}L^2(\mathbf{R}^n)$, see Proposition \ref{classification-2m-odd} and  Proposition \ref{classification-2m-even}.  From the proof of the identification processes, these two definitions are equivalent, see subsection \ref{identfication} below. Furthermore, the last kind of zero ``resonance" is actually zero eigenvalue.
		
	   \par (2) The projections $S_{j}$  $(1\le j\le k+1)$ are finite rank operators. Indeed, $T_{0}=U+vG_{0}v$ which is a compact perturbation of the invertible operator $U$, thus  the Fredholm alternative theorem guarantees that $S_{1}$ is of finite rank. By  Definition \ref{resonance}, we have $S_{k+1}\leq S_k\le\cdots \le S_2\le S_1$, hence all $S_{j}$ ($1\le j \le k+1$)  are finite rank.
	   \par (3) By  Definition \ref{resonance}, let $D_j=(T_j+S_{j+1})^{-1}$ for $0\le j\le k$, then we have
		\begin{align*}
			&S_{j+1}D_j=D_jS_{j+1} =S_{j+1},\ \ \ 0\le j\le k;\\
			&S_jD_j=D_jS_j=D_j,\ \ \ 1\le j\le k.
		\end{align*}
		\par (4) For Schr\"odinger operator $-\Delta+V$ in 3 and 4 dimensional cases, i.e. $m=1$ with $n=3, 4$, we have  $k=1$.  Thus $-\Delta+V$ has only one kind resonance for $n=3, 4$. Recall that $-\Delta+V$ does not exist zero resonance for $n\geq5$, see \cite{J}. Thus  Definition \ref{resonance} matches the resonance definition of Kato and Jensen in \cite{JK, J}.
	\end{remark}
	
	Now, we give the asymptotic expansions of $R_{V}(\mu^{2m})$ in $B(s, -s')$ with suitable chosen $s, s'\in\mathbf{R}$ with the presence of each kind of zero rensonance. Note that, the following expansions imply that how the zero rensonance affect the behavior of spectral density $dE(\lambda)$ of $H=(-\Delta)^m+V$ as $\lambda$ close to 0.
	\begin{theorem}\label{RV-expansions-2m-odd}
	For $n=4m+1-2k$ with $k\in\mathbf{N}$ chosen as follows, let $|V(x)|\lesssim (1+|x|)^{-\beta}$ with some $\beta>n+4k$.  Then for  $0<|\mu|\ll1$, we have the following expansions of $R_V(\mu^{2m})$ in $B(s, -s')$ with $s, s'>\frac{n}{2}+2k$ :
		
		(i) For $1\le k\le m$, if zero is a regular point of $H$, then we have
		\begin{equation*}\label{RV-2m-odd-regular-case}
			R_{V}(\mu^{2m})=B_0+\sum_{\ell=1}^k\beta_\ell\mu^{2(m+\ell-k)-1}B_\ell+\mu^{2m}B_{k+1}+O(\mu^{2m+})
		\end{equation*}
		where $B_{\ell}\in B(s, -s')$ are selfadjoint operators and $\beta_\ell\in\mathbf{C}\setminus\mathbf{R}$. Furthermore, $B_0=G_0-G_0v(T_0)^{-1}vG_0$.
		
		(ii) For $1\le k\le[\frac{m}{2}]+1$, if zero is  the $j$-th kind of resonance  of $H$  with $1\le j\le k$, then we have
		\begin{equation*}\label{RV-2m-odd-j-res}
			\begin{split}
				R_{V}(\mu^{2m})=&\frac{\beta_0^j B_0^j}{\mu^{2(m-k+j)-1}}+\sum_{\ell=1}^{k-j}\frac{\beta_\ell^jB_\ell^j}{\mu^{2(m-k+j-\ell)-1}}
				+\sum_{\ell=k-j+1}^{2m-3k+3j-2}\frac{\beta_\ell^jB_{\ell}^j}{\mu^{2m-3k+3j-\ell-1}}\\
				& +\beta_{2m-3k+3j-1}^jB_{2m-3k+3j-1}^j+O(\mu),
			\end{split}
		\end{equation*}
		where  $B_\ell^j\in B(s, -s')$ are selfadjoint operators and $\beta_\ell^j\in\mathbf{C}\setminus\mathbf{R}$. Furthermore, $B_0^j=-G_0vS_{j}(T_j)^{-1}S_{j}vG_0$ and all $B_{\ell}^{j}~ (0\leq\ell\leq 2m-3k+3j-2)$ are finite rank.
		
		(iii) For $1\le k\le[\frac{m}{2}]+1$, if zero is an eigenvalue of $H$, then we have
		\begin{equation*}\label{RV-2m-odd-k+1-res}
			\begin{split}
				R_{V}(\mu^{2m})=&\frac{B_0^{k+1}}{\mu^{2m}}+\sum_{\ell=1}^{2m-1}\frac{\beta_{\ell}^{k+1}B_\ell^{k+1}}{\mu^{2m-\ell}}+\beta_{2m}^{k+1}B_{2m}^{k+1}+O(\mu),
			\end{split}
		\end{equation*}
		where  $B_\ell^{k+1}\in B(s,-s')$ are selfadjoint operators and $\beta_\ell^{k+1}\in\mathbf{C}\setminus\mathbf{R}$. Furthermore, $B_{0}^{k+1}=-G_0 vS_{k+1}(T_{k+1})^{-1}S_{k+1}v G_0$ and all $B_{\ell}^{k+1}~ (0\leq\ell\leq 2m-1)$ are finite rank.
	\end{theorem}

	\textbf{Case 3: $n=4m+2-2k$ with $k=1, 2, \cdots, m$. } (i.e. $2m<n\le 4m$ and $n$ even).
	In this case, by Proposition \ref{free-expansions} (iv),  we have the following expansions of $R_0(\mu^{2m}; x, y)$ in $B(s, -s')$ with $s, s'>\frac{n}{2}+2k+1$:
	\begin{equation}\label{2m-even-free-2}
		\begin{split}
			R_{0}(\mu^{2m}; x, y)=&G_{0}(x, y)+d_1\mu^{2(m-k+1)}I+\sum_{j=2}^{k-1}d_j\mu^{2(m+j-k)}G_j(x, y)\\
			&+\mu^{2m}g(\mu)G_k(x, y)+\mu^{2m}G_{k+1}(x, y)+E_4(\mu; x, y)
		\end{split}
	\end{equation}
	where
	\begin{align*}
		&G_{0}(x, y)=d_{0}|x-y|^{2m-n};\\
		&G_{j}(x, y)=|x-y|^{2j-2}, \,\,\,\, 2\le j\le k;\\
		&G_{k+1}(x, y)=d_{k+1}|x-y|^{2k-2}\ln(|x-y|).
	\end{align*}
	Substituting \eqref{2m-even-free-2} into  identity \eqref{symmetric-resolvent-idnetity}, we have
	\begin{equation}\label{M-2m-even}
		\begin{split}
			M(\mu)=&U+vG_0v+\tilde{d}_1\mu^{2(m-k+1)}P+\sum_{j=2}^{k-1}d_j\mu^{2(m+j-k)}vG_jv\\
			&+\mu^{2m}g(\mu)vG_kv
			+\mu^{2m}vG_{k+1}v+vE_4(\mu)v.
		\end{split}
	\end{equation}	
	
	The definition of resonance for this case is the same as Definition \ref{resonance} by replacing the representations of $G_{j}$ in the corresponding case for $n$ is even. Next, we give the expansions of $R_{V}(\mu)$ in $B(s, -s')$ with $s, s'>\frac{n}{2}+2k$. Recall that $g(\mu)=d_{k}\ln(\mu)+c_{k}$ with $d_{k}\in\mathbf{R}\setminus\{0\}$ and $c_{k}\in\mathbf{C}\setminus\mathbf{R}$, see Proposition \ref{free-expansions} $(iv)$.
	\begin{theorem}\label{RV-expansions-2m-even}
	For $n=4m+2-2k$ with $k\in\mathbf{N}$ chosen as follows, let $|V(x)|\lesssim (1+|x|)^{-\beta}$ with some $\beta>n+4k$. Then for $0<|\mu|\ll1$, we have the following expansions of $R_V(\mu^{2m})$ in $B(s, -s')$ with $s, s'>\frac{n}{2}+2k$ :
		
		(i) For $1\le k\le m$, if zero is a regular point of $H$, then we have
		\begin{equation*}\label{RV-2m-odd-regular}
			R_{V}(\mu^{2m})=C_0+\sum_{\ell=1}^{k-1}\mu^{2(m+\ell-k)}\tau_{\ell}C_{\ell}+\mu^{2m}g(\mu)C_k+\mu^{2m}C_{k+1}+O(\mu^{2m+})
		\end{equation*}
		where  $C_{\ell}\in B(s, -s')$ are selfadjoint operators and $\tau_{\ell}\in\mathbf{C}\setminus\mathbf{R}$ for all $\ell$. Furthermore, $C_0=G_0-G_0v(T_0)^{-1}vG_0$.
		
		(ii) For $1 \le k \le [\frac{m}{2}]+1$, if zero is  the $j$-th kind of resonance of $H$ with $1\le j\le k-1$, then  we have
		\begin{equation*}\label{RV-2m-odd-j}
			\begin{split}
				R_{V}(\mu^{2m})=&\frac{\tau_{0, 0}^jC_{0, 0}^j}{\mu^{2(m-k+j)}}+\sum_{\ell=1}^{k-j}\frac{\tau_{\ell, 0}^jC_{\ell,0}^j}{\mu^{2(m-k+j-\ell)}}+
				\sum_{\ell=k-j+1}^{m-k+j-1}\Bigg[\frac{\tau_{\ell, 0}^jC_{\ell, 0}^j}{\mu^{2(m-k+j-\ell)}}+\frac{\ln(\mu)\tau_{\ell, 1}^jC_{\ell, 1}^j}{\mu^{2(m-k+j-\ell)}}\Bigg]\\
				& +\ln(\mu)\tau_{m-k+j, 1}^jC_{m-k+j, 1}^j+\tau_{m-k+j, 0}^jC_{m-k+j, 0}^j+O(\mu^{0+}),
			\end{split}
		\end{equation*}
		where  $C_{\ell, 0}^j, C_{\ell, 1}^j\in B(s, -s')$ are selfadjoint operators and $\tau_{\ell, 0}^j, \tau_{\ell, 1}^j\in\mathbf{C}\setminus\mathbf{R}$ for all $\ell$. Furthermore, $C_{0, 0}^j=-G_0vS_{j}(T_j)^{-1}S_{j}vG_0$ and all $C_{\ell, 0}^{j} , C_{\ell, 1}^{j}$ for $ 0\leq\ell\leq m-k+j-1 $ are finite rank.
		
		(iii) For $1\le k\le [\frac{m}{2}]+1$, if zero is  the $k$-th kind of resonance of $H$, then  we have
		\begin{equation*}\label{RV-2m-odd-k}
			\begin{split}
				R_{V}(\mu^{2m})=&\frac{\tau_{0, 1}^kC_{0, 1}^k}{\mu^{2m}g(\mu)}+\frac{\tau_{0, 2}^k C_{0, 2}^k}{\mu^{2m}\big(g(\mu)\big)^2}+\sum_{\ell=1}^{m-1}\Bigg[\frac{\tau_{\ell, 0}^kC_{\ell, 0}^k}{\mu^{2(m-\ell)}}+\frac{\tau_{\ell, 1}^kC_{\ell, 1}^k}{\mu^{2(m-\ell)}g(\mu)}
				+\frac{\tau_{\ell, 2}^kC_{\ell, 2}^k}{\mu^{2(m-\ell)}\big(g(\mu)\big)^2}\Bigg]\\
				&+\tau_{m, 0}^kC_{m, 0}^k+\big(g(\mu)\big)^{-1}\tau_{m, 1}^kC_{m, 1}^k+\big(g(\mu)\big)^{-2}\tau_{m, 2}^k C_{m, 2}^k+O(\mu^{0+}),
			\end{split}
		\end{equation*}
		where  $C_{\ell, 0}^k, C_{\ell, 1}^k, C_{\ell, 2}^k\in B(s, -s')$ are selfadjoint operators and $\tau_{\ell, 0}^k, \tau_{\ell,1}^k, \tau_{\ell, 2}^k\in\mathbf{C}\setminus\mathbf{R}$ for all $\ell$. Furthermore, $C_{0, 1}^k=-G_0vS_{k}(T_k)^{-1}S_{k}vG_0$ and all $C_{\ell, 0}^{k} , C_{\ell, 1}^{k}$ for $ 0\leq\ell\leq m-1 $ are finite rank.
		
		(iv)  For $1\le k\le [\frac{m}{2}]+1$, if zero is an eigenvalue of $H$, then we have
		\begin{equation*}\label{RV-2m-odd-k+1}
			\begin{split}
				R_{V}(\mu^{2m})=&\frac{C_{0, 0}^{k+1}}{\mu^{2m}}+\frac{\tau_{0, 1}^{k+1}C_{0, 1}^{k+1}}{\mu^{2m}g(\mu)}+\sum_{\ell=1}^{m-1}\Bigg[\frac{\tau_{\ell, 0}^{k+1}C_{\ell, 0}^{k+1}}{\mu^{2(m-\ell)}}
				+\frac{\tau_{\ell, 1}^{k+1}C_{\ell, 1}^{k+1}}{\mu^{2m}g(\mu)}\Bigg]\\
				&+\tau_{m, 0}^{k+1}C_{m, 0}^{k+1}+\big(g(\mu)\big)^{-1}\tau_{m, 1}^{k+1}C_{m, 1}^{k+1}+O(\mu^{0+}),
			\end{split}
		\end{equation*}
		where $C_{\ell, 0}^{k+1}, C_{\ell, 1}^{k+1}\in B(s, -s')$ are selfadjoint operators and $\tau_{\ell, 0}^{k+1}, \tau_{\ell, 1}^{k+1}\in\mathbf{C}\setminus\mathbf{R}$ for all $\ell$. Furthermore, $C_{0, 0}^{k+1}=-G_0vS_{k+1}(T_{k+1})^{-1}S_{k+1}vG_0$ and all $C_{\ell, 0}^{k+1} , C_{\ell, 1}^{k+1}$ for $ 0\leq\ell\leq m-1 $ are finite rank.
	\end{theorem}

	\subsection{Identification of zero resonance spaces}
	
	In the above subsection, we obtain the asymptotic expansion of  $R_{V}(\mu^{2m})$ as $\mu\rightarrow0$ with the presence of zero resonance or eigenvalue. For different dimensional cases, the number of the kind of zero resonance is different by Definition \ref{resonance}.  In this subsection, we identify all the kinds of zero resonance spaces. The proofs of the following propositions,  Proposition \ref{classification-4m} --  \ref{classification-2m-even} are placed in Section \ref{proof}.
	
	\begin{proposition}\label{classification-4m}
	For $n>4m$,	let $|V(x)|\lesssim(1+|x|)^{-\beta}$ with some $\beta>n+4$. Then $\phi\in S_1L^2(\mathbf{R}^d)\setminus\{0\}$ if and only if $\phi=Uv\psi$ with $\psi\in L^2(\mathbf{R}^n)$ such that $H\psi=0$ holds in the distributional sense and
		\[\psi(x)=-c\int_{\mathbf{R}^n}\frac{v(y)\phi(y)}{|x-y|^{n-2m}}dy=-G_{0}v\psi.\]
	\end{proposition}
	
	\begin{proposition}\label{classification-2m-odd}
	For $n=4m+1-2k$ with $1\le k\le [\frac{m}{2}]+1$ and $n$ is odd,	assume that $|V(x)|\lesssim(1+|x|)^{-\beta}$ with some $\beta>n+4k$. Then the following statements hold:
		
		(i) $\phi(x)\in S_{1}L^{2}(\mathbf{R}^{n})\setminus\{0\}$ if and only if $\phi(x)=Uv(x)\psi(x)$ where $\psi(x)\in W_{2m-\frac{n}{2}}(\mathbf{R}^{n})$ satisfies $H\psi(x)=0$ in the  distributional sense.
		
		(ii) For  $2\le j\le k$,  $\phi(x)=Uv(x)\psi(x)\in S_{j}L^{2}(\mathbf{R}^{n})\setminus\{0\}$ if and only if $\psi(x)\in W_{2m-\frac{n}{2}-j+1}(\mathbf{R}^{n})$  satisfies $H\psi(x)=0$ in the  distributional sense.
		
		(iii) $\phi(x)=Uv(x)\psi(x)\in S_{k+1}L^{2}(\mathbf{R}^{n})\setminus\{0\}$ if and only if $\psi(x)\in L^{2}(\mathbf{R}^{n})$ satisfies $H\psi(x)=0$ in the  distributional sense.
	\end{proposition}
	
	\begin{proposition}\label{classification-2m-even}
		 For $n=4m+2-2k$ with $1\le k\le [\frac{m}{2}]+1$ and $n$ is even, assume that $|V(x)|\lesssim (1+|x|)^{-\beta}$ with some $\beta>n+4k$. Then the following statements hold:
		
		(i) $\phi(x)\in S_{1}L^{2}(\mathbf{R}^{n})\setminus\{0\}$ if and only if $\phi(x)=Uv(x)\psi(x)$ where $\psi(x)\in W_{2m-\frac{n}{2}}(\mathbf{R}^{n})$ satisfies $H\psi(x)=0$ in the  distributional sense.
		
		(ii) For $2\le j\le k-1$,  $\phi(x)=Uv(x)\psi(x)\in S_{j}L^{2}(\mathbf{R}^{n})\setminus\{0\}$ if and only if $\psi(x)\in W_{2m-\frac{n}{2}-j+1}(\mathbf{R}^{n})$  satisfies $H\psi(x)=0$ in the distributional sense.
		
		(iii) $\phi(x)=Uv(x)\psi(x)\in S_{k}L^{2}(\mathbf{R}^{n})\setminus\{0\}$ if and only if $\psi(x)\in W_{0}(\mathbf{R}^{n})$ satisfies $H\psi(x)=0$ in the distributional sense.
		
		(iv) $\phi(x)=Uv(x)\psi(x)\in S_{k+1}L^{2}(\mathbf{R}^{n})\setminus\{0\}$ if and only if $\psi(x)\in L^{2}(\mathbf{R}^{n})$ satisfies $H\psi(x)=0$ in the distributional sense.
	\end{proposition}

	\begin{remark}
		(1) For $H=(-\Delta)^m+V$ and $|V(x)|\lesssim (1+|x|)^{-\beta}$ with some $\beta>n+4$,  Proposition \ref{classification-4m} shows that  $H$ does not exist zero resonance  if $n>4m$. \par (2) If $S_{1}=0$, then zero is a regular point of $H$.   In the case of $3m\leq n\leq 4m$, if $S_{1}=S_{k+1}$, then zero is purely an eigenvalue of $H$; if $S_{k+1}=0$, then zero is purely a resonance of $H$; if $S_{k+1}\neq 0$, then $S_{j}L^{2}(\mathbf{R}^n)$ for $1\leq j< k+1$ contain the mixed state i.e. zero is both eigenvalue and resonance of $H$.
	\end{remark}

	\section{Higher energy decay estimates of $R_V(z)$}\label{high energy}
	This section are devoted the two aims. The first one is to obtain the higher energy decay estimates of $R_V(z)$. The other one is to study the boundary behavior of $R_V(z)$, which is to show that the following limit
	$$R_V(\lambda\pm i0)=s-\lim_{\epsilon\downarrow 0}R_V(\lambda\pm i\epsilon)$$
	exists in $B(s, -s')$ with suitable $s, s'>0$.
	
	\subsection{Higher energy decay estimates}
	In this section, we aim to study the decay rate of  $R_V(z)$ in $B(s, -s')$ as $z$ goes to infinity with some suitable $s, s'>0$.  Recall the symmetric resolvent identity \eqref{symmetric-resolvent-idnetity}:
	\begin{equation*}
		R_V(z)=R_0(z)-R_0(z)v\big[M(z)\big]^{-1}vR_0(z).
	\end{equation*}
	If  $M(z)=U+vR_0(z)v$ has uniform bounded inverse in $L^{2}(\mathbf{R}^n)$, then  $R_V(z)$ has the same decay rate as the  $R_{0}(z)$ in $B(s, -s')$ as $z\rightarrow\infty$.
	
	The free resolvent decomposition identity \eqref{free-resolvent-identity} shows that one can obtain higher energy decay estimates of $R_0(z)$ from the respect estimate of $R(-\Delta; \zeta)$. The following result is the fundamental Agmon-Kato estimate on decay rate for the free resolvent of Schr\"{o}dinger  operator $R(-\Delta; \zeta)$ as $\zeta$ goes to infinity in the weighted Lebesgue norms. It plays a crucial role in time-decay estimates of the solution to Schr\"{o}dinger equation.
	\begin{lemma}\label{high-free}
		(\cite[Theorem 16.1]{KK}) For $\zeta\in\mathbf{C}\setminus[0,+\infty)$, $k=0,1,2,3,\cdots$, any $s, s'>k+\frac{1}{2}$ and any $a>0$, then
		\begin{equation*}\label{22}
			\big\|R^{(k)}(-\triangle; \zeta)\big\|_{B(s, -s')}\leq C(s, a)|\zeta|^{-k/2},\,\,|\zeta|\ge a.
		\end{equation*}
	\end{lemma}
	
	Note that the proof of Theorem 16.1 in \cite{KK} does not depend on the dimension $n$. The following is the similar conclusion for $H_0=(-\Delta)^m$.
	\begin{proposition}\label{free-high energy}
		For $z\in\mathbf{C}\setminus[0,+\infty),$ $k=0,1,2,3,\cdots$, any $s, s'>k+\frac{1}{2}$ and $a>0$, then
		\begin{equation}\label{eq-high-free}
			\big\|R^{(k)}_0(z)\big\|_{B(s, -s')}\le C(s, a)|z|^{-\frac{(2m-1)(k+1)}{2m}},\ \ \ |z|\ge a.
		\end{equation}
	\end{proposition}
	\begin{proof}
		Firstly, we prove decay estimate \eqref{eq-high-free} for $k=0$.  By identity \eqref{free-resolvent-identity} and Lemma \ref{high-free}, we have
		\begin{equation*}\label{eq-free-1}
			\big\| R_0(z)\big\|_{B(s, -s')}
				\le\frac{1}{m|z|}\sum_{k=0}^{m-1}|z_k|\cdot\big\| R(-\Delta; z_k)\big\|_{B(s, -s')}
				\le C(s, a)|z|^{-\frac{(2m-1)}{2m}}.
		\end{equation*}
		Now we check decay estimate \eqref{eq-high-free} for $k\ge 1$. For $R_0(z)$ we have the recurrent relations
		\begin{equation}\label{eq-recurrent}
			zR^{(k)}_0(z)=-kR^{(k-1)}_0(z)+\frac{1}{2m}\big[x\cdot\nabla,\,\, R^{(k-1)}_0(z)\big].
		\end{equation}
		By a similar induction process as in \cite[P. 65]{KK}, we get the desired estimate \eqref{eq-high-free}.
	\end{proof}

	Next, we prove that $M(z)=U+vR_0(z)v$  has uniform bounded inverse in $L^2(\mathbf{R}^n)$. Note that, $V$ is $H_0-$relative bounded under the assumption of $V(x)$ in Theorem \ref{Kato-Jensen}.  Thus there exists a finite constant $V_0\in\mathbf{R}$, such that for any $\lambda\in\mathbf{R}\setminus[V_0,+\infty)$, $H-\lambda=(-\Delta)^m+V-\lambda>0$, then we have $\mathbf{C}\setminus[V_0,+\infty)\subset\rho(H)$ where $\rho(H)$ is the resolvent set of $H$ .
	\begin{lemma}\label{RV-inverse}
		Let $|V(x)|\lesssim (1+|x|)^{-\beta}$ with some $\beta>2$.  Then for $z\in\mathbf{C}\setminus[0, +\infty)$, $vR_0(z)v\in B(0, 0)$ are compact operators. Moreover, $M(z)=U+vR_0(z)v$ is invertible in $L^{2}(\mathbf{R}^n)$ for $z\in\mathbf{C}\setminus[V_0, +\infty)$.
	\end{lemma}
	\begin{proof}
		By the free resolvent decomposition identity \eqref{free-resolvent-identity}, we have
		\begin{equation*}
			vR_0(z)v=\frac{1}{mz}\sum_{k=0}^{m-1}z_k v(-\Delta-z_k)^{-1}v.
		\end{equation*}
		Thus we only need to show $v(-\Delta-z_k)^{-1}v$ is compact operator in $L^{2}(\mathbf{R}^n)$. Since $z\in\mathbf{C}\setminus [0, +\infty)$ and recall that we choose $0<\arg(z)<2\pi$, thus $z_{k}\neq0$ and ${\rm Im} (z_{k}^{1/2})\geq0$. Furthermore, from \cite[Lemma 3.1]{J}, we know the integral kernel $k(z_{k}; x, y)$ of $(-\Delta-z_k)^{-1}$ satisfies the following bound
		\begin{equation}
			\big|k(z_{k}; x, y)\big|\lesssim\big(|x-y|^{2-n}+|x-y|^{(1-n)/2}\big)(1+|z_{k}|)^{(n-3)/4},\,\, n\geq3.
		\end{equation}
		Thus one can check that the Hilbert-Schmidt norm of $v(-\Delta-z_k)^{-1}v$ is finite under the assumption on $V(x)$.
		
		Next, we show that $M(z)=U+vR_0(z)v$ is invertible by  Fredholm's alternative theorem.  We claim that $(H-z)\psi=0$ only has trivial solution in $L^{2}(\mathbf{R}^n)$.  In fact, for $z\in \mathbf{C}\setminus\mathbf{R}$, if $\psi\neq0$, then
		\begin{equation*}
			{\rm Im}\big((H-z)\psi, \psi\big)=-{\rm Im}\, z(\psi, \psi)\neq0.
		\end{equation*}
		Since $\big((H-z)\psi, \psi\big)=(H\psi, \psi)-z(\psi, \psi)$ and $(H\psi, \psi)\in\mathbf{R}$, hence $(H-z)\psi\neq0$ if $\psi\neq0$.  For $z\in\mathbf{C}\setminus\mathbf{R}$ and ${\rm Re} (z)< V_{0}$, then $${\rm Re}\big((H-z)\psi, \psi\big)\geq {\rm Re}(V_{0}-z)(\psi, \psi)\neq0 $$
		provided $\psi\neq0$. Thus $(H-z)\psi\neq0$ if $\psi\neq0$.
		
		Note that operator $M(z)=U+vR_0(z)v$ is the symmetric form of $1+VR_{0}(z)$. Since $M(z)\psi=0$ implies that $(H-z)\psi=0$, thus $M(z)\psi=0$ only has zero solution. Hence, Fredholm's alternative theorem tells that $M(z)=U+vR_0(z)v$ is invertible in $L^{2}(\mathbf{R}^n)$.
	\end{proof}
	
	Using Proposition \ref{free-high energy} and Lemma \ref{RV-inverse}, we obtain the following higher energy decay estimates for $R_V(z)$.
	\begin{proposition}\label{RV-high energy}
		For $k=0,1,2,3,\cdots$, let $|V(x)|\lesssim (1+|x|)^{-\beta}$ with some $\beta>2+2k$.  For $z\in\mathbf{C}\setminus[V_0, +\infty)$, any $s, s'>k+\frac{1}{2}$ and $a>0$, then
		\begin{equation}\label{eq-high}
			\Big\|R^{(k)}_V(z)\Big\|_{B(s, -s')}\le C(s, a)|z|^{-\frac{(2m-1)(1+k)}{2m}}.
		\end{equation}
	\end{proposition}
	\begin{proof}
		For $k=0$, by  identity \eqref{symmetric-resolvent-idnetity} and Proposition \ref{free-high energy}, the above bound \eqref{RV-high energy} holds by the uniformly boundedness of $M(z)^{-1}$ for large $z\in\mathbf{C}\setminus[V_0, +\infty)$ in $L^{2}$.  It is equivalent to prove that for large $z\in\mathbf{C}\setminus[0,+\infty),$
		\begin{equation*}\label{eq-f}
			\big\|f\big\|_{L^2(\mathbf{R}^n)}\le C\big\|\big(U+vR_0(z)v\big)f\big\|_{L^2(\mathbf{R}^n)}.
		\end{equation*}
		In fact, by the triangle inequality we have
		\begin{equation*}\label{eq-triangle}
			\big|\|f\|_{L^2}-\|vR_0(z)vf\|_{L^2}\big|\le\big\|\big(U+vR_0(z)v\big)f\big\|_{L^2}\le \|f\|_{L^2}+\big\|vR_0(z)vf\big\|_{L^2}.
		\end{equation*}
		By the decay estimate \eqref{eq-high-free}, thus for $|z|$ large enough, we have
		\begin{equation*}
			\big\|vR_0(z)vf\big\|_{L^2}\le C(s, a)|z|^{-\frac{2m-1}{2m}}\|f\|_{L^2}\leq \frac{1}{4}\|f\|_{L^2}.
		\end{equation*}
		
		For $k\geq1$,  differentiating \eqref{symmetric-resolvent-idnetity} $k$ times in $z$, we have
		\begin{equation}
			R^{(k)}_{V}(z)=R^{(k)}_{0}(z)-\sum\limits_{k_1+k_2+k_3=k}R^{(k_1)}_{0}(z)v\frac{d^{k_2}}{dz^{k_2}}\Big[M(z)^{-1}\Big]vR^{(k_3)}_{0}(z).
		\end{equation}
		Note that the derivative term $\frac{d^{k_2}}{dz^{k_2}}\big[M(z)^{-1}\big]$ is the linear combination of terms such as $\big[M(z)^{-1}\big]^{j}M^{(\ell)}(z)$ with $0\leq j, \ell< k_{2}$. By the representation of $M(z)$, we know $M^{(\ell)}(z)=vR^{(\ell)}_{0}(z)v$	for $\ell\geq 1$. Since $v(x)(1+|x|)^{\ell+1/2}\in L^{\infty}$ under the assumption of $V(x)$, thus  \eqref{RV-high energy} holds by the  mathematical induction.
	\end{proof}

	\subsection{Limiting absorption principle of $R_V(z)$}
	The limiting absorption principle  means the existence and continuity of the resolvent in the continuous spectrum. The continuity of the resolvent for Schr\"odinger operator in the weighted Sobolev norms was established by Agmon \cite{Agmon}. H\"ormander \cite{H2} also considered such problem for  general  selfadjoint operator $P(D)$ with real coefficient.
	
	Denote by $\mathbf{C}^{+}=\{ z\in\mathbf{C}: {\rm Im}\,z>0\}$ and $\mathbf{C}^{-}=\{ z\in\mathbf{C}: {\rm Im}\,z<0\}$. Define ${\bf \Xi}$ be the disjoint union of $\overline{\mathbf{C}^{+}}$ and $\overline{\mathbf{C}^{-}}$ with the identified points $z\leq0$. Recall the limiting absorption principle results  of free Schr\"{o}dinger operator $R(-\Delta; \zeta)$ (see \cite{JK, J, J1}), we have:
	\begin{lemma}(\cite[Theorem 8.1]{JK})
		For $k\ge 0$ and $s, s'>k+\frac{1}{2}$, we have $R^{(k)}(-\Delta; \zeta)\in B(s, -s')$ is analytic for $\zeta \in{\bf \Xi}\setminus\{0\}$. Furthermore, the boundary values
		\begin{equation*}
			R^{(k)}(-\Delta; \lambda\pm i0)=\lim_{\epsilon\downarrow0}R^{(k)}(-\Delta; \lambda\pm i\epsilon)\in B(s, -s')
		\end{equation*}
		exist for any $\lambda\in(0, \infty)$. The decay estimate \eqref{22} can be extended from $\zeta \in\mathbf{C}\setminus[0,+\infty)$ to $\zeta\in{\bf \Xi}\setminus\{0\}$.
	\end{lemma}
	
	By the resolvent identity \eqref{free-resolvent-identity}, then for  $R_0(z)$ we have:
	\begin{corollary}\label{analytic}
		For $k\ge 0$ and $s, s'>k+\frac{1}{2}$, we have $R_0^{(k)}(z)\in B(s, -s)$ is analytic for $z \in{\bf \Xi}\setminus\{0\}$. Furthermore, the boundary values
		\begin{equation*}
			R_0^{(k)}(\lambda\pm i0)=\lim_{\epsilon\downarrow0}R^{(k)}_0(\lambda\pm i\epsilon)\in B(s, -s')
		\end{equation*}
		exist for any $\lambda\in(0, \infty)$, and the bound
		\begin{equation}
			\Big\|R_0^{(k)}(z)\Big\|_{B(s, -s')}=O\Big(|z|^{-\frac{(2m-1)(k+1)}{2m}}\Big)
		\end{equation}
		holds as $z\rightarrow\infty$ in ${\bf \Xi}\setminus\{0\}$.
	\end{corollary}
	%\begin{remark}
%		For general operator $f(-\Delta)$,  Ben-Artzi and  Nemirovsky \cite[Theorem 2A]{Ben-Nem}  established the limiting absorption results for the free resolvent of $f(-\Delta)$. They assumed $f$ satisfies the following assumptions: $f(\theta)$ is a real-valued positive continuously differentiable function for $\theta\in(0, +\infty)$ and its derivative $f'(\theta)$ is positive and locally H\"{o}lder continuous. For instance, $f(-\Delta)=\sqrt{1-\Delta}$.
%	\end{remark}

	\begin{lemma}\label{compact-LAP}
		(1) Let $|V(x)|\lesssim (1+|x|)^{-\beta}$ with some $\beta>2m$. Then  $vR_0(0\pm i0)v\in B(0, 0) $  are compact operators.
	\par	(2) Let $|V(x)|\lesssim (1+|x|)^{-\beta}$ with some $\beta>2$.  Then for $\lambda>0$,  $vR_0(\lambda\pm i0)v\in B(0, 0) $  are compact operators.
	\end{lemma} 	
	\begin{proof}
		When $\lambda=0$, since $v(x)(1+|x|)^{m+}\in L^{\infty}$ and the Hilbert-Schmidt norm of $(1+|x|)^{-m-}|x-y|^{2m-n}(1+|y|)^{-m-}$ is finite, thus $vR_0(0\pm i0)v$  is compact in $L^{2}(\mathbf{R}^n)$.  The second conclusion holds by the same argument as in Lemma \ref{RV-inverse}.
	\end{proof}
	
	\begin{lemma}\label{RV-LAP}
		For $H=(-\Delta)^m+V$, assume that $H$ has no positive embedded eigenvalue.
		\par (1) Let $|V(x)|\lesssim (1+|x|)^{-\beta}$ with some $\beta>2$. Then for
		$s, s'>1/2$, $R_V(z)\in B(s, -s')$ is continuous for $z\in{\bf  \Xi}\setminus(\Sigma\cup\{0\})$. Furthermore, the boundary value
		\[R_V(\lambda\pm i0)=\lim_{\epsilon\downarrow0}R_V(\lambda\pm i\epsilon)\in B(s, -s')\]
		exists for $\lambda\in \sigma_c(H)\setminus (\Sigma\cup\{0\})$.
		\par (2) Let $|V(x)|\lesssim (1+|x|)^{-\beta}$ with some $\beta>2m$. Assume that $0$ is a regular point of $H$. Then for $s, s'>m$, the function $R_V(z)\in B(s, -s')$ defined on $z\in{\bf  \Xi}\setminus\Sigma$ is continuous at $z=0$.
	\end{lemma}
	\begin{proof}
		The conclusions follow from Lemma \ref{compact-LAP} and the symmetric resolvent identity \eqref{symmetric-resolvent-idnetity} provided
		\begin{equation*}
			\big[U+vR_{0}(\lambda\pm i\epsilon)v\big]^{-1}\rightarrow \big[U+vR_{0}(\lambda\pm i0)v\big]^{-1},\,\, \epsilon\downarrow 0.
		\end{equation*}
		The convergence holds if and only if both limit operators $U+vR_{0}(\lambda\pm i0)v: L^2(\mathbf{R}^n)\rightarrow L^2(\mathbf{R}^n) $ is invertible. According to Lemma \ref{compact-LAP} and Fredholm's alternative theorem, it is enough to show that $\big[U+vR_{0}(\lambda\pm i0)v\big]u=0$ only admits zero solution in $L^{2}(\mathbf{R}^n)$. Note that $\big[U+vR_{0}(\lambda\pm i0)v\big]u=0$ implies that $(H-\lambda)u=0$. Thus $u=0$ under the assumptions that zero is a regular point of $H$ and $H$ has no positive eigenvalue.
	\end{proof}

	\begin{theorem}\label{RV-limiting absorption}
		For $k=0,1,2,3,\cdots$, let $|V(x)|\lesssim (1+|x|)^{-\beta}$ with some $\beta>2+2k$.  Then for
		$s, s'>k+\frac{1}{2}$, $R_V^{(k)}(z)\in B(s, -s')$ is continuous for $z\in{\bf \Xi}\setminus(\Sigma\cup\{0\})$. Furthermore, the decay estimates \eqref{eq-high} can be extended from $z\in{\bf \Xi}\setminus[V_0,+\infty)$ to $z\in{\bf \Xi}\setminus (\Sigma \cup\{0\})$,  i.e. the bound
		\begin{equation}\label{eq-limiting}
			\Big\|R^{(k)}_V(z)\Big\|_{B(s, -s')}\lesssim|z|^{-\frac{(2m-1)(1+k)}{2m}}
		\end{equation}
		holds as $z\rightarrow \infty$ in ${\bf \Xi}\setminus(\Sigma\cup\{0\})$.
	\end{theorem}
	\begin{proof}
		This is a corollary of Proposition \ref{RV-high energy} and Lemma \ref{RV-LAP}.
	\end{proof}
	
	\section{Kato-Jensen type decay estimates}\label{Kato-Jensen-estimate}
	
	In this section, with the helps of lower energy asymptotic expansions, higher energy decay estimates and the limiting absorption principle,  we derive the behavior of $\lambda\rightarrow 0$ and $\lambda\rightarrow\infty$ for the spectral density $dE(\lambda)$ of  $H=(-\Delta)^m+V$ in $B(s, -s')$ with suitable $s, s'>0$.  By Stone's formula, the difference of the perturbed resolvent provides the spectral density
	$$E'(\lambda)=\frac{1}{2\pi i}\big(R_{V}(\lambda+i0)-R_{V}(\lambda-i0)\big).$$
	Using the spectral representation theorem, we give the proof of Theorem \ref{Kato-Jensen decay} and the local decay estimates.

	According to the asymptotic expansions of $R_V(\mu^{2m})$ in Section \ref{free-resolvent}, we get the following results of spectral density $dE(\lambda)$ as $\lambda\rightarrow 0$.
	\begin{proposition}\label{E-low}
		For $H=(-\Delta)^m+V$ with $|V(x)|\lesssim (1+|x|)^{-\beta}$ for some $\beta>0$. Let $dE(\lambda)$ be the spectral density of $H$. For $0< \lambda\ll 1$, we obtain the following results:
		
        (I) For $n>2m$, let $\beta>n$ and $s, s'>\frac{n}{2}+2$. If zero is a regular point of $H$, then
			\begin{equation*}\label{E-4m-regular}
				dE(\lambda)=\lambda^{\frac{n-2m}{2m}}L_1+\lambda^{\frac{n-2m+2}{2m}}L_{2}+o(\lambda^{\frac{n-2m+2}{2m}}),
			\end{equation*}
			where $L_1={\rm Im}(\alpha_{\aleph})A_{\aleph}$ and $L_2={\rm Im}(\alpha_{\aleph+1})A_{\aleph+1}$ if $n>4m$. If $2m< n\leq 4m$, for $j=1, 2$:
 \begin{equation*}
 L_{j}=\begin{cases}
 {\rm Im} (\beta_{j})B_{j},\,\,\, & 2m<n\leq 4m\,\, \text{and}\,\, n\,\, \text {is odd};\\
 {\rm Im} (\tau_{j})C_{j},\,\,\, & 2m<n\leq 4m\,\, \text{and}\,\, n\,\, \text {is even}.
 \end{cases}
 \end{equation*}

		(II) For $n>4m$, let $\beta>n+4$ and $s, s'>\frac{n}{2}+2$.  If zero is an eigenvalue of $H$,  then
			\begin{equation*}\label{E-4m-eigenvalue}
				dE(\lambda)=\lambda^{\frac{n-6m}{2m}}\widetilde{A}_\aleph^e+\lambda^{\frac{n-6m+2}{2m}}\widetilde{A}_{\aleph+1}^e+o(\lambda^{\frac{n-6m+2}{2m}}),
			\end{equation*}
			where $\widetilde{A}_\ell^{e}={\rm Im}(\alpha_\ell^e)A_{\ell}^e$ for $\ell=\aleph,\, \aleph+1$.
		
		(III) For $n\leq 4m$ and $n$ is odd. Let $\beta>n+4k$ where $2k=4m-n+1$ and $1\le k \le [\frac{m}{2}]+1$.  Let $s, s'>\frac{n}{2}+2k$, we have the following asymptotic expansions in $B(s, -s')$:
		\begin{itemize}
						
			\item If zero is  the $j$-th kind resonance of $H$ with $1\le j\le k$, then
			\begin{equation*}\label{E-2m-odd-j}
				\begin{split}
					dE(\lambda)=&\frac{\widetilde{B}_0^j}{\lambda^{\frac{2(m-k+j)-1}{2m}}}+\sum_{\ell=1}^{k-j}\frac{\widetilde{B}_\ell^j}{\lambda^{\frac{2(m-k+j-\ell)-1}{2m}}}
					+\sum_{\ell=k-j+1}^{2m-3k+3j-2}\frac{\widetilde{B}_\ell^j}{\lambda^{\frac{2m-3k+3j-\ell-1}{2m}}}\\
					&\ +\widetilde{B}_{2m-3k+3j-1}^j+o(\lambda^{\frac{1}{2m}}),
				\end{split}
			\end{equation*}
			where $\widetilde{B}_{\ell}^{j}={\rm Im}(\beta_\ell^j)B_\ell^j$ for $\ell=0, 1, \cdots, 2m-3k+3j-1$.
			
			\item If  zero is an eigenvalue of $H$, then
			\begin{equation*}\label{E-2m-odd-k+1}
				dE(\lambda)=\frac{\widetilde{B}_{1}^{k+1}}{\lambda^{\frac{2m-1}{2m}}}+\sum_{\ell=2}^{2m-1}\frac{\widetilde{B}_{\ell}^{k+1}}{\lambda^{\frac{2m-l}{2m}}}+\widetilde{B}_{2m}^{k+1}+o(\lambda^{\frac{1}{2m}}),
			\end{equation*}
			where $\widetilde{B}_{\ell}^{k+1}={\rm Im}(\beta_\ell^{k+1})B_\ell^{k+1}$ for $\ell=0, 1, \cdots, 2m$.
		\end{itemize}
		
		(IV) For $n\leq 4m$ and $n$ is even. Let $\beta>n+4k+2$ where $2k=4m-n+2$ and $1\le k \le [\frac{m}{2}]+1$.  Let $s, s'>\frac{n}{2}+2k+1$, we have the following asymptotic expansions in $B(s, -s')$:
		\begin{itemize}	
			
			\item If zero is  the  $j$-th kind of resonance of $H$ with $1\le j\le k-1$, then
			\begin{equation*}\label{E-2m-even-j}
				\begin{split}
				dE(\lambda)=&\frac{\widetilde{C}_{0,0}^j}{\lambda^{\frac{m-k+j}{m}}}+\sum_{\ell=1}^{k-j}\frac{\widetilde{C}_{\ell ,0}^j}{\lambda^{\frac{m-k+j-l}{m}}}+
					\sum_{\ell=k-j+1}^{m-k+j-1}\Bigg[\frac{\widetilde{C}_{\ell ,0}^j}{\lambda^{\frac{m-k+j-\ell}{m}}}+\frac{\ln(\lambda^{1/2m})\widetilde{C}_{\ell,1}^j}{\lambda^{\frac{m-k+j-\ell}{m}}}\Bigg]\\
					& +\ln(\lambda^{1/2m})\widetilde{C}_{m-k+j,1}^j+\widetilde{C}_{m-k+j,0}^j+o(\lambda^{0+}),
				\end{split}
			\end{equation*}
			where $\widetilde{C}_{\ell, \theta}^{j}={\rm Im}(\tau_{\ell, 0}^j)C_{\ell, 0}^j$ for $\ell=0, 1, \cdots, m-k+j$ and $\theta=0, 1$.
			
			\item If zero is  the $k$-th kind of resonance of $H$, then
			\begin{equation*}\label{E-2m-even-k}
				\begin{split}
					dE(\lambda)=&\frac{\widetilde{C}_{0, 1}^k}{\lambda g(\lambda^{1/2m})}+\frac{\widetilde{C}_{0, 2}^k}{\lambda\big(g(\lambda^{1/2m})\big)^2}+\sum_{\ell=1}^{m-1}\Bigg[\frac{\widetilde{C}_{\ell, 0}^k}{\lambda^{\frac{m-\ell}{m}}}+\frac{\widetilde{C}_{\ell, 1}^k}{\lambda^{\frac{m-\ell}{m}}g(\lambda^{1/2m})}
					+\frac{\widetilde{C}_{\ell, 2}^k}{\lambda^{\frac{m-\ell}{m}}\big(g(\lambda^{1/2m})\big)^2}\Bigg]\\
					& +\widetilde{C}_{m, 0}^k+\big(g(\lambda^{1/2m})\big)^{-1}\widetilde{C}_{m, 1}^k+\big(g(\lambda^{1/2m})\big)^{-2}\widetilde{C}_{m, 2}^k+o(\lambda^{0+}),
				\end{split}
			\end{equation*}
			where $\widetilde{C}_{\ell, \theta}={\rm Im}(\tau_{\ell, \theta}^k)C_{\ell, \theta}^k$ for $\ell=0, 1, \cdots, m$ and $\theta=0, 1, 2$.
			
			\item If zero is an eigenvalue of $H$, then
			\begin{equation*}\label{E-2m-even-k+1}
					dE(\lambda)=\frac{\widetilde{C}_{0, 1}^{k+1}}{\lambda g(\lambda^{1/2m})}+\sum_{\ell=1}^{m-1}\Bigg[\frac{\widetilde{C}_{\ell, 0}^{k+1}}{\lambda^{\frac{m-\ell}{m}}}
					+\frac{\widetilde{C}_{\ell, 1}^{k+1}}{\lambda g(\lambda^{1/2m})}\Bigg]
					 +\widetilde{C}_{m, 0}^{k+1}\widetilde{C}_{m, 1}^{k+1}+o(\lambda^{0+}),
			\end{equation*}
			where $\widetilde{C}_{\ell, \theta}= {\rm Im}(\tau_{\ell, \theta}^{k+1})C_{\ell, \theta}^{k+1}$ for $\ell=0, 1, \cdots, m$ and $\theta=0, 1$.
		\end{itemize}
	\end{proposition}
	
	\begin{proof}
		Using Stone's formula and the resolvent asymptotic expansions of $R_V(\mu^{2m})$ in Section \ref{free-resolvent}, we can immediately obtain the above expansions.
	\end{proof}
	
	For spectral density $dE(\lambda)$ as $\lambda\rightarrow\infty$, we have the following conclusions.
	\begin{proposition}\label{E-high}
		For $H=(-\Delta)^m+V$, let $|V(x)|\lesssim (1+|x|)^{-\beta}$ with some $\beta>2+2k$ and $ k\in\mathbf{N}$.  Then for any $s, s'>k+\frac{1}{2}$, the following estimate
		\begin{equation}\label{E-HIGH}
			\frac{d^{k+1}}{d\lambda^{k+1}}E(\lambda)=O\Big(\lambda^{-\frac{(2m-1)(k+1)}{2m}}\Big)
		\end{equation}
		holds in $B(s, -s')$ as $\lambda\rightarrow\infty$ .
	\end{proposition}
	\begin{proof}
		Using Stone's formula and Theorem \ref{RV-limiting absorption}, we  get the above estimate.
	\end{proof}
	
	As an directly application of the lower energy asymptotic expansions and higher energy decay estimates, we prove the local decay estimates and the time asymptotic expansion for the propagator $e^{-itH}$.
	%Actually, the local decay estimate tells that operator $(1+|x|)^{-\sigma}(\sigma>m)$ is $(-\Delta)^{m}+V$-smooth in the sense of Kato, please see Reed and Simon's book \cite{RS2}.
	\begin{proof}[\bf Proof of Theorem \ref{theo-local decay}]
		For $|z|\rightarrow \infty$, by Proposition \ref{E-high} we have
		\begin{equation*}
			\big\|\langle x\rangle^{-\sigma}(H-z)^{-1}\langle x\rangle^{-\sigma}\big\|_{L^{2}(\mathbf{R}^{n})}=O\big(|z|^{-\frac{2m-1}{2m}}\big),\,\,\, \sigma>1/2.
		\end{equation*}
		For $|z|\rightarrow 0$, by the lower energy asymptotic expansion of $R_{V}( z)$, we have
		\begin{equation*}
			\big\|\langle x\rangle^{-\sigma}(H-z)^{-1}\langle x\rangle^{-\sigma}\big\|_{L^{2}(\mathbf{R}^{n})}=O(1), \,\, \sigma>n/2.
		\end{equation*}
		Then the theorem holds by the Corollary of \cite[P.146]{RS2}.
	\end{proof}
	
		\begin{theorem}\label{Kato-Jensen}
For $H=(-\Delta)^m+V(x)$ with $|V(x)|\lesssim (1+|x|)^{-\beta}$ for some $\beta>0$. Assume that $H$ has no positive embedded eigenvalue.
		
        (I) For $n>2m$, let $\beta>n$ and $s, s'>\frac{n}{2}$. If zero is a regular point of $H,$ then in $B(s, -s')$
			\begin{equation*}\label{e-4m-regular}
			e^{-itH}P_{ac}(H)=|t|^{-\frac{n}{2m}}\bar{A}+o(|t|^{-\frac{n}{2m}}),\,\,t\rightarrow\infty
			\end{equation*}
			where the operator $\bar{A}\in B(s,-s').$

		(II) For $n>4m$, let $\beta>n+4$ and $s, s'>\frac{n}{2}+2$. If zero is an eigenvalue of $H$, then in $B(s, -s')$ we have:
			\begin{equation*}\label{e-4m-eigenvalue}
			e^{-itH}P_{ac}(H)=|t|^{2-\frac{n}{2m}}\bar{A}^e+o(|t|^{2-\frac{n}{2m}}), \,\,t\rightarrow\infty
			\end{equation*}
			where the operator $\bar{A}^e\in B(s,-s').$

		(III) For $n\leq 4m$ and $n$ is odd. Let $\beta>n+4k$ where $2k=4m-n+1$ and $1\le k \le [\frac{m}{2}]+1$.  Let $s, s'>\frac{n}{2}+2k$, then in $B(s, -s')$ we have the following asymptotic expansions:
		\begin{itemize}
			
			\item If zero is  the $j$-th kind of resonance of $H$ with $1\le j\le k$, then
			\begin{equation*}\label{Eith-2m-odd-j}
			e^{-itH}P_{ac}(H)=|t|^{\frac{2(j-k)-1}{2m}}\bar{B}^j+o(|t|^{\frac{2(j-k)-1}{2m}}), \,\,t\rightarrow\infty
			\end{equation*}	
		\end{itemize}
			where the operators $\bar{B}^j\in B(s, -s')$.
				\begin{itemize}	
			\item If  zero is an eigenvalue of $H$, then
			\begin{equation*}\label{e-2m-odd-k+1}
			e^{-itH}P_{ac}(H)=|t|^{-\frac{1}{2m}}\bar{B}^{k+1}+o(|t|^{-\frac{1}{2m}}), \,\,t\rightarrow\infty
			\end{equation*}
		\end{itemize}
	where the operator $\bar{B}^{k+1}\in B(s, -s')$.	
		
(IV) For $ n\leq 4m$ and $n$ is even.  Let $\beta>n+4k+2$ where $2k=4m-n+2$ and $1\le k \le [\frac{m}{2}]+1$.  Let $s, s'>\frac{n}{2}+2k+1$, then  in $B(s, -s')$ we have the following asymptotic expansions:
		\begin{itemize}					
			\item If zero is the $j$-th kind of resonance of $H$ with $1\le j\le k-1$, then
			\begin{equation*}\label{e-2m-even-j}
			e^{-itH}P_{ac}(H)=|t|^{\frac{j-k}{m}}\bar{C}^j+o(|t|^{\frac{j-k}{m}}), \,\,t\rightarrow\infty
			\end{equation*}
	\end{itemize}
			where the operators $\bar{C}^j\in B(s, -s')$.
		\begin{itemize}	
			\item If zero is  the $k$-th kind of resonance of $H$, then
			\begin{equation*}\label{e-2m-even-k}
			e^{-itH}P_{ac}(H)=(\ln|t|)^{-1}\bar{C}_{1}^k+(\ln|t|)^{-2}\bar{C}_{2}^k+o\big((\ln|t|)^{-1}\big), \,\,t\rightarrow\infty
			\end{equation*}	\end{itemize}
			where the operators $\bar{C}_1^k, \bar{C}_2^k\in B(s, -s')$.
		\begin{itemize}		
			\item If zero is an eigenvalue of $H$, then
			\begin{equation*}\label{e-2m-even-k+1}
			e^{-itH}P_{ac}(H)=(\ln|t|)^{-1}\bar{C}^{k+1}+o\big((\ln|t|)^{-1}\big), \,\,t\rightarrow\infty
			\end{equation*}
		\end{itemize}where the operator $\bar{C}^{k+1}\in B(s, -s')$.
	\end{theorem}
	\begin{proof}
		Let $\chi(\lambda)$ be a smooth cutoff function with support in $0<\lambda<1$. Then
		\begin{equation*}
			\begin{split}
				e^{-itH}P_{ac}(H)=&2m\int_{0}^{\infty}e^{-it\lambda^{2m}}\lambda^{2m-1} E'(\lambda^{2m}) d\lambda\\
				=&2m\int_{0}^{\infty}e^{-it\lambda^{2m}}\lambda^{2m-1} \chi(\lambda)E'(\lambda^{2m}) d\lambda+2m\int_{0}^{\infty}e^{-it\lambda^{2m}}\big(1-\chi(\lambda)\big)\lambda^{2m-1} E'(\lambda^{2m}) d\lambda\\
				:=&I+II.
			\end{split}
		\end{equation*}
		
		For term $II$, note that the integral is supported in $[1, \infty)$. By \eqref{E-HIGH} we know, it is actually the Fourier transform of the $L^{1}(\mathbf{R})$ integrable function $\big(1-\chi(\lambda)\big)\lambda^{2m-1} E'(\lambda^{2m})$. Thus, by the Riemann-Lebesgue's lemma, we know the contribution of term $II$ is $|t|^{-k}$ for any large $k>0$. 	
		
		For term $I$, we have:  for $f(x)\in C_{0}^{\infty}(\mathbf{R})$,
		\begin{equation}
			\int_{0}^{\infty}e^{i\lambda x}f(x)x^{\tau} dx\sim \sum_{j=0}\theta_{j}\lambda^{-j-1-\tau}
			, \,\, {\rm Re(\tau)>-1},
		\end{equation}
		where $\theta_{j}=i^{j+\tau+1}\frac{j!}{\Gamma(j+1+\tau)}f^{(j)}(0)$. See Stein's book \cite[P. 355]{Stein}.
		
		On the other hand, for $n=4m+2-2k$ with $1\le k\le [\frac{m}{2}]+1$,  we need to treat such term
		$$\int_{0}^{\infty}\frac{e^{-it\lambda}}{\lambda\big((\ln\lambda-a)^2+\pi^2\big)} d\lambda, \,\, a\neq 0.$$
		However, the integral is $O(1/\ln t)$ as $t\rightarrow\infty$. See \cite{J1}.
	\end{proof}	
	
	From the above expansions of $e^{-itH}P_{ac}(H)$ in $B(s, -s')$, then Theorem \ref{Kato-Jensen decay} is a corollary of Theorem \ref{Kato-Jensen}.

	\section{$L^{p}$-decay estimate}\label{Lp}
	
	In this section, we use the Kato-Jensen type decay estimates to derive the $L^{p}$-decay estimate for $e^{-itH}$. For the free propagator $e^{-it(-\Delta)^{m}}$, by the stationary phase methods, we have ( see e.g  Kim, Arnold and Yao \cite{KAY}):
	
	\begin{lemma}\label{free 1to infty} For any $m\ge 1$,
the kernel $K_{t}(x)$ of $e^{-it(-\Delta)^{m}}$ is smooth and satisfies the following pointwise estimate
		\begin{equation*}
			\big|D^\alpha K_{t}(x)\big|\lesssim |t|^{-\frac{n+|\alpha|}{2m}}\big(1+|t|^{-\frac{1}{2m}}|x|\big)^{-\frac{(m-1)n-|\alpha|}{2m-1}},\,\,\, t\neq0.
		\end{equation*}
which implies that for $0\le |\alpha|\leq (m-1)n$,
		\begin{equation*}
			\big\|D^{\alpha}e^{it(-\Delta)^m}u\big\|_{L^{\infty}(\mathbf{R}^n)}\lesssim |t|^{-\frac{n+|\alpha|}{2m}}\|u\|_{L^{1}\mathbf{R}^n}, \ \ t\neq0.
		\end{equation*} In particular,
		\begin{equation*}
			\big\|e^{-it(-\Delta)^{m}}u\big\|_{L^{\infty}(\mathbf{R}^{n})}\lesssim |t|^{-\frac{n}{2m}}\|u\|_{L^{1}(\mathbf{R}^{n})}, \ \ t\neq0.
		\end{equation*}
	\end{lemma}

	In the case that $V(x)\neq0$, we can derive the weak estimates: $L^{1}\cap L^{2}\rightarrow L^{\infty}+L^{2}$-decay estimates.
	
	\begin{definition}
		For any measurable function $f,$ if $f=f_1+f_2$ with $f_1\in L^2(\mathbf{R}^n), f_2\in L^\infty(\mathbf{R}^n)$ and satisfies
		\[\inf \big\{\|f_1\|_{L^2(\mathbf{R}^n)}+\|f_2\|_{L^\infty(\mathbf{R}^n)}\big\}<\infty,\]
		where the infimum takes over all the splitting of $f$. Then we denote $f\in L^2+L^\infty(\mathbf{R}^n)$ and $L^2+L^\infty(\mathbf{R}^n)$ is a Banach space with the norm
		\[\|f\|_{L^2+L^\infty(\mathbf{R}^n)}=\inf \big\{ \|f_1\|_{L^2(\mathbf{R}^n)}+\|f_2\|_{L^\infty(\mathbf{R}^n)} \big\}.\]
	\end{definition}

Note that for $f\in L^2+L^\infty(\mathbf{R}^n)$,  since $f$ can be divided as $f=f+0=0+f$, then $\|f\|_{L^2+L^\infty(\mathbf{R}^n)}\le \|f\|_{L^2(\mathbf{R}^n)}$ and $\|f\|_{L^2+L^\infty(\mathbf{R}^n)}\le \|f\|_{L^\infty(\mathbf{R}^n)}$.

	In the sequel, we need the boundness of $P_{ac}(H)$ in the weighted $L^2$ spaces. Denotes $P_{disc}(H)$ be the projection onto the subspace of discrete spectrum of $H$. Denotes $\sharp$ be the number of discrete spectrum point.  Thus $P_{ac}(H)=I-P_{disc}(H)$ and $P_{disc}(H)=\sum_{j=0}^{\sharp}\langle\cdot, e_{j}\rangle e_{j}$ where $e_{j}$ is the eigenvector. Denotes $N(\gamma; H)$ be the number (count the multiplicity) of eigenvalues of $H$ which is less than or equal to $\gamma$. By using the Birman-Solomyak bound for operator $A_{\ell}(\alpha V)=(-\Delta)^{\ell}-\alpha V$ where $\ell, \alpha>0$ in \cite{BS}, then  for $n>2m$ we have
	\begin{equation}\label{Birman-Solom-bound}
		N\big(0; (-\Delta)^m+V\big)\leq C(n) \int_{\mathbf{R}^n}V_{-}^{n/2m}(x)dx
	\end{equation}
	where $V_{-}(x)=-\min\{0, V(x)\}$.
	
	\begin{lemma}\label{projectionbounded}
		For $H=(-\Delta)^m+V$, $V(x)\in L^{\infty}(\mathbf{R}^n)\cap L^{n/2m}(\mathbf{R}^n)$. Assume that 0 is a regular point of $H$ and there are only finite embedded eigenvalues. Then for $\sigma\in \mathbf{R}$ and any $1\leq p\leq \infty$, we have
		\begin{equation}\label{Pac-bound}
			\|(1+|x|)^{-\sigma}P_{ac}(H)(1+|x|)^{\sigma}f\|_{L^{p}(\mathbf{R}^n)}\lesssim \|f\|_{L^{p}(\mathbf{R}^n)}.
		\end{equation}
	\end{lemma}
	\begin{proof}
		By the Birman-Solomyak bound \eqref{Birman-Solom-bound} of $H$, we know the number $\sharp$ is finite. It is enough to show that $P_{disc}(H)$ satisfies the estimate \eqref{Pac-bound}. Since
		\begin{equation*}
			\begin{split}
				&\big\|(1+|\cdot|)^{-\sigma}P_{ac}(H)(1+|\cdot|)^{\sigma}f\big\|_{L^{p}(\mathbf{R}^n)}\\
				=\,\,&\bigg\|\sum_{j=0}^{\sharp}\bigg\langle(1+|x|)^{\sigma} f,\,\,e_{j}\bigg\rangle (1+|y|)^{-\sigma}e_{j}\bigg\|_{L^{p}(\mathbf{R}^n)} \\
				\lesssim\,\, & \sum_{j=0}^{\sharp} \bigg|\bigg\langle f,\,\,(1+|x|)^{\sigma} e_{j}\bigg\rangle\bigg|\|(1+|y|)^{-\sigma}e_{j}\|_{L^{p}(\mathbf{R}^n)}\\
				\lesssim\,\, &\sum_{j=0}^{\sharp} \|f\|_{L^{p}(\mathbf{R}^n)}\|(1+|x|)^{\sigma} e_{j}\|_{L^{p^{\prime}}(\mathbf{R}^n)}\|(1+|y|)^{-\sigma}e_{j}\|_{L^{p}(\mathbf{R}^n)}.
			\end{split}
		\end{equation*}	
		Furthermore, by the Theorem 14.5.2 in \cite{H2}, we know that eigenfunction $e_{j}$ satisfies
		\begin{equation*}
			(1+|x|)^{N}e_{j}(x)\in L^{2}(\mathbf{R}^n)\,\, \text{for all}\,\,  N\in\mathbf{R}.
		\end{equation*}
		Hence, by the H\"older's inequality we know the sum in the last step is finite.
	\end{proof}
	Now, we start the proof of the $L^{1}\cap L^{2}- L^{\infty}+L^{2}$-decay estimates from the following lemma.

\begin{lemma}\label{time decay}
For any $a, b>0$ and $a+b\neq1$, we have
\begin{equation}\label{eq-time decay}
\int_0^t\frac{ds}{(1+|t-s|)^a(1+|s|)^b}\lesssim\begin{cases}
(1+|t|)^{-a-b+1},\ \  &\ \ 0<a,b<1,\\
(1+|t|)^{-\min\{a,b\}},\ \ &\ \ \text{otherwise}.
\end{cases}
\end{equation}
and we have
\begin{equation}\label{eq-time decay 2}
\int_0^t\frac{ds}{(1+|\ln (t-s)|)(1+|s|)^{a}}\lesssim \begin{cases}
(1+|t|)^{-a+1}(1+|\ln t|)^{-1},\ \  &\ \ 0<a<1,\\
(1+|\ln t|)^{-1},\ \ &\ \ \ a\geq1.
\end{cases}
\end{equation}
\end{lemma}
\begin{proof}
The proof follows from the proof of Lemma 4.4 in \cite{FSY}.
\end{proof}
	\begin{proof}[\bf Proof of Theorem \ref{Ginibre-argument}]
		Our strategy is applying the iterated Duhamel formula
		\begin{equation}\label{Duhamel formula}
			\begin{split}
				e^{-itH}P_{ac}(H) =& e^{-itH_0}P_{ac}(H)+i\int_0^t e^{-i(t-s)H_0}VP_{ac}(H)e^{isH_0}ds\\
				&\ -\int_0^t\int_0^s e^{-i(t-s)H_0}Ve^{-i(s-\tau)H}P_{ac}(H)Ve^{-i\tau H_0}d\tau ds\\
				:=& I+II+III
			\end{split}
		\end{equation}
		and then estimate each term of \eqref{Duhamel formula}. By the definition of $L^2+L^\infty(\mathbf{R}^n),$ we have
		$$\big\|e^{-itH_0}P_{ac}(H)u\big\|_{L^2+L^\infty(\mathbf{R}^n)}\le\min\Big\{\big\|e^{-itH_0}P_{ac}(H)u\big\|_{L^2(\mathbf{R}^n)},\,\,  \big\|e^{-itH_0}P_{ac}(H)u\big\|_{L^\infty(\mathbf{R}^n)}\Big\}.$$
		Then for $0<|t|\le 1,$ we have
		\[\big\|e^{-itH_0}P_{ac}(H)u\big\|_{L^2+L^\infty(\mathbf{R}^n)}\le\|u\|_{L^2(\mathbf{R}^n)}.\]
		And for $|t|>1,$ by Lemma \ref{free 1to infty}, we have
		\[\big\|e^{-itH_0}P_{ac}(H)u\big\|_{L^2+L^\infty(\mathbf{R}^n)}\le |t|^{-\frac{n}{2m}}\|u\|_{L^1(\mathbf{R}^n)}.\]
		Thus for the first term $I$ we have
		\begin{equation}\label{estimates-I}
			\big\|e^{-itH_0}P_{ac}(H)u\big\|_{L^2+L^\infty(\mathbf{R}^n)}\lesssim(1+|t|)^{-\frac{n}{2m}}\|u\|_{L^1\cap L^2(\mathbf{R}^n)}.
		\end{equation}
		\indent For the second term $II$ of \eqref{Duhamel formula}, we have
		\begin{align*}
			&\int_0^t\big\|e^{-i(t-s)H_0}VP_{ac}(H)e^{is H_0}u\big\|_{L^2+L^\infty(\mathbf{R}^n)}ds\\
			\lesssim& \int_0^t\big(1+|t-s|\big)^{-\frac{n}{2m}}\big\|VP_{ac}(H)e^{-isH_0}u\big\|_{L^1\cap L^2(\mathbf{R}^n)}ds\\
			\lesssim&\int_0^t\big(1+|t-s|\big)^{-\frac{n}{2m}}\big\|V\langle x\rangle^\sigma\big\|_{L^\infty\cap L^2(\mathbf{R}^n)}\big\|\langle x\rangle^{-\sigma} P_{ac}(H)e^{-isH_0}u\big\|_{L^2(\mathbf{R}^n)}ds\\
			\lesssim&\int_1^t\big(1+|t-s|\big)^{-\frac{n}{2m}}\big\|\langle x\rangle^{-\sigma}\big\|_{L^2(\mathbf{R}^n)}\big\|P_{ac}(H)e^{-isH_0}u\big\|_{L^\infty(\mathbf{R}^n)}ds\\
			&\ \ +\int_0^1\big(1+|t-s|\big)^{-\frac{n}{2m}}\big\|\langle x\rangle^{-\sigma}\big\|_{L^\infty(\mathbf{R}^n)}\big\|P_{ac}(H)e^{-isH_0}u\big\|_{L^2(\mathbf{R}^n)}ds\\
			\lesssim &\int_1^t\big(1+|t-s|\big)^{-\frac{n}{2m}}s^{-\frac{n}{2m}}ds\|u\|_{L^1\cap L^2(\mathbf{R}^n)}+\int_0^1\big(1+|t-s|\big)^{-\frac{n}{2m}}2^\frac{n}{2m}\big(1+|s|\big)^{-\frac{n}{2m}}ds\|u\|_{L^1\cap L^2(\mathbf{R}^n)}\\
			\lesssim &\,\, \big(1+|t|\big)^{-\frac{n}{2m}}\|u\|_{L^2\cap L^1(\mathbf{R}^n)}.
		\end{align*}
	\par	For the third term $III$ of \eqref{Duhamel formula}, we have
		\begin{align*}
			&\ \int_0^t\int_0^s\Big\|e^{-i(t-s)H_0}Ve^{i(s-\tau)H}P_{ac}(H)Ve^{-i\tau H_0u}\Big\|_{L^2+L^\infty}d\tau ds\\
			\lesssim &\int_0^t\int_0^s\big(1+|t-s|\big)^{-\frac{n}{2m}}\big\|Ve^{-i(s-\tau)H}P_{ac}(H)Ve^{-i\tau H_0}u\big\|_{L^1\cap L^2}d\tau ds\\
			\lesssim & \int_0^t\int_0^s \big(1+|t-s|\big)^{-\frac{n}{2m}}\big\|V\langle x\rangle^\sigma\big\|_{L^\infty+L^2}\big\|\langle x\rangle^{-\sigma}e^{-i(s-\tau)H}P_{ac}(H)Ve^{-i\tau H_0}u\big\|_{L^2}d\tau ds\\
			\lesssim &\int_0^t\int_0^s\big(1+|t-s|\big)^{-\frac{n}{2m}}\big\|\langle x\rangle^{-\sigma}e^{-i(s-\tau)H}P_{ac}(H)\langle x\rangle^{-\sigma}\big\|_{L^2\rightarrow L^2}\big\|\langle x\rangle^{\sigma}V e^{-i\tau H_0}u\big\|_{L^2}d\tau ds\\
			\lesssim & \int_0^t\int_0^1\big(1+|t-s|\big)^{-\frac{n}{2m}}\big(1+|s-\tau|\big)^{-\frac{n}{2m}}\big\|\langle x\rangle^{\sigma}V\big\|_{L^\infty}\big\|e^{-i\tau H_0}u\big\|_{L^2}d\tau ds\\
			&\ \ + \int_0^t\int_1^s \big(1+|t-s|\big)^{-\frac{n}{2m}}\big(1+|s-\tau|\big)^{-\frac{n}{2m}}\big\|\langle x\rangle^\sigma V\big\|_{L^2}\big\|e^{-i\tau H_0}u\big\|_{L^\infty}d\tau ds\\
			\lesssim &\int_0^t\int_0^1\big(1+|t-s|\big)^{-\frac{n}{2m}}\big(1+|s-\tau|\big)^{-\frac{n}{2m}}2^\frac{n}{2m}\big(1+|\tau|\big)^{-\frac{n}{2m}}d\tau ds\|u\|_{L^2}\\
			&\ \ +\int_0^t\int_1^s\big(1+|t-s|\big)^{-\frac{n}{2m}}\big(1+|s-\tau|\big)^{-\frac{n}{2m}}\tau^{-\frac{n}{2m}}d\tau ds \big\|u\big\|_{L^1}\\
			\lesssim &\,\, \big(1+|t|\big)^{-\frac{n}{2m}}\big\|u\big\|_{L^1\cap L^2(\mathbf{R}^n)}.
		\end{align*}
		Thus we can combine the steps above to conclude the proof  for  the regular case of Theorem \ref{Ginibre-argument}. The resonance cases hold by the same processes as to the regular case.  We only need to plug into the respectively time decay rate of Kato-Jensen decay estimates instead of the regular case.
	\end{proof}

	\section{Endpoint Strichartz estimates} \label{endpoint-strichartz}
	In this section, we consider the following nonlinear higher-order Schr\"odinger equation
	\begin{equation}\label{NHSE}
		\begin{cases}i\partial_t\psi=\big((-\Delta)^m+V\big)\psi+h(t,x),& \\
			\psi(0,x)=\psi_0(x)\in L^2(\mathbf{R}^n),
		\end{cases}
	\end{equation} 	
	where $h(t, x)$ is the source term.	We aim to establish the endpoint Strichartz estimates for the solution of problem \eqref{NHSE}. With the lack of  the $L^{1} - L^{\infty}$ dispersive estimate of $e^{-it((-\Delta)^m+V)}$, we will apply the Kato-Jensen type decay estimates (see Corollary \ref{Kato-Jensen decay}) and the local decay estimates of $e^{-it((-\Delta)^m+V)}$ to obtain the endpoint Strichartz estimates of problem \eqref{NHSE}.
	
	\begin{definition}
		For $n>2m$,  if $q, r\in \mathbf{R}$ and $2\le q\le \infty$ satisfies
		\begin{equation}\label{eq-rongxudui}
			\frac{2m}{q}+\frac{n}{r}=\frac{n}{2},
		\end{equation}
		then we say $(q, r)$ is $m$-admissible pair. Note that if $q=2$, then $r=\frac{2n}{n-2m}$.
	\end{definition}
	
	\begin{theorem}\label{theo-endpoint-strichartz}
		Consider the nonlinear problem  \eqref{NHSE}. Let $|V(x)|\lesssim (1+|x|)^{-\beta}$ with some $\beta>n+4$.  Assume zero is a regular point of $H=(-\Delta)^m+V$ and there does not exist positive embedded eigenvalue of $H$. Then for any $m$-admissible pairs $(q, r)$ and $(\tilde{q}, \tilde{r})$, the following estimates hold:
		\begin{enumerate}
			\item Homogeneous Strichartz estimate
			\begin{equation}\label{homo-strichartz}
			\big\|e^{-itH}P_{ac}(H)\psi_0\big\|_{L_t^qL_x^r(\mathbf{R}\times\mathbf{R}^n)}\le C(n)\|\psi_0\|_{L^2(\mathbf{R}^n)};
			\end{equation}
			
			\item Dual homogeneous Strichartz estimate
			\begin{equation}\label{dual homo-strichartz}
				\Big\|\int_{\mathbf{R}}e^{-isH}P_{ac}(H)h(s,\cdot)ds\Big\|_{L_x^2(\mathbf{R}^n)}\le C(n)\|h\|_{L_t^{\tilde{q}'}L_x^{\tilde{r}'}(\mathbf{R}\times\mathbf{R}^{n})};
			\end{equation}
			
			\item The solution $\psi(x,t)$ of the problem \eqref{NHSE} satisfies
			\begin{equation}\label{strichartz-3}
				\big\|P_{ac}(H)\psi(x,t)\big\|_{L_t^qL_x^r(\mathbf{R}\times\mathbf{R}^n)}\le C(n)\|\psi_0\|_{L^2(\mathbf{R}^n)}+\|h\|_{L_t^{\tilde{q}'}L_x^{\tilde{r}'}(\mathbf{R}\times\mathbf{R}^n)}.
			\end{equation}
		\end{enumerate}
	\end{theorem}
	
	By Keel and Tao's $TT^*$-methods (see \cite{KeelTao}) and $L^1$-$L^\infty$ estimates of the free propagator $e^{-it(-\Delta)^{m}}$, we have the following Strichartz estimates.
	\begin{lemma}\label{free-strichartz}
		For the free propagator $e^{-it(-\Delta)^{m}}$, we have
		\begin{equation}\label{freestri}
			\big\|e^{-it(-\Delta)^{m}}\psi\big\|_{L_t^qL_x^r(\mathbf{R}\times\mathbf{R}^n)}\lesssim\|\psi\|_{L^2(\mathbf{R}^n)},
		\end{equation}
		\begin{equation}\label{freeretarded}
			\Big\|\int_{s<t}e^{-i(t-s)(-\Delta)^{m}}f(s)ds\Big\|_{L_t^qL_x^r(\mathbf{R}\times\mathbf{R}^n)}\lesssim\|f\|_{L_t^{\tilde{q}'}L_x^{\tilde{r}'}(\mathbf{R}\times\mathbf{R}^n)},
		\end{equation}
		where $(q,r)$ and $(\tilde{q}', \tilde{r}')$ are $m$-admissible pairs.
	\end{lemma}
	
	Next, we give the proof of Theorem \ref{theo-endpoint-strichartz}.
	\begin{proof}[\bf Proof of Theorem \ref{theo-endpoint-strichartz}]
		We divide the proof into the following three steps.
		
		Step 1)\,\,We aim to show the homogeneous Strichartz estimate \eqref{homo-strichartz} and the dual homogeneous Strichartz estimate \eqref{dual homo-strichartz}.
		
		Consider the following problem
		\begin{equation}\label{freefourthorder equ}
			\left\{ \begin{gathered}
				i\partial_t \psi = H_0\psi+V\psi, \hfill \\
				\psi(0,\cdot)=\psi_0 \in L^{2}(\mathbf{R}^n) . \hfill \\
			\end{gathered}  \right.
		\end{equation}
		For the homogeneous Strichartz estimate \eqref{homo-strichartz}, using Duhamel formula it is enough to show
		\begin{equation*}
			\Big\|\int_{0}^{t} e^{-i(t-s)H_{0}}Ve^{-isH}P_{ac}(H)\psi ds\Big\|_{L_{t}^{q}L_{x}^{r}(\mathbf{R}\times\mathbf{R}^n)}\lesssim\|\psi\|_{L^{2}(\mathbf{R}^n)}.
		\end{equation*}
		In fact, by \eqref{freeretarded} and the local decay estimates \eqref{local decay}, we have
		\begin{equation*}
			\begin{split}
				&\quad\Big\|\int_{0}^{t} e^{-i(t-s)H_{0}}Ve^{-isH}P_{ac}(H)\psi ds\Big\|_{L_{t}^{q}L_{x}^{r}(\mathbf{R}\times\mathbf{R}^n)}
				\lesssim\big\|Ve^{-itH}P_{ac}(H)\psi \big\|_{L_{t}^{2}L_{x}^{\frac{2n}{n+2m}}(\mathbf{R}\times\mathbf{R}^n)}\\
				&\lesssim\big\|V\langle x\rangle^{\sigma}\big\|_{L^{\frac{n}{m}}(\mathbf{R}^n)}\big\|\langle x\rangle^{-\sigma}e^{-itH}P_{ac}(H)\psi\big\|_{L_{t}^{2}L_{x}^{2}(\mathbf{R}\times\mathbf{R}^n)}\lesssim\|\psi\|_{L^{2}(\mathbf{R}^n)}.
			\end{split}
		\end{equation*}
		Further, the dual homogeneous Strichartz estimate \eqref{dual homo-strichartz} follows by the $TT^*$-method.
		
		Step 2)\,\, We aim to show the retarded Strichartz estimate \eqref{strichartz-3}. The solution $\Psi(t, x)$ of equation \eqref{NHSE} satisfies
		\begin{equation*}
			P_{ac}(H)\Psi(t,x)=e^{-itH_0}P_{ac}(H)\Psi_0-i\int_{0}^{t}e^{-i(t-s)H_0}VP_{ac}(H)\Psi(s)ds-i\int_{0}^{t}e^{-i(t-s)H_0}P_{ac}(H)h(s)ds.
		\end{equation*}
		Then by H\"older's inequality and Step 1, we have
		\begin{equation*}
			\begin{split}
				&\big\|P_{ac}(H)\Psi(t,x)\big\|_{L_{t}^{q}L_{x}^{r}(\mathbf{R}\times\mathbf{R}^n)}\\
				\lesssim&\big\|e^{-itH_0}P_{ac}(H)\Psi_0\big\|_{L_{t}^{q}L_{x}^{r}(\mathbf{R}\times\mathbf{R}^n)}
				+\Big\|\int_{0}^{t}e^{-i(t-s)H_0}P_{ac}(H)h(s,\cdot)ds\Big\|_{L_{t}^{q}L_{x}^{r}(\mathbf{R}\times\mathbf{R}^n)}\\
				&  +\Big\|\int_{0}^{t}e^{-i(t-s)H_0}VP_{ac}(H)\Psi(s)ds\Big\|_{L_{t}^{q}L_{x}^{r}(\mathbf{R}\times\mathbf{R}^n)}\\
				\lesssim &\|\Psi_0\|_{L^2}+\|h(t)\|_{L_{t}^{\tilde{q}^{\prime}}L_{x}^{\tilde{r}^{\prime}}(\mathbf{R}\times\mathbf{R}^n)}
				+\big\|VP_{ac}(H)\Psi(t)\big\|_{L_{t}^{2}L_{x}^{\frac{2n}{n+2m}}(\mathbf{R}\times\mathbf{R}^n)}\\
				\lesssim &\|\Psi_0\|_{L^2}+\|h(t)\|_{L_{t}^{\tilde{q}^{\prime}}L_{x}^{\tilde{r}^{\prime}}(\mathbf{R}\times\mathbf{R}^n)}
				+\|V\langle x\rangle^{\sigma}\|_{L^{\frac{n}{m}}(\mathbf{R}^n)}\big\|\langle x\rangle^{-\sigma}P_{ac}(H)\Psi(t)\big\|_{L_{t}^{2}L_{x}^{2}(\mathbf{R}\times\mathbf{R}^n)}.
			\end{split}
		\end{equation*}
		
		Now we aim to show that
		\begin{equation}
			\big\|\langle x\rangle^{-\sigma}P_{ac}(H)\Psi(t)\big\|_{L_{t}^{2}L_{x}^{2}(\mathbf{R}\times\mathbf{R}^n)}\lesssim\|\Psi_0\|_{L^2(\mathbf{R}^n)}
			+\|h\|_{L_{t}^{\tilde{q}^{\prime}}L_{x}^{\tilde{r}^{\prime}}(\mathbf{R}\times\mathbf{R}^n)}.
		\end{equation}
		First, by Duhamel's formula for $\Psi$, we have
		\begin{equation*}
			\begin{split}
			&	\big\|\langle x\rangle^{-\sigma}P_{ac}(H) \Psi(t)\big\|_{L_{t}^{2}L_{x}^{2}(\mathbf{R}\times\mathbf{R}^n)}\\
		\lesssim &   \big\|\langle x\rangle^{-\sigma}e^{-itH}P_{ac}(H)\Psi_{0}\big\|_{L_{t}^{2}L_{x}^{2}(\mathbf{R}\times\mathbf{R}^n)} +\Big\|\int_{0}^{t}\langle x\rangle^{-\sigma}e^{-i(t-s)H}P_{ac}(H)h(s)ds\Big\|_{L_{t}^{2}L_{x}^{2}(\mathbf{R}\times\mathbf{R}^n)}.
			\end{split}
		\end{equation*}
		Then, we will finish the proof which only needs to show
		\begin{equation}\label{source local decay}
			\Big\|\int_{0}^{t}\langle x\rangle^{-\sigma}e^{-i(t-s)H}P_{ac}(H)h(s)ds\Big\|_{L_{t}^{2}L_{x}^{2}(\mathbf{R}\times\mathbf{R}^n)}
			\lesssim\|\Psi_0\|_{L^2(\mathbf{R})}
			+\|h\|_{L_{t}^{\tilde{q}^{\prime}}L_{x}^{\tilde{r}^{\prime}}(\mathbf{R}\times\mathbf{R}^n)}.
		\end{equation}

		Step 3)\,\, We show the local decay estimate of the source term \eqref{source local decay}.
		
		Consider the Cauchy problem
		\begin{equation}
			\left\{ \begin{gathered}
				i\partial_{t}\phi= H_0\phi+h(t)= H\phi-V\phi+h(t), \hfill \\
				\phi(0,\cdot)=\Psi_{0}. \hfill \\
			\end{gathered}  \right.
		\end{equation}
		Then Duhamel formula for the solution $\phi(t,x)$ reads
		\begin{equation}\label{01}
			P_{ac}(H)\phi(t)=e^{-itH}P_{ac}(H)\Psi_{0}+i\int_{0}^{t}e^{-i(t-s)H}P_{ac}(H)V\phi(s)ds-i\int_{0}^{t}e^{-i(t-s)H}P_{ac}(H)h(s)ds.
		\end{equation}
		For the left hand side of \eqref{01}, since $\phi(t)$ is also a solution of $i\partial_{t}\phi=H_0\phi+h(t)$,
		by \eqref{freestri}, \eqref{freeretarded} and Duhamel formula again, we have
		\begin{equation*}
			\begin{split}
				&\,\,\big\|\langle x\rangle^{-\sigma}P_{ac}(H)\phi(t)\big\|_{L^{2}_{t}L_{x}^{2}(\mathbf{R}\times\mathbf{R}^n)}\\
				\lesssim&\,\,\big\|\langle x\rangle^{-\sigma}P_{ac}(H)\langle x\rangle^{\sigma}\langle x\rangle^{-\sigma}e^{-itH_{0}}\Psi_0\big\|_{L^{2}_{t}L_{x}^{2}(\mathbf{R}\times\mathbf{R}^n)}\\
				& +\Big\|\langle x\rangle^{-\sigma}P_{ac}(H)\langle x\rangle^{\sigma}\langle x\rangle^{-\sigma}\int_{0}^{t}e^{-i(t-s)H_0}h(s)ds\Big\|_{L^{2}_{t}L_{x}^{2}(\mathbf{R}\times\mathbf{R}^n)}\\
				\lesssim&\,\,\|\Psi_0\|_{L^2(\mathbf{R}^n)}+\Big\|\int_{0}^{t}e^{-i(t-s)H_0}h(s)ds\Big\|_{L^{2}_{t}L_{x}^{\frac{2n}{n-2m}}(\mathbf{R}\times\mathbf{R}^n)}\\
				\lesssim&\,\,\|\Psi_0\|_{L^2(\mathbf{R}^n)}+\|h\|_{L_{t}^{\tilde{q}^{\prime}}L_{x}^{\tilde{r}^{\prime}}(\mathbf{R}\times\mathbf{R}^n)}.
			\end{split}
		\end{equation*}
		Here, we apply Lemma \ref{projectionbounded}, the boundness of $P_{ac}(H)$. The local decay estimate for the first term on the right hand side of \eqref{01} follows from \eqref{local decay}.
		
		For the second term of the right hand side of \eqref{01}, notice that
		\begin{equation*}
			\begin{split}
				&\Big\|\int_{0}^{t}\langle x\rangle^{-\sigma}e^{-i(t-s)H}P_{ac}(H)V\phi(s)ds\Big\|_{L^{2}_{t}L_{x}^{2}(\mathbf{R}\times\mathbf{R}^n)}\\
				\lesssim &\,\, \Big\|\int_{0}^{t}\big\|\langle x\rangle^{-\sigma}e^{-i(t-s)H}P_{ac}(H)V\phi(s)\big\|_{L_{x}^{2}(\mathbf{R}^n)}ds\Big\|_{L^{2}_{t}(\mathbf{R})}\\
				\lesssim &\,\, \Big\|\int_{0}^{t}(1+|t-s|)^{-\frac{n}{2m}}\big\|\langle x\rangle^{\sigma}V\phi(s)\big\|_{L_{x}^{2}(\mathbf{R}^n)}ds\Big\|_{L^{2}_{t}(\mathbf{R})}\\
				\lesssim &\,\, \|\Psi_0\|_{L^2}+\|h\|_{L_{t}^{\tilde{q}^{\prime}}L_{x}^{\tilde{r}^{\prime}}(\mathbf{R}\times\mathbf{R}^n)}.
			\end{split}
		\end{equation*}
	\end{proof}

	\section{Absence of embedded eigenvalues}\label{eigenvalue-problem}
	%A natural question is that for what potential $V$,  for $H_{P}=P(D)+V$ we have
%	$$\sigma_{ac}\big(P(D)+V\big)=\sigma_{ac}\big(P(D)\big).$$
%	For higher order operator $P(D)$, this question is complex. By the following examples, for higher order operator $P(D)$ even with rapidly decaying potential $V(x)$ is not enough to ensure that $P(D)+V$ does not hold positive eigenvalue. The examples show that the size and geometry shape of potential make a difference to the existence of positive eigenvalue.
	
	\subsection{Examples of existence of positive embedded eigenvalue}
	In 1929, von Neumann and Wigner \cite{vonNeumannWigner} found an example of  Schr\"{o}dinger operator acting in $L^{2}(\mathbf{R}^{3})$ with a spherically symmetric potential which vanishes like $O(|x|^{-1})$ at infinity such that it possesses a positive eigenvalue embedded in the continuous spectrum. On the other hand, Kato in the famous work \cite{K} also showed that $-\Delta+V$ has no positive eigenvalues if the potential $V(x)$ decays fast enough at infinity (e.g. $ o(|x|^{-1})$). So the Wigner-von Neumann counterexample shows that Kato's result is sharp in essence. %$+1$ is an eigenvalue of the Schr\"{o}dinger operator $-\Delta+V$,
%where $V$ is radial as follows:
%	\begin{equation*}
%		V(x)=\big[1+w^{2}(|x|)\big]^{-2}(-32\sin|x|)\big[w^{3}(|x|)\cos|x|-3w^{2}(|x|)\sin^{3}|x|+w(|x|)\cos|x|+\sin^{3}|x|\big]
%	\end{equation*}
%	and the related eigenfunction $\psi$:
%	\begin{equation*}
%		\psi(x)=(|x|^{-1}\sin|x|)\big(1+w^{2}(|x|)\big)^{-1},
%	\end{equation*}
%	and here $w(|x|)=2|x|-\sin2|x|$.

	For higher order Schr\"{o}dinger type operator $H=(-\Delta)^m+V$ with  $m\geq2$, we will give some higher-order examples to show that there still exists some positive eigenvalue embedded in the continuous spectrum, even for $C_{0}^{\infty}$-potential. For even $m$, if $\phi\in L^{2}(\mathbf{R}^{n})$ is the eigenfunction of $H=(-\Delta)^m+V$ with eigenvalue $+1$, i.e. $\big[(-\Delta)^m+V\big]\phi=\phi$. If $\phi$ is strictly positive(or negative), then
	\begin{equation}\label{V}
		V(x)=\phi^{-1}(x)\big(\phi(x)-(-\Delta)^m\phi(x)\big).
	\end{equation}
	Our strategy is to find a $\phi(x)$ such that $\phi(x)=(-\Delta)^m\phi(x)$ for $|x|>r_{0}>0$, then by \eqref{V} the potential $V$ has compact support in $ B(0, r_{0})$.
\begin{lemma}
Let $\Phi(x)=\mathscr{F}^{-1}\big((1+|\cdot|^2)^{-1}\big)(x)$, then $\Phi\in C^{\infty}(\mathbf{R}^n\setminus\{0\})$. Moreover, $\Phi(x)$ is positive and $\Phi(x)=O(|x|^{-N})$ for any $N>0$ as $|x|\rightarrow\infty$.
\end{lemma}
Indeed, $\Phi(x)$ is the kernel of Bessel potential $(1-\Delta)^{-1}$. Thus, we have
\begin{equation}
\Phi(x)=\frac{1}{4\pi}\int_{0}^{\infty}e^{-|x|^2/t}e^{-t/4\pi}t^{-n/2}dt,\,\, x\in\mathbf{R}^n.
\end{equation}
In additional, if $n\geq3$, then
\begin{equation}
\Phi(x)=\frac{e^{-|x|}}{2(2\pi)^{(n-1)/2}\Gamma((n-1)/2)}\int_{0}^{\infty}e^{-|x|t}(t+t^2)^{(n-3)/2}dt,\,\, x\in\mathbf{R}^n.
\end{equation}
See Section V.3 in Stein's book \cite{Stein70} or \cite{Bessel-potential}. Note that $\Phi(x)=\Phi(r)$ where $r=|x|$. For $r>r_{0}>0$, by the above lemma, then $(-\Delta)^m\Phi=\Phi$ if $m$ is an even integer.
%Especially, we construct the eigenfunction $\phi$ in 3-dimension by the fundamental solution of $-\Delta$. Recall that the fundamental solution $E(r)$ of $-\Delta$ with $n\geq3$ as follows:
%	\begin{equation}
%		E(r)=\Big((n-2)|\mathbf{S}^{n-1}|\Big)^{-1}|x|^{2-n}, \, n\geq3.
%	\end{equation}
%	where $r=|x|$. In the spherical coordinates,  $$\Delta_{x}=\frac{d^{2}}{dr^{2}}+\frac{n-1}{r}\frac{d}{dr}+r^{-2}\Delta_{S},$$
%	where $\Delta_{S}$ is the Laplace-Beltrami on unit sphere. Since $E(r)$ is the fundamental solution, then $-\Delta E(r)=0$ for $r>r_{0}>0$ .
	%
%	Next, for different dimensional case, we give examples to show that $H=(-\Delta)^{m}+V$ exist positive embedded eigenvalue.
	%\begin{example}[$C_{0}^{\infty}$-potential ]\label{d=3exa}
%		In $n=3$,  for $r>r_{0}>0$, then
%		\begin{equation}
%			\Delta(r^{-1}e^{-r})=\Big(\frac{d^{2}}{dr^{2}}+\frac{2}{r}\frac{d}{dr}\Big)(r^{-1}e^{-r})=r^{-1}e^{-r}.
%		\end{equation}
	For any $\epsilon>0$,  define a radial function $u(r)$:
		\begin{equation}
			u(r)=\begin{cases}\Phi(r), \,\, & r>r_{0}+\epsilon;  \\ u_{0}(r), \,\, & 0\leq r\leq r_{0}-\epsilon;\end{cases}
		\end{equation}
 such that $0<u_{0}(r)\in C^{\infty}(\mathbf{R^+})$. Let the eigenfunction $\phi(x)=u(|x|)$,  and then $\phi\in L^{2}(\mathbf{R}^{3})$.  By \eqref{V},  we have $V(x)\in C_{0}^{\infty}(\mathbf{R}^{n})$ and $(\Delta^m+V)\phi=\phi$ for each even integer $m\ge 2$.

	\subsection{Proof of Theorem \ref{absence of positive eigenvalue}}
	In this part, we give the proof of Theorem \ref{absence of positive eigenvalue}. The proof relies on the virial identity for  homogeneous  operator $h(D)+V$,  where $h$ is a homogeneous function including $(-\Delta)^{s}$ for $s\in\mathbf{R}^{+}$. The virial identity is useful in the area of spectral analysis.
	
	\begin{lemma}\label{the virial theorem h}
		For operators $H=h(D)+V(x)$, $h$ is a homogeneous function with degree $\varrho$. $V(x)$ is a real-valued function on $\mathbf{R}^{n}$. Suppose that $V$ satisfies the following:
		
		(i) $V$ is $h(D)$-bounded with relative bound less than one;
		
		(ii) There exists a multiplication operator $\mathcal{V}$ on $L^{2}(\mathbf{R}^{n})$ with $\mathscr{D}(\mathcal{V})\supset \mathscr{D}(h(D))$,
		such that for all $\phi\in \mathscr{D}(h(D))$,
		\begin{equation}\label{V-limit-1}
			s-\lim_{\theta\rightarrow1}(\theta-1)^{-1}(V_{\theta}-V)\phi(x)=\mathcal{V}\phi(x),
		\end{equation}
		where $\mathscr{D}(h(D))$ and $\mathscr{D}(\mathcal{V})$ are the self-adjoint domain of $h(D)$ and $\mathcal{V}$ respectively.
		
		For any eigenfunction $\psi$ of $H$ with related eigenvalue $\lambda$, that is
		\begin{equation*}
			h(D)\psi+V\psi=\lambda\psi, \quad \psi\in\mathscr{D}(h(D)).
		\end{equation*}
		Then we have
		\begin{equation}\label{virial identity}
			\varrho\big(\psi, h(D)\psi\big)=\big(\psi,  \mathcal{V}\psi\big)=\varrho\big(\psi,  (\lambda-V)\psi\big).
		\end{equation}
	\end{lemma}
	
	\begin{remark}
		Notice that $\mathcal{V}$ is just $x\cdot\nabla V$ formally. In fact, if \eqref{V-limit-1} holds, $\mathcal{V}$ is given by $x\cdot\nabla V$ where the derivatives are interpreted in the sense of distributions. On the other hand, if there exists a function $\mathcal{V}(x)$ such that
		\begin{equation*}
			\lim_{\theta\rightarrow1}(\theta-1)^{-1}(V_{\theta}(x)-V(x))=\mathcal{V}(x)\,\, a. e. \,\, x\in\mathbf{R}^{n},
		\end{equation*}
		then \eqref{V-limit-1} holds.
	\end{remark}
	\begin{proof}
		For $\theta>0$,  $U(\theta)$ be the unitary family defined as follows:
		\begin{equation*}
			(U(\theta)\psi)(x)=\theta^{n/2}\psi(\theta x), \quad \psi\in L^{2}(\mathbf{R}^{n}).
		\end{equation*}
		Denote $V_{\theta}(x)=V(\theta x)$, then for $V(x)$  as a multiplication operator $V$, we have
		\begin{equation}\label{dilation identity V}
			V_{\theta}=U(\theta)VU^{-1}(\theta).
		\end{equation}

		Notice that for $h(D)$, we have
		\begin{equation}\label{dilation identity}
			U(\theta)h(D)U^{-1}(\theta)=\theta^{-\varrho}h(D)=h\big(U(\theta)DU^{-1}(\theta)\big).
		\end{equation}
		Indeed, for $-i\frac{\partial }{\partial x_{j}}$, $j=1,2,\cdots,d$, we have
		\begin{equation*}
				U(\theta)(-i\frac{\partial }{\partial x_{j}})U^{-1}(\theta)f(x) = \theta^{-1}(-i\frac{\partial }{\partial x_{j}})f(x).
\end{equation*}
		Thus \eqref{dilation identity} holds by the homogeneous of $h$.
		\par
		
		Since $\psi$ is an eigenfunction of $H=h(D)+V$ with related eigenvalue $\lambda$, by the identities \eqref{dilation identity V} and \eqref{dilation identity}, we have
		\begin{equation*}
			\lambda\psi_{\theta}=U(\theta)\lambda\psi U^{-1}(\theta)=U(\theta)\big(h(D)\psi+V\psi\big)U^{-1}(\theta)=\big(\theta^{-\varrho}h(D)+V_{\theta}\big)\psi_{\theta},
		\end{equation*}
		where $\psi_{\theta}(x)=\psi(\theta x)$. Thus we get
		\begin{equation}\label{theta-eigen}
			\big(h(D)+\theta^{\varrho}V_{\theta}\big)\psi_{\theta}=\theta^{\varrho}\lambda\psi_{\theta}.
		\end{equation}
		By the above equality \eqref{theta-eigen} and $H\psi=\lambda\psi$, we have
		\begin{equation*}
			\begin{split}
				\lambda(\theta^{\varrho}-1)(\psi_{\theta}, \psi) & = \big((h(D)+\theta^{\varrho}V_{\theta})\psi_{\theta}, \psi\big)-\big(\psi_{\theta}, (h(D)+V)\psi\big) \\
				& = \theta^{\varrho}(V_{\theta}\psi_{\theta}, \psi)-(\psi_{\theta}, V\psi)\\
				& = (\theta^{\varrho}-1)(V_{\theta}\psi_{\theta}, \psi)+\big(\psi_{\theta}, (V_{\theta}-V)\psi\big).
			\end{split}
		\end{equation*}
		since $V(x)$ is real valued. Thus
		\begin{equation*}
			\frac{\theta^{\varrho}-1}{\theta-1}(V_{\theta}\psi_{\theta}, \psi)+\frac{1}{\theta-1}\big(\psi_{\theta}, (V_{\theta}-V)\psi\big)=\lambda\frac{\theta^{\varrho}-1}{\theta-1}(\psi_{\theta}, \psi).
		\end{equation*}
		Taking the limit as $\theta\rightarrow1$ yields \eqref{virial identity}.
	\end{proof}
	
	\begin{proof}[\bf Proof of Theorem \ref{absence of positive eigenvalue}]
		Since $V$ satisfies the hypotheses of Proposition \ref{the virial theorem h}, then for any eigenfunction $\psi$ of $H_{h}=h(D)+V$,  we have the following virial identity:
		\begin{equation}\label{short identity}
			\big(\psi,\,\, h(D)\psi\big)=\frac{1}{\varrho}\big(\psi,\,\,  \mathcal{V}\psi\big).
		\end{equation}
		Note that the positivity of $h(D)$ implies that ${\rm{Ker}}(h(D))=\{0\}$.
		\par In case (1), $V(\gamma x)<V(x)$ with $\gamma>1$ implies that $\mathcal{V}(x)\leq0$. Hence the virial identity \eqref{short identity} holds only if $\psi=0$.
		\par In case (2), $\mathcal{V}=-\nu V$, so by \eqref{virial identity}, we know
		\begin{equation*}
			\begin{split}
				\lambda(\psi, \psi)& = (\varrho)^{-1}(\varrho-\nu)(\psi, V\psi)=-(\nu\varrho)^{-1}(\tau-\nu)(\psi, \mathcal{V}\psi)\\
				& = -\nu^{-1}(\varrho-\nu)(\psi, h(D)\psi).
			\end{split}
		\end{equation*}
		By the positivity of $h(D)$ and $\varrho-\nu>0$,  we conclude that $\lambda<0$ if $\psi\neq0$.
		\par In case (3), by \eqref{short identity} we have
		\begin{equation*}
			\begin{split}
				-a\lambda(\psi, \psi) & =-a\big(\psi,\,\, (h(D)+V)\psi\big)+(a+1)\big(\psi,\,\, (h(D)-\varrho^{-1}\mathcal{V})\psi\big) \\
				& =\big(\psi,\,\, \big(h(D)-\tau^{-1}(1+a)\mathcal{V}-aV\big)\psi\big).
			\end{split}
		\end{equation*}
		Thus \eqref{V condition 3} implies that $\lambda\leq0$ if $\psi\neq0$.
	\end{proof}

	\section{Proof of asymptotic expansions and identification of  resonance spaces}\label{proof}
	
	By the symmetric resolvent identity \eqref{symmetric-resolvent-idnetity}, we need to derive the asymptotic expansions of $M^{-1}(\mu)$ in $L^{2}(\mathbf{R}^{n})$ for $\mu$ near zero. In this section, we show the processes of deriving the expansion of $M^{-1}(\mu)$ and identify the resonance spaces case by case.
	
	Here we give the needed lemmas in the following.
	\begin{lemma}\label{Feshbach-formula}
		Let $A$ be a closed operator and $S$ a projection. Suppose $A+S$ has a bounded inverse, then $A$ has a bounded inverse if and only if
		\begin{equation}
			B\equiv S-S(A+S)^{-1}S
		\end{equation}
		has a bounded inverse in $S\mathcal{H}$. Furthermore,
		\begin{equation}
			A^{-1}=(A+S)^{-1}+(A+S)^{-1}SB^{-1}S(A+S)^{-1}.
		\end{equation}
	\end{lemma}
	
	\begin{lemma}\label{inverse formula}
		(\cite[Proposition 1]{JN2}) Let $F\subset\mathbf{C}$ have zero as an accumulation point. Let $T(z), z\in F$ be a family of bounded operators of the form
		\begin{equation*}
			T(z)=T_{0}+zT_{1}(z)
		\end{equation*}
		with $T_{1}(z)$ uniformly bounded as $z\rightarrow0$. Suppose $0$ is an isolated point of the spectrum of $T_{0}$, and let $S$ be the corresponding Riesz projection. If $T_{0}S=0$, then for sufficient small $z\in F$ the operator $\widetilde T(z):S\mathcal{H}\rightarrow S\mathcal{H}$ defined by
		\begin{equation}\label{equ:61}
			\widetilde{T}(z)=\frac{1}{z}\big(S-S(T(z)+S)^{-1}S\big)=\sum_{j=0}^{\infty}(-1)^{j}z^{j}S\big[T_{1}(z)(T_{0}+S)^{-1}\big]^{j+1}S
		\end{equation}
		is uniformly bounded as $z\rightarrow0$. The operator $T(z)$ has a bounded inverse in $\mathcal{H}$ if and only if $\widetilde{T}(z)$ has a bounded inverse in $S\mathcal{H}$, and in this case
		\begin{equation}\label{equ:62}
			T(z)^{-1}=\big(T(z)+S\big)^{-1}+\frac{1}{z}\big(T(z)+S\big)^{-1}S\widetilde{T}(z)^{-1}S\big(T(z)+S\big)^{-1}.
		\end{equation}
	\end{lemma}

	\begin{lemma}\label{Riesz-potential-boundedness}
		(\cite[Lemma 2.3]{J}) Denote the Riesz potential $I_{\alpha}=(-\Delta)^{-\alpha/2}$ on $\mathbf{R}^{n}$ with $0<\alpha<n$.
		\begin{enumerate}
			\item If  $0<\alpha<n/2$, $s, s'\geq0$ and $s+s'\geq\alpha$, then
			$I_{\alpha}\in B(s, -s').$
			\item If  $n/2\leq\alpha<n$, $s, s'>\alpha-n/2$ and $s+s'\geq\alpha$, then
			$I_{\alpha}\in B(s, -s').$
		\end{enumerate}
	\end{lemma}

	In the following, we denote $D_{j}=\big(T_{j}+S_{j+1}\big)^{-1}$ where $T_{j}$ and $S_{j}$ see Definition \ref{resonance}.
	
	\subsection{Proof of resolvent asymptotic expansions} In this subsection, we give the proof of the resolvent asymptotic expansions at zero-resonance case by case.
	
	{\bf Proof of Theorem \ref{RV-4m} ($n>4m$).}
	In this case, for $|V(x)|\lesssim(1+|x|)^{-\beta}$ with some $\beta>n+4$,  we have the following expansions for $M(\mu)$ in $B(0,0)$:
	\begin{equation}\label{8.5}
		M(\mu)=U+vG_0v+\sum_{j=1}^{\aleph-1}\mu^{2mj}vG_jv+\mu^{n-2m}vG_\aleph v+vE_0(\mu)v.
	\end{equation}
	Denote
	\[\mathcal{M}(\mu)=\sum_{j=1}^{\aleph-1}\mu^{2mj}vG_jv+\mu^{n-2m}vG_\aleph v+vE_0(\mu)v.\]
	
	Recall that $\aleph=[\frac{n}{2m}]$. Theorem \ref{RV-4m} holds by symmetric resolvent identity \eqref{symmetric-resolvent-idnetity} and the following Proposition.
	
	\begin{proposition}\label{4m}
		Assume that $|V(x)|\lesssim(1+|x|)^{-\beta}$ with some $\beta>n+4$. For $n>4m$ and $0<|\mu|\ll1,$ we obtain the following expansions of $\big(M(\mu)\big)^{-1}$ in $B(0, 0):$
		
		(i) If zero is a regular point  of $H,$ then we have
		\begin{equation}\label{M-4m-regular}
			\begin{split}
				\Big(M(\mu)\Big)^{-1}=& T_0^{-1}+\sum_{l=1}^{\aleph-1}\mu^{2ml}\widetilde{A}_l+\tilde{\alpha}_{\aleph}\mu^{n-2m}\widetilde{A}_{\aleph}+O(\mu^{n-2m+2})
			\end{split}
		\end{equation}
		where the operators $\widetilde{A}_l\in B(0,0)$ and $\tilde{\alpha}_{\aleph}\in\mathbf{C}\setminus\mathbf{R}$.
		
		(ii) If zero is an eigenvalue of $H$, then we have
		\begin{equation}\label{M-4m-eigenvalue}
			\Big(M(\mu)\Big)^{-1}=\frac{\widetilde{A}_1^e}{\mu^{2m}}+\sum_{l=2}^{\aleph-1}\mu^{2m(l-2)}\widetilde{A}_j^e+\tilde{\alpha}_{\aleph}^e\mu^{n-6m}\widetilde{A}_\aleph^e
			+O(\mu^{n-6m+2})
		\end{equation}
		where the operators $\widetilde{A}_l^e\in B(0,0)$ and $\tilde{\alpha}_{\aleph}^e\in\mathbf{C}\setminus\mathbf{R}$. Furthermore, $\widetilde{A}_1^e=S_{1}\big(S_1vG_1vS_1\big)^{-1}S_{1}$.
	\end{proposition}
	\begin{proof}
		(i) Since
		\[M(\mu)=T_0+\mathcal{M}(\mu),\]
		and $T_0$ is invertible, then $\mathcal{M}(\mu)T_0^{-1}$ is uniformly bounded with bound less than 1 for tiny $|\mu|.$ Thus
		\begin{equation}\label{4m-M-inverse}
			\begin{split}
				\Big(M(\mu)\Big)^{-1}=&T_0^{-1}\big(I+\mathcal{M}(\mu)T_0^{-1}\big)^{-1}\\
				=&T_{0}^{-1}-T_{0}^{-1}\mathcal{M}(\mu)T_{0}^{-1}+T_{0}^{-1}\Big(\mathcal{M}(\mu)T_{0}^{-1}\Big)^{2}\big(I+\mathcal{M}(\mu)T_{0}^{-1}\big)^{-1}.
			\end{split}
		\end{equation}
		
		(ii) If $T_0$ is not invertible, then we know that zero is an eigenvalue of $H$. 	Applying Lemma \ref{Feshbach-formula} to $\big(M(\mu)\big)^{-1}$, we have
		\begin{equation}\label{M-inverse}
		\Big(M(\mu)\Big)^{-1}=\Big(M(\mu)+S_1\Big)^{-1}+\Big(M(\mu)+S_1\Big)^{-1}S_1\Big(M_1(\mu)\Big)^{-1}S_1\Big(M(\mu)+S_1\Big)^{-1}
		\end{equation}
	where $M_1(\mu)=S_1-S_1(M(\mu)+S_1)^{-1}S_1$.	By \eqref{M-4m}, we have
		\begin{equation}\label{8.9}
			\begin{split}
				\Big(M(\mu)+S_1\Big)^{-1}
				=&\big(T_0+S_1+\mathcal{M}(\mu)\big)^{-1}=D_0\big(I+\mathcal{M}(\mu)D_0\big)^{-1}\\
				=&D_0-D_0\mathcal{M}(\mu)D_0+D_0\big(\mathcal{M}(\mu)D_0\big)^2+O(\mu^{6m}).
			\end{split}
		\end{equation}
		Thus for $M_1(\mu)$, we have
		\begin{equation}
			\begin{split}
				M_1(\mu)=&S_1-S_1\Big(D_0-D_0\mathcal{M}(\mu)D_0+D_0\big(\mathcal{M}(\mu)D_0\big)^2+O(\mu^{6m})\Big)S_1\\
				=&S_1\mathcal{M}(\mu)S_1-S_1\big(\mathcal{M}(\mu)D_0\big)^2S_1+O(\mu^{6m})\\
				:=&\mu^{2m}S_1vG_1vS_1+\mathcal{M}_1(\mu)
			\end{split}
		\end{equation}
	where $\mathcal{M}_{1}(\mu)=O(\mu^{4m})$.	Since $S_1vG_1vS_1$ is invertible on $S_1L^2(\mathbf{R}^n)$, see Lemma \ref{4m-S1-kernel}. Then we can obtain the asymptotic expansions of $\big(M_1(\mu)\big)^{-1}$ by expanding $\big(\mu^{-2m}M_1(\mu)\big)^{-1}=\big(S_1vG_1vS_1+\mu^{-2m}\mathcal{M}_1(\mu)\big)^{-1}$ into  Neumann series.

		Substituting the expansions \eqref{8.9} of $\big(M(\mu)+S_1\big)^{-1}$ and the Neumann series of  $\big(M_1(\mu)\big)^{-1}$ into \eqref{M-inverse}, we obtain the expansion \eqref{M-4m-eigenvalue}.
	\end{proof}
	\begin{lemma}\label{4m-S1-kernel}
		Assume that $|V(x)|\lesssim(1+|x|)^{-\beta}$ with some $\beta>n+4$. Then we obtain
		\[\ker\big(S_1vG_1vS_1\big)=\{0\}.\]
	\end{lemma}
	\begin{proof}
		Take $\phi\in S_1L^2(\mathbf{R}^n)$ with $S_1vG_1vS_1\phi=0.$ Then using \eqref{4m-free}, we have
		$$G_1=\displaystyle\lim_{\mu\rightarrow0}\frac{R_0(\mu^{2m})-G_0}{\mu^{2m}},$$
		thus
		\begin{align*}
			0=&\langle S_1vG_1vS_1\phi,\phi\rangle=\langle G_1v\phi,v\phi\rangle=\lim_{\mu\rightarrow 0}\Bigg\langle\Big(\frac{R_0(\mu^{2m})-G_0}{\mu^{2m}}\Big)v\phi,v\phi\Bigg\rangle\\
			=&\lim_{\mu\rightarrow 0}\frac{1}{\mu^{2m}}\int_{\mathbf{R}^n}\Big(\frac{1}{|\xi|^{2m}-\mu^{2m}}-\frac{1}{|\xi|^2}\Big)\widehat{v\phi}(\xi)\overline{\widehat{v\phi}}(\xi)d\xi\\
			=&\lim_{\mu\rightarrow0}\int_{\mathbf{R}^n}\frac{|\widehat{v\phi}(\xi)|^2}{|\xi|^{2m}(|\xi|^{2m}-\mu^{2m})}d\xi=\int_{\mathbf{R}^n}\frac{|\widehat{v\phi}(\xi)|^2}{|\xi|^{4m}}d\xi
			=\langle G_0v\phi, G_0v\phi\rangle.
		\end{align*}
		Here we used the dominated convergence theorem as $\mu\rightarrow 0$ with chosen $\mu$ such that $\rm{Re}(\mu^{2m})<0$ on the last equality. This implies that $\widehat{v\phi}=0$ and $v\phi=0.$ Recall that $S_{k+1}\leq S_{1}$, then $\phi\in S_{1}L^2$ which implies that $\phi=-UvG_{0}v\phi$. Thus the kernel of $S_1vG_1vS_1$ is trivial.
	\end{proof}
	
	{\bf Proof of Theorem \ref{RV-expansions-2m-odd} ($2m<n\le 4m$ and odd).}
	In this case, for $|V(x)|\lesssim (1+|x|)^{-\beta}$ with some $\beta>0$ , we have the following expansions for $M(\mu)$ in $B(0,0)$:
	\begin{equation}
		M(\mu)=U+vG_0v+\tilde{c}_1\mu^{2(m-k)+1}P+\sum_{j=2}^kc_j\mu^{2(m+j-k)-1}vG_jv+\mu^{2m}vG_{k+1}v+vE_3(\mu)v.
	\end{equation}
	Denote
	\begin{align*}
		\mathcal{M}(\mu)=\tilde{c}_1\mu^{2(m-k)+1}P+\sum_{j=2}^kc_j\mu^{2(m+j-k)-1}vG_jv+\mu^{2m}vG_{k+1}v+vE_3(\mu)v.
	\end{align*}
	
	Theorem \ref{RV-expansions-2m-odd} holds by symmetric resolvent identity \eqref{symmetric-resolvent-idnetity} and the following Proposition.
	
	\begin{proposition}\label{M-2m odd}
		For $n=4m+1-2k$ with $k$ chosen as follows and $0<|\mu|\ll 1$.  Assume that $|V(x)|\lesssim (1+|x|)^{-\beta}$ with some $\beta>n+4k$.  We obtain the following expansions of $\big(M(\mu)\big)^{-1}$ in $B(0, 0):$
		
		(i) For $1\le k\le m$, if zero is a regular  point of $H$, then  we have
		\begin{equation}\label{M-2m-odd-regular}
			\Big(M(\mu)\Big)^{-1}=\widetilde{B}_0+\sum_{l=1}^k\tilde{\beta}_l\mu^{2(m+l-k)-1}\widetilde{B}_l+\mu^{2m}\widetilde{B}_{k+1}+O(\mu^{2m+})
		\end{equation}
		where the operators $\widetilde{B}_l\in B(s, -s')$ with $s,s'>\frac{n}{2}+2k$ and $\tilde{\beta}_l\in\mathbf{C}\setminus\mathbf{R}$. Furthermore, $\widetilde{B}_0=T_0^{-1}$.
		
		(ii) For $1\le k\le [\frac{m}{2}]+1$, if zero is of the  $j$-th kind of resonance of $H$ with $1\le j\le k$, then  we have
		\begin{equation}\label{M-2m-odd-j}
			\begin{split}
				\Big(M(\mu)\Big)^{-1}=&\frac{\beta_0^jB_0^j}{\mu^{2(m-k+j)-1}}+\sum_{l=1}^{k-j}\frac{\beta_l^jB_l^j}{\mu^{2(m-k+j-l)-1}}
				+\sum_{l=k-j+1}^{2m-3k+3j-2}\frac{\beta_l^jB_l^j}{\mu^{2m-3k+3j-l-1}}\\
				&\ \ \ \tilde{\beta}_{2m-3k+3j-1}^j\widetilde{B}_{2m-3k+3j-1}^j+O(\mu),
			\end{split}
		\end{equation}
		where the operators $\widetilde{B}_l^j\in B(s,-s')$ with $s,s'>\frac{n}{2}+2k$ and $\tilde{\beta}_l^j\in\mathbf{C}\setminus\mathbf{R}$. Furthermore, $\widetilde{B}_0^j=S_{j}\big(S_jvG_jvS_j\big)^{-1}S_{j}$ and all $\widetilde{B}_l^j$ for $0\leq l\leq2m-3k+3j-2$ are finite rank.
		
		(iii) For $1\le k\le[\frac{m}{2}]+1$, if zero is an eigenvalue of $H$, then  we have
		\begin{equation}\label{M-2m-odd-k+1}
			\begin{split}
				\Big(M(\mu)\Big)^{-1}=&\frac{\widetilde{B}_0^{k+1}}{\mu^{2m}}+\sum_{l=1}^{2m-1}\frac{\tilde{\beta}_{l}^{k+1}\widetilde{B}_l^{k+1}}{\mu^{2m-l}}
				+\tilde{\beta}_{2m}^{k+1}\widetilde{B}_{2m}^{k+1}+O(\mu),
			\end{split}
		\end{equation}
		where the operators $\widetilde{B}_l^{k+1}\in B(s,-s')$ with $s,s'>\frac{n}{2}+2k$ and $\tilde{\beta}_l^{k+1}\in\mathbf{C}\setminus\mathbf{R}$. Furthermore, $\widetilde{B}_{0}^{k+1}=S_{k+1}\big(S_{k+1}vG_{k+1}vS_{k+1}\big)^{-1}S_{k+1}$ and all $\widetilde{B}_l^{k+1}$ for $0\leq l\leq2m-1$ are finite rank.
	\end{proposition}
	
	\begin{proof}
		(i) The proof of \eqref{M-2m-odd-regular} follows from the proof of \eqref{M-4m-regular}.
		
		(ii)Using \eqref{M-2m-odd}, we have
		\begin{equation}
			\begin{split}
				\Big(M(\mu)+S_1\Big)^{-1}
				=&\big(T_0+S_1+\mathcal{M}(\mu)\big)^{-1}\\
				=&D_0\big(I+\mathcal{M}(\mu)D_0\big)^{-1}\\
				=&D_0-D_0\mathcal{M}(\mu)D_0+D_0\big(\mathcal{M}(\mu)D_0\big)^2+O\big(\mu^{4(m-k)+2+}\big)
			\end{split}
		\end{equation}
		Then for $M_1(\mu)=S_1-S_1\big(M(\mu)+S_1\big)S_1,$ since $S_1D_0=D_0S_1=S_1,$ we have
		\begin{equation}\label{M1-2m-odd}
			\begin{split}
				M_1(\mu)=&S_1-S_1\Big(D_0-D_0\mathcal{M}(\mu)D_0+D_0\big(\mathcal{M}(\mu)D_0\big)^2+O(\mu^{4(m-k)+2+})\Big)S_1\\
				=&S_1\mathcal{M}(\mu)S_1-S_1\big(\mathcal{M}(\mu)D_0\big)^{2}S_1+O(\mu^{4(m-k)+2+})\\
				=&\tilde{c}_1\mu^{2(m-k)+1}S_1PS_1+\sum_{l=2}^kc_l\mu^{2(m+l-k)-1}S_1vG_lvS_1\\
				&\ \ +\mu^{2m}S_1vG_{k+1}vS_1-\tilde{c}_1^2\mu^{4(m-k)+2}S_1D_0PD_0PD_0S_1+O(\mu^{4(m-k)+2+})\\
				:=&\tilde{c}_1\mu^{2(m-k)+1}S_1PS_1+\mathcal{M}_1(\mu)
			\end{split}
		\end{equation}
		
		By the definition of the first kind of resonance, the operator $T_1=S_1PS_1$ is invertible on $S_1L^2(\mathbf{R}^n),$ then we can obtain the expansions of $\big(M_1(\mu)\big)^{-1}$ by Neumann series.
		
		From Lemma \ref{Feshbach-formula}, we also have \eqref{M-inverse}. Substituting the expansions of $\big(M(\mu)+S_1\big)^{-1}$ and $\big(M_1(\mu)\big)^{-1}$ into \eqref{M-inverse}, we obtain \eqref{M-2m-odd-j} with $j=1$.
		
		 For the second kind of resonance, we aim to derive the expansions of $\big(M_1(\mu)\big)^{-1}.$ Define $\widetilde{M}_1(\mu)=\frac{M_1(\mu)}{\tilde{c}_1\mu^{2(m-k)+1}},$ using \eqref{M1-2m-odd} we have
		\begin{equation}
			\begin{split}
				\widetilde{M}_1(\mu)=& S_1PS_1+\sum_{l=2}^{k}\tilde{c}_l\mu^{2l-2}S_1vG_lvS_1-\tilde{c}_1\mu^{2(m-k)+1}S_1PD_0PS_1\\
				&\ \ +\tilde{c}_{k+1}\mu^{2k-1}S_1vG_{k+1}vS_1+O(\mu^{2k-1+})\\
				:=&T_1+\widetilde{\mathcal{M}}_1(\mu).
			\end{split}
		\end{equation}
		By the definition of the second kind of resonance, we know that the operator $T_1$ is not invertible on $S_1L^2(\mathbf{R}^n).$ Since $S_2$ is the Riesz projection onto the kernel of $T_1$ on $S_1L^2(\mathbf{R}^n),$ so $T_1+S_2$ is invertible on $S_1L^2(\mathbf{R}^n).$ Furthermore, we can obtain
		\begin{equation}
			\begin{split}
				\Big(\widetilde{M}_1(\mu)+S_2\Big)^{-1}
				=&\big(T_1+S_2+\widetilde{\mathcal{M}}_1(\mu)\big)^{-1}\\
				=&D_1-D_1\widetilde{\mathcal{M}}_1(\mu)D_1+D_1\big(\widetilde{\mathcal{M}}_1(\mu)D_1\big)^2+O(\mu^{4+})
			\end{split}
		\end{equation}
		Then for $M_2(\mu)=S_2-S_2\big(\widetilde{M}_1(\mu)+S_2\big)^{-1}S_2,$ since $S_2D_1=D_1S_2=S_2$ and $S_2P=PS_2=0,$ we have
		\begin{equation}\label{M2-odd}
			\begin{split}
				M_2(\mu)
				=&S_2-S_2\Big(D_1-D_1\widetilde{\mathcal{M}}_1(\mu)D_1+D_1\big(\widetilde{\mathcal{M}}_1(\mu)D_1\big)^2+O(\mu^{4+})\Big)S_2\\
				=&S_2\widetilde{\mathcal{M}}_1(\mu)S_2-S_2\big(\widetilde{\mathcal{M}}_1(\mu)D_1\big)^2S_2+O(\mu^{4+})\\
				=&\tilde{c}_2\mu^2S_2vG_2v+\sum_{l=3}^k\tilde{c}_l\mu^{2l-2}S_2vG_lvS_2+\tilde{c}_{k+1}\mu^{2k-1}S_2vG_{k+1}vS_2+O(\mu^{2k-1})\\
				:=&\tilde{c}_2\mu^2S_2vG_2vS_2+\mathcal{M}_2(\mu).
			\end{split}
		\end{equation}
		Since the operator $S_2vG_2vS_2$ is invertible on $S_2L^2(\mathbf{R}^n),$ then we can obtain the expansions of $\big(M_2(\mu)\big)^{-1}$ by Neumann series. From Lemma \ref{Feshbach-formula}, we have
		\begin{equation}\label{M1-inverse}
			\Big(\widetilde{M}_1(\mu)\Big)^{-1}=\Big(\widetilde{M}_1(\mu)+S_2\Big)^{-1}+\Big(\widetilde{M}_1(\mu)+S_2\Big)^{-1}S_2\Big(M_2(\mu)\Big)^{-1}\Big(\widetilde{M}_1(\mu)+S_2\Big)^{-1}
		\end{equation}
		Substituting the expansions of $\big(\widetilde{M}_1(\mu)+S_2\big)^{-1}$ and $\big(M_2(\mu)\big)^{-1}$ into \eqref{M1-inverse}, we obtain the inverse of $\widetilde{M}_1(\mu)$, then we can obtain the expansions of $\big(M_1(\mu)\big)^{-1}.$ Furthermore, we can obtain \eqref{M-2m-odd-j} with $j=2$.
		
		By an induction process we can get the expansions of \eqref{M-2m-odd-j} and \eqref{M-2m-odd-k+1}.
	\end{proof}
	
	Now, we show that the kernel of the operator $S_{k+1}vG_{k+1}vS_{k+1}$ on $S_{k+1}L^2(\mathbf{R}^n)$ is trivial which means the iterated process stop here.
	\begin{lemma}\label{2m-odd-kernel}
		Assume that $|V(x)|\lesssim(1+|x|)^{-\beta}$ with some $\beta>n+4k$.  Then we have
		\[\ker(S_{k+1}vG_{k+1}vS_{k+1})=\{0\}.\]
	\end{lemma}
	\begin{proof}
		Take $\phi\in S_{k+1}L^2(\mathbf{R}^n)$ with $S_{k+1}vG_{k+1}vS_{k+1}\phi=0.$ Then we have
		\[\langle S_{k+1}vG_{k+1}vS_{k+1}\phi,\phi\rangle=\langle G_{k+1}v\phi,v\phi\rangle=0.\]
		By the definition of $S_j$ with $1\le j\le k+1,$ we also have
		\[\langle v,\phi\rangle=\langle G_jv\phi,v\phi\rangle=0.\]
		Then using \eqref{2m-odd-free}, we obtain
		\[G_{k+1}=\lim_{\mu\rightarrow0}\frac{R_0(\mu^{2m})-G_0-\sum_{j=1}^kc_j\mu^{2(m-k+j)-1}G_j}{\mu^{2m}}.\]
		Thus, we have
		\begin{align*}
			0=&\langle S_{k+1}vG_{k+1}vS_{k+1}\phi,\phi\rangle=\langle G_{k+1}v\phi,v\phi\rangle\\
			=&\lim_{\mu\rightarrow0}\Bigg\langle\Bigg(\frac{R_0(\mu^{2m})-G_0-\sum_{j=1}^kc_j\mu^{2(m-k+j)-1}G_j}{\mu^{2m}}\Bigg)v\phi,v\phi\Bigg\rangle\\
			=&\lim_{\mu\rightarrow0}\frac{1}{\mu^{2m}}\Big\langle\big(R_0(\mu^{2m})-G_0\big)v\phi,v\phi\Big\rangle\\
			=&\lim_{\mu\rightarrow0}\frac{1}{\mu^{2m}}\int_{\mathbf{R}^n}\Big(\frac{1}{|\xi|^{2m}-\mu^{2m}}-\frac{1}{|\xi|^{2m}}\Big)\widehat{v\phi}(\xi)\overline{\widehat{v\phi}}(\xi)d\xi\\
			=&\lim_{\mu\rightarrow0}\int_{\mathbf{R}^n}\frac{|\widehat{v\phi}(\xi)|^2}{|\xi|^{2m}(|\xi|^{2m}-\mu^{2m})}d\xi=\int_{\mathbf{R}^n}\frac{|\widehat{v\phi}(\xi)|^2}{|\xi|^{4m}}d\xi
			=\langle G_0v\phi, G_0v\phi\rangle.
		\end{align*}
		Here we used the dominated convergence theorem as $\mu\rightarrow 0$ with chosen $\mu$ such that $\rm{Re}(\mu^{2m})<0$ on the last equality. This implies that $\widehat{v\phi}=0$ and $v\phi=0$. Recall that $S_{k+1}\leq S_{1}$, then $\phi\in S_{1}L^2$ which implies that $\phi=-UvG_{0}v\phi$. Thus the kernel of $S_{k+1}vG_{k+1}vS_{k+1}$ is trivial.
	\end{proof}

	{\bf Proof of Theorem \ref{RV-expansions-2m-even} ($2m<n\le 4m$ and even).}
	In this case, for $|v(x)|\lesssim(1+|x|)^{-\beta}$ with some $\beta>0$, we have the following expansions for $M(\mu)$ in $B(0, 0):$
	\begin{equation}
		\begin{split}
			M(\mu)=&U+vG_0v+\tilde{d}_1\mu^{2(m-k+1)}P+\sum_{j=2}^{k-1}d_j\mu^{2(m+j-k)}vG_jv\\
			&+\mu^{2m}g(\mu)vG_kv
			+\mu^{2m}vG_{k+1}v+vE_4(\mu)v.
		\end{split}
	\end{equation}
	Denote
	\[\mathcal{M}(\mu)=\tilde{d}_1\mu^{2(m-k+1)}P+\sum_{j=2}^{k-1}d_j\mu^{2(m+j-k)}vG_jv+\mu^{2m}g(\mu)vG_kv
	+\mu^{2m}vG_{k+1}v+vE_4(\mu)v.\]

	Theorem \ref{RV-expansions-2m-even} holds by the following Proposition.
	\begin{proposition}\label{M-expansions-2m-even}
		For $n=4m+2-2k$ with $k$ chosen as follows and $0<|\mu|\ll1$. Assume $|V(x)|\lesssim (1+|x|)^{-\beta}$ with some $\beta>n+4k$.  we obtain the following expansions of $R_V(\mu^{2m})$ in $B(s, -s')$:
		
		(i) For $1\le k\le m$, if zero is a regular  point of $H$, then  we have
		\begin{equation}\label{M-2m-even-regular}
			\Big(M(\mu)\Big)^{-1}=\widetilde{C}_0+\sum_{l=1}^{k-1}\mu^{2(m+l-k)}\tilde{\tau}_l\widetilde{C}_l+\mu^{2m}g(\mu)\widetilde{C}_k+\mu^{2m}\widetilde{C}_{k+1}+O(\mu^{2m+})
		\end{equation}
		where the operators $\widetilde{C}_l\in B(s, -s')$ with $s,s'>\frac{n}{2}+2k$ and $\tilde{\tau}_l\in\mathbf{C}\setminus\mathbf{R}$. Furthermore, $\widetilde{C}_0=T_0^{-1}.$
		
		(ii) For $1\le k\le [\frac{m}{2}]+1$, if zero is of the $j-$th kind of resonance of $H$ with $1\le j\le k-1$, then  we have
		\begin{equation}\label{M-2m-even-j}
			\begin{split}
				\Big(M(\mu)\Big)^{-1}=&\frac{\tilde{\tau}_{0,0}^j\widetilde{C}_{0,0}^j}{\mu^{2(m-k+j)}}+\sum_{l=1}^{k-j}\frac{\tilde{\tau}_{l,0}^j\widetilde{C}_{l,0}^j}{\mu^{2(m-k+j-l)}}+
				\sum_{l=k-j+1}^{m-k+j-1}\Bigg[\frac{\tilde{\tau}_{l,0}^j\widetilde{C}_{l,0}^j}{\mu^{2(m-k+j-l)}}+\frac{\ln(\mu)\tilde{\tau}_{l,1}^j\widetilde{C}_{l,1}^j}{\mu^{2(m-k+j-l)}}\Bigg]\\
				&\ \ \ +\ln(\mu)\tilde{\tau}_{m-k+j,1}^j\widetilde{C}_{m-k+j,1}^j+\tilde{\tau}_{m-k+j,0}^j\widetilde{C}_{m-k+j,0}^j+O(\mu^{0+}),
			\end{split}
		\end{equation}
		where the operators $\widetilde{C}_{l,0}^j,\widetilde{C}_{l,1}^j\in B(s,-s')$ with $s,s'>\frac{n}{2}+2k$ and $\tilde{\tau}_{l,0}^j,\tilde{\tau}_{l,1}^j\in\mathbf{C}\setminus\mathbf{R}$. Furthermore, $\widetilde{C}_{0,0}^j=S_{j}\big(S_jvG_jvS_j\big)^{-1}S_{j}$ and all $\widetilde{C}_{l, 0}^j,~ \widetilde{C}_{l, 1}^j$ for $0\leq l\leq m-k+j-1$ are finite rank.
		
		(iii) For $1\le k\le [\frac{m}{2}]+1$, if zero is of the $k-$th kind of resonance of $H$, then  we have
		\begin{equation}\label{M-2m-even-k}
			\begin{split}
				\Big(M(\mu)\Big)^{-1}=&\frac{\tilde{\tau}_{0,1}^k\widetilde{C}_{0,1}^k}{\mu^{2m}g(\mu)}+\frac{\tilde{\tau}_{0,2}^k \widetilde{C}_{0,2}^k}{\mu^{2m}\big(g(\mu)\big)^2}+\sum_{l=1}^{m-1}\Bigg[\frac{\tilde{\tau}_{l,0}^k\widetilde{C}_{l,0}^k}{\mu^{2(m-l)}}
				+\frac{\tilde{\tau}_{l,1}^k\widetilde{C}_{l,1}^k}{\mu^{2(m-l)}g(\mu)}
				+\frac{\tilde{\tau}_{l,2}^k\widetilde{C}_{l,2}^k}{\mu^{2(m-l)}\big(g(\mu)\big)^2}\Bigg]\\
				&\ \ \ +\tilde{\tau}_{m,0}^k\widetilde{C}_{m,0}^k+\big(g(\mu)\big)^{-1}\tilde{\tau}_{m,1}^k\widetilde{C}_{m,1}^k+\big(g(\mu)\big)^{-2}\tilde{\tau}_{m,2}^k \widetilde{C}_{m,2}^k+O(\mu^{0+}),
			\end{split}
		\end{equation}
		where the operators $\widetilde{C}_{l,0}^k, \widetilde{C}_{l,1}^k, \widetilde{C}_{l,2}^k\in B(s,-s')$ with $s,s'>\frac{n}{2}+2k$ and $\tilde{\tau}_{l,0}^k, \tilde{\tau}_{l,1}^k,\tilde{\tau}_{l,2}^k\in\mathbf{C}\setminus\mathbf{R}$. Furthermore, $\widetilde{C}_{0,1}^k=S_{k}\big(S_kvG_kvS_k\big)^{-1}S_{k}$ and all $\widetilde{C}_{l, 0}^k,~ \widetilde{C}_{l, 1}^k,~ \widetilde{C}_{l, 2}^k$ for $0\leq l\leq m-1$ are finite rank.
		
		(iv) For $1\le k\le[\frac{m}{2}]+1$, if zero is an eigenvalue of $H$, then  we have
		\begin{equation}\label{M-2m-even-k+1}
			\begin{split}
				\Big(M(\mu)\Big)^{-1}=&\frac{\widetilde{C}_{0,0}^{k+1}}{\mu^{2m}}+\frac{\tilde{\tau}_{0,1}^{k+1}\widetilde{C}_{0,1}^{k+1}}{\mu^{2m}g(\mu)}
				+\sum_{l=1}^{m-1}\Bigg[\frac{\tilde{\tau}_{l,0}^{k+1}\widetilde{C}_{l,0}^{k+1}}{\mu^{2(m-l)}}
				+\frac{\tilde{\tau}_{l,1}^{k+1}\widetilde{C}_{l,1}^{k+1}}{\mu^{2m}g(\mu)}\Bigg]\\
				&\ \ \ +\tilde{\tau}_{m,0}^{k+1}\widetilde{C}_{m,0}^{k+1}+\big(g(\mu)\big)^{-1}\tilde{\tau}_{m,1}^{k+1}\widetilde{C}_{m,1}^{k+1}+O(\mu^{0+}),
			\end{split}
		\end{equation}
		where the operators $\widetilde{C}_{l,0}^{k+1},\widetilde{C}_{l,1}^{k+1}\in B(s,-s')$ with $s,s'>\frac{n}{2}+2k$ and $\tilde{\tau}_{l,0}^{k+1},\tilde{\tau}_{l,1}^{k+1}\in\mathbf{C}\setminus\mathbf{R}$. Furthermore, $\widetilde{C}_{0,0}^{k+1}=S_{k+1}\big(S_{k+1}vG_{k+1}vS_{k+1}\big)^{-1}S_{k+1}$ and all $\widetilde{C}_{l, 0}^{k+1},~ \widetilde{C}_{l, 1}^{k+1},~ \widetilde{C}_{l, 2}^{k+1}$ for $0\leq l\leq m-1$ are finite rank.
	\end{proposition}
	
	\begin{proof}
		(i) The proof of \eqref{M-2m-even-regular} follows from the proof of \eqref{M-4m-regular}.
		
		(ii)Using \eqref{M-2m-even}, we have
		\begin{equation}
			\begin{split}
				\Big(M(\mu)+S_1\Big)^{-1}
				=&\big(T_0+S_1+\mathcal{M}(\mu)\big)^{-1}\\
				=&D_0\big(I+\mathcal{M}(\mu)D_0\big)^{-1}\\
				=&D_0-D_0\mathcal{M}(\mu)D_0+D_0\big(\mathcal{M}(\mu)D_0\big)^2+O(\mu^{4(m-k)+2+})
			\end{split}
		\end{equation}
		Then for $M_1(\mu)=S_1-S_1\big(M(\mu)+S_1)S_1,$ since $S_1D_0=D_0S_1=S_1,$ we have
		\begin{equation}\label{M1-2m-even}
			\begin{split}
				M_1(\mu)=&S_1-S_1\Big(D_0-D_0\mathcal{M}(\mu)D_0+D_0\big(\mathcal{M}(\mu)D_0\big)^2+O(\mu^{4(m-k)+2+})\Big)S_1\\
				=&S_1\mathcal{M}(\mu)S_1-S_1(\mathcal{M}(\mu)D_0)^{2}S_1+O(\mu^{4(m-k)+2+})\\
				=&\tilde{d}_1\mu^{2(m-k)+2}S_1PS_1+\sum_{l=2}^{k-1}d_{l}\mu^{2(m+l-k)}S_1vG_lvS_1
				+\mu^{2m}g(\mu)S_1vG_kvS_1\\
				&\ \ +\mu^{2m}S_1vG_{k+1}vS_1-\tilde{d}_1^2\mu^{4(m-k)+4}S_1PD_0PS_1+O(\mu^{4(m-k)+4})\\
				:=&\tilde{d}_1\mu^{2(m-k)+2}S_1PS_1+\mathcal{M}(\mu)
			\end{split}
		\end{equation}
		
		By the definition of the first kind of resonance, the operator $T_1=S_1PS_1$ is invertible on $S_1L^2(\mathbf{R}^n),$ then we can obtain the expansions of $\big(M_1(\mu)\big)^{-1}$ by Neumann series.
		
		From Lemma \ref{Feshbach-formula}, we also have \eqref{M-inverse}. Substituting the expansions of $\big(M(\mu)+S_1\big)^{-1}$ and $\big(M_1(\mu)\big)^{-1}$ into \eqref{M-inverse}, we obtain \eqref{M-2m-even-j} with $j=1$.
		
		 For the second kind of resonance, we aim to derive the expansions of $\big(M_1(\mu)\big)^{-1}.$ Define $\widetilde{M}_1(\mu)=\frac{M_1(\mu)}{\tilde{d}_1\mu^{2(m-k)+2}},$ using \eqref{M1-2m-even} we have
		\begin{equation}
			\begin{split}
				\widetilde{M}_1(\mu)=& S_1PS_1+\sum_{l=2}^{k-1}\tilde{d}_l\mu^{2l-2}S_1vG_lvS_1-\tilde{d}_1\mu^{2(m-k)+2}S_1PD_0PS_1\\
				&\ \ +\tilde{d}_{k}\mu^{2k-2}g(\mu)S_1vG_{k}vS_1+\tilde{d}_{k+1}\mu^{2k-2}S_1vG_{k+1}vS_1+O(\mu^{2k-2+})\\
				:=&T_1+\widetilde{\mathcal{M}}_1(\mu).
			\end{split}
		\end{equation}
		By the definition of the second kind of resonance, we know that the operator $T_1$ is not invertible on $S_1L^2(\mathbf{R}^n).$ Since $S_2$ is the Riesz projection onto the kernel of $T_1$ on $S_1L^2(\mathbf{R}^n),$ so $T_1+S_2$ is invertible on $S_1L^2(\mathbf{R}^n).$ Furthermore, we can obtain
		\begin{equation}
			\begin{split}
				\Big(\widetilde{M}_1(\mu)+S_2\Big)^{-1}
				=&\big(T_1+S_2+\widetilde{\mathcal{M}}_1(\mu)\big)^{-1}\\
				=&D_1-D_1\widetilde{\mathcal{M}}_1(\mu)D_1+D_1\big(\widetilde{\mathcal{M}}_1(\mu)D_1\big)^2+O(\mu^{4+})
			\end{split}
		\end{equation}
		Then for $M_2(\mu)=S_2-S_2\big(\widetilde{M}_1(\mu)+S_2\big)^{-1}S_2,$ since $S_2D_1=D_1S_2=S_2$ and $S_2P=PS_2=0,$ we have
		\begin{equation}\label{M2-odd-n}
			\begin{split}
				M_2(\mu)
				=&S_2-S_2\Big(D_1-D_1\widetilde{\mathcal{M}}_1(\mu)D_1+D_1\big(\widetilde{\mathcal{M}}_1(\mu)D_1\big)^2+O(\mu^{4+})\Big)S_2\\
				=&S_2\widetilde{\mathcal{M}}_1(\mu)S_2-S_2\big(\widetilde{\mathcal{M}}_1(\mu)D_1\big)^2S_2+O(\mu^{4+})\\
				=&\tilde{d}_2\mu^2S_2vG_2vS_2+\sum_{l=3}^{k-1}\tilde{d}_l\mu^{2l-2}S_2vG_2vS_2+\tilde{d}_k\mu^{2k-2}g(\mu)S_2vG_kvS_2\\
				&\ \ +\tilde{d}_{k+1}\mu^{2k-2}S_2vG_{k+1}vS_2+O(\mu^{(2k-2)+})\\
				:=&\tilde{d}_2\mu^2S_2vG_2vS_2+\mathcal{M}_2(\mu).
			\end{split}
		\end{equation}
		Since the operator $S_2vG_2vS_2$ is invertible on $S_2L^2(\mathbf{R}^n),$ then we can obtain the expansions of $\big(M_2(\mu)\big)^{-1}$ by Neumann series.
		From Lemma \ref{Feshbach-formula}, we have
		\begin{equation}\label{M1-inverse-n}
			\Big(\widetilde{M}_1(\mu)\Big)^{-1}=\Big(\widetilde{M}_1(\mu)+S_2\Big)^{-1}+\Big(\widetilde{M}_1(\mu)+S_2\Big)^{-1}S_2\Big(M_2(\mu)\Big)^{-1}\Big(\widetilde{M}_1(\mu)+S_2\Big)^{-1}
		\end{equation}
		Substituting the expansions of $\big(\widetilde{M}_1(\mu)+S_2\big)^{-1}$ and $\big(M_2(\mu)\big)^{-1}$ into \eqref{M1-inverse}, we obtain the inverse of $\widetilde{M}_1(\mu)$, then we can obtain the expansions of $\big(M_1(\mu)\big)^{-1}.$ Furthermore, we can obtain \eqref{M-2m-even-j} when $j=2$.
		
		By an induction process we can get the expansions of \eqref{M-2m-even-j} for all $2\le j\le k-1$, then we derive \eqref{M-2m-even-k} and \eqref{M-2m-even-k+1}. What's more, by Lemma \ref{2m-odd-kernel}, we know that $\ker(S_{k+1}vG_{k+1}vS_{k+1})=\{0\},$ which means the iterated process stop here.
	\end{proof}
	
	\subsection{Proof of identification of resonance subspace}\label{identfication}
	  In this part, we aim to identify the resonance subspace of each kind for different dimensional cases with $n>2m$.
	
	\begin{proof}[\bf Proof of Proposition \ref{classification-4m}($n>4m$)]
			
		We first note that
		\[\big[(-\Delta)^m+V\big]\psi=0\Longleftrightarrow (I+G_0V)\psi=0.\]
		First, suppose that $\phi\in S_1L^2\setminus\{0\}$. Then $(U+vG_0v)\phi=0$. Multiplying by $U$, one has
		\[\phi(x)=-UvG_0v\phi=Uv(x)\int_{\mathbf{R}^n}\frac{c~v(y)\phi(y)}{|x-y|^{n-2m}}dy.\]
		Accordingly, we define
		\begin{equation}
			\psi(x)=c\int_{\mathbf{R}^n}\frac{v(y)\phi(y)}{|x-y|^{n-2m}}dy\ \ \ \ (=-G_0v\phi).
		\end{equation}
		Since $v\phi\in L^2_{\frac{\beta}{2}}(\mathbf{R}^n)\subset  L^2_{\frac{n}{2}+}(\mathbf{R}^n)$, we have $\psi\in L^2(\mathbf{R}^n)$ by Lemma \ref{Riesz-potential-boundedness} (1). Further, $\phi(x)=Uv(x)\psi(x)$ and
		$$\psi(x)=-G_0v\phi(x)=-G_0V\psi(x),$$
		which implies $(I+G_0V)\psi(x)=0$.
		
		Secondly, assume $\phi(x)=Uv(x)\psi(x)$ for $\psi(x)$ a non-zero distributional solution to $H\psi=0$. It is clear that $\phi\in L^2_{\frac{\beta}{2}}(\mathbf{R}^n)\subset L^2_{\frac{n}{2}+2-}(\mathbf{R}^n)$ and now
		\[(U+vG_0v)\phi(x)=v(x)\psi(x)+v(x)G_0V\psi(x)=v(x)(I+G_0V)\psi(x)=0.\]
		Thus showing that $\phi\in S_1L^2(\mathbf{R}^n)$.
	\end{proof}
	
	For $3m\leq n\leq 4m$, there exists $N$ kinds of resonances where $N\geq2$ by Definition \ref{resonance}. If $n$ goes down to close $3m$, then $N$ goes up to close $[m/2]+1$. Next, we give the identity condition of each projection $S_{N}$ for $1\leq N\leq [m/2]+1$.
	
	\begin{lemma}\label{S-projection}
		For $1\leq N\leq [m/2]+1$, then $\phi\in S_{N+1}L^{2}$ if and only if
			\begin{equation}\label{Y}
		\int_{\mathbf{R}^n}P_{N-1}(y_1, y_{2}\cdots y_{n})v(y)\phi(y)dy=0
		\end{equation}
		for any polynomial $P_{N-1}$ of degree $N-1$.
%		\begin{equation}\label{Y}
%			\int_{\mathbf{R}^n}y_{a_1}y_{a_2}\cdots y_{a_{N-1}}v(y)\phi(y)dy=0
%		\end{equation}
%		for $a_{\ell}\in \{1, 2, \cdots, n\}$ with $1\leq \ell\leq N-1$ and $y_{a_0}=1$.
	\end{lemma}
	\begin{proof}
		For $N=1$, if $\phi\in S_{2}L^{2}\setminus\{0\}$,  by the definition of $S_{2}$,  then
		\begin{equation*}
			0=\langle S_{1}PS_{1}\phi, \phi \rangle=\|P\phi\|_{L^{2}}^{2},
		\end{equation*}
		thus $P\phi=0$, that is
		\begin{equation}
			\int_{\mathbf{R}^n}v(y)\phi(y)dy=0.
		\end{equation}
		
		For $N=2$, if $\phi\in S_{3}L^{2}\setminus\{0\}$, then
		\begin{equation*}
		\begin{split}
			0=&\langle S_{2}vG_{2}vS_{2}\phi, \phi \rangle=\langle G_{2}v\phi, v\phi \rangle\\
		=&\int_{\mathbf{R}^n}\int_{\mathbf{R}^n}v(x)\phi(x)|x-y|^2v(y)\bar{\phi}(y)dxdy\\
		=& \int_{\mathbf{R}^n}\int_{\mathbf{R}^n}v(x)\phi(x)\big(|x|^2+|y|^2-2x\cdot y\big)v(y)\bar{\phi}(y)dxdy\\
		=& \int_{\mathbf{R}^n}\int_{\mathbf{R}^n}v(x)\phi(x)(-2x\cdot y)v(y)\bar{\phi}(y)dxdy\\
		=& -2\bigg |\int_{\mathbf{R}^n} y v(y) \phi(y) dy \bigg| ^2
		\end{split}
		\end{equation*}
		thus for all $j=1, 2, \cdots, n$,
		\begin{equation}
			\int_{\mathbf{R}^n}y_{j}v(y)\phi(y)dy=0.
		\end{equation}
		
		Suppose that the conclusion \eqref{Y} holds for all $1\leq j\leq N$.
		If $\phi\in S_{N+1}L^{2}\setminus\{0\}$, then by the multinomial theorem, we have
		\begin{equation*}
			\begin{split}
				0=&\langle S_{N}vG_{N}vS_{N}\phi, \phi \rangle=\langle G_{N}v\phi, v\phi \rangle\\
				=&\int_{\mathbf{R}^n}\int_{\mathbf{R}^n}v(x)\phi(x)|x-y|^{2N-2}v(y)\bar{\phi}(y)dxdy\\
				=& \int_{\mathbf{R}^n}\int_{\mathbf{R}^n}v(x)\phi(x)\big(|x|^2-2x\cdot y+|y|^2\big)^{N-1}v(y)\bar{\phi}(y)dxdy\\
				=& \int_{\mathbf{R}^n}\int_{\mathbf{R}^n}v(x)\phi(x)
				\sum_{k_1+k_2+k_3=N-1}C({\bf k})|x|^{2k_1}(-2x\cdot y)^{k_2}|y|^{2k_3}v(y)\bar{\phi}(y)dxdy\\
				=& \int_{\mathbf{R}^n}\int_{\mathbf{R}^n}v(x)\phi(x)
				\sum_{k_1+k_2+k_3=N-1}C({\bf k})|x|^{2k_1}\\
				&\times \sum_{a_1+\cdots+a_n=k_2}C({\bf a})(-2)^{k_2}(x_{1}y_{1})^{a_1}\cdots(x_{n}y_{n})^{a_n}|y|^{2k_3}v(y)\bar{\phi}(y)dxdy\\
				=& \int_{\mathbf{R}^n}\int_{\mathbf{R}^n}v(x)\phi(x)
				\sum_{k_1+k_2+k_3=N-1}C({\bf k})
				\sum_{ a_1, \cdots, a_{k_2}=1; a_1\leq \cdots\leq a_{k_2}}^{k_2}C({\bf a})(-2)^{k_2}\\
				&\times |x|^{2k_1}(x_{a_1}x_{a_2}\cdots x_{a_{k_2}})(y_{a_{1}}y_{a_2}\cdots y_{a_{k_2}})|y|^{2k_3}v(y)\bar{\phi}(y)dxdy
			\end{split}
		\end{equation*}
		where $C({\bf k})=\frac{(N-1)!}{k_{1}!\cdot k_{2}! \cdots k_{m}!}$ and $C({\bf a})=\frac{k_2!}{a_{1}!\cdot a_{2}! \cdots a_{n}!}$.
		
		Note that the term $$|x|^{2k_1}\sum_{ a_1, \cdots, a_{k_2}=1; a_1\leq \cdots\leq a_{k_2}}^{k_2}(x_{a_1}x_{a_2}\cdots x_{a_{k_2}})$$ has actually the same form as $\sum_{ a_1, \cdots, a_{k_2}=1; a_1\leq \cdots\leq a_{k_2}}^{k_2}(x_{a_1}x_{a_2}\cdots x_{a_{k_2}})$. Since $|x-y|^{2N-2}$ is symmetric about $x$ and $y$, by $S_{N+1}\leq S_{N}\leq\cdots\leq S_{1}$, then the terms contribute zero is  the term as following ($x, y$ has the same order)
		$$\int_{\mathbf{R}^n}\int_{\mathbf{R}^n}v(x)\phi(x)
		\sum_{a_1, \cdots, a_{N-1}=1; a_1\leq \cdots\leq a_{N-1}}^{N-1}(x_{a_1}x_{a_2}\cdots x_{a_{N-1}})(y_{a_{1}}y_{a_2}\cdots y_{a_{N-1}})v(y)\bar{\phi}(y)dxdy.$$
		Note that the above terms come from terms like $|x|^{2a}(-2x\cdot y)^{N-1-2a}|y|^{2a}$, thus all the above terms have the same sign ($+$ or $-$). Thus Lemma \ref{S-projection} holds by the mathematical induction.
	\end{proof}

	\begin{proof}[\bf Proof of Proposition \ref{classification-2m-odd}($2m<n\le 4m$ and odd)]
      We give the detail proof of Proposition \ref{classification-2m-odd} for the cases when $j=1,2,3,4,$  then by an induction process, we prove the general cases for $5\le j\le k+1. $

		(i) For $j=1,$ suppose that $\phi\in S_1L^2\setminus\{0\}$. Then $(U+vG_0v)\phi=0$, and multiplying by $U$,  one has
		\[\phi(x)=-UvG_0v\phi=Uv(x)\int_{\mathbf{R}^n}\frac{c~v(y)\phi(y)}{|x-y|^{n-2m}}dy.\]
		Accordingly, we define
		\begin{equation}\label{psi}
			\psi(x)=c\int_{\mathbf{R}^n}\frac{v(y)\phi(y)}{|x-y|^{n-2m}}dy\ \ (=-G_0v\phi).
		\end{equation}
		Since $v\phi\in L^2_{\frac{\beta}{2}}(\mathbf{R}^n)\subset L^2_{\frac{n}{2}+2+}(\mathbf{R}^n)$, we have that $\psi\in W_{2m-\frac{n}{2}}(\mathbf{R}^n)$ by Lemma \ref{Riesz-potential-boundedness}(2). Further $\phi(x)=Uv(x)\psi(x)$ and
		\[\psi(x)=-G_0v\phi(x)=-G_0V\psi(x) \Longrightarrow (I+G_0V)\psi(x)=0.\]
		
		On the other hand, assume $\phi(x) =Uv(x)\psi(x)$ for $\psi(x)$ a non-zero distributional solution to $H\psi=0$. It is clear that $\phi\in L^2_{(\frac{\beta}{2}-2m+\frac{n}{2})-}(\mathbf{R}^n) \subset L^2_{(n+2k-2m)-}(\mathbf{R}^n)$ and then
		\[(U+vG_0v)\phi(x)=v(x)\psi(x)+v(x)G_0V\psi(x)=v(x)(I+G_0V)\psi(x)=0.\]
		Thus showing that $\phi\in S_1L^2(\mathbf{R}^n)$.
		
		(ii) For $j=2,$ assume first that $\phi\in S_2L^2(\mathbf{R}^n)\setminus\{0\}$. Since $S_2\le S_1$, then by Lemma \ref{S-projection} and our definition of $\psi(x)$, we have
		\begin{align*}
        \psi(x)=&~c\int_{\mathbf{R}^n}\Big(\frac{1}{|x-y|^{n-2m}}-\frac{1}{(1+|x|)^{n-2m}}\Big)v(y)\phi(y)dy\\
              :=&~c\int_{\mathbf{R}^n}K_2(x,y)v(y)\phi(y)dy.
        \end{align*}
        And we can rewrite the integral kernel $K_2(x,y)$ as
        \begin{align*}
        K_2(x,y)=& \frac{\sum_{l=0}^{n-2m-1}C_{n-2m}^l|x|^l}{(1+|x|)^{n-2m}|x-y|^{n-2m}}
				+\frac{|x|^{n-2m}-|x-y|^{n-2m}}{(1+|x|)^{n-2m}|x-y|^{n-2m}}\\
               :=& K_2^1(x,y)+K_2^2(x,y).
        \end{align*}
        Here and subsequently, for $k, l\in\mathbf{N}$, $C_{k}^{l}=\frac{k!}{l!(k-l)!}$ is the combination number.
		Since
        \begin{equation}\label{eq-subspace2-1}
        \Big|K_2^1(x,y)\Big|\lesssim\frac{1}{(1+|x|)|x-y|^{n-2m}},
        \end{equation}
        and
        \begin{equation}\label{eq-subspace2-2}
        \begin{split}
        \Big|K_2^2(x,y)\Big|
        \lesssim~& \frac{|x|^{n-2m-1}\Big||x|-|x-y|\Big|}{(1+|x|)^{n-2m}|x-y|^{n-2m}}+\frac{\Big||x|^{n-2m-1}-|x-y|^{n-2m-1}\Big|}{(1+|x|)^{n-2m}|x-y|^{n-2m-1}}\\
        \lesssim~& \frac{1+|y|}{(1+|x|)|x-y|^{n-2m}}+\frac{(1+|y|)^{n-2m-1}}{(1+|x|)^2|x-y|^{n-2m-1}}.
        \end{split}
        \end{equation}
       then by \eqref{eq-subspace2-1} and \eqref{eq-subspace2-2} we can obtain that
        \begin{equation}\label{eq-subspace2}
        \begin{split}
        \Big|K_2(x,y)\Big|
                          \lesssim~\frac{1+|y|}{(1+|x|)|x-y|^{n-2m}}+\frac{(1+|y|)^{n-2m-1}}{(1+|x|)^2|x-y|^{n-2m-1}}.
        \end{split}
        \end{equation}
		Thus Lemma \ref{Riesz-potential-boundedness} shows that $\psi\in W_{2m-\frac{n}{2}-1}(\mathbf{R}^n)$ as desired.
		
		On the other hand, if $\phi=Uv\psi$ with $\psi(x)\in W_{2m-\frac{n}{2}-1}(\mathbf{R}^n)$, we have
		\begin{equation*}
			\psi(x)=c\int_{\mathbf{R}^n}K_2(x,y)v(y)\phi(y)dy+\frac{c}{(1+|x|)^{n-2m}}\int_{\mathbf{R}^n}v(y)\phi(y)dy.
		\end{equation*}
		The first term in the right hand and $\psi(x)$ are in $W_{2m-\frac{n}{2}-1}(\mathbf{R}^n)$, thus we must have that
		\[\frac{1}{(1+|x|)^{n-2m}}\int_{\mathbf{R}^n}v(y)\phi(y)dy\in W_{2m-\frac{n}{2}-1}(\mathbf{R}^n),\]
		this necessitates that
		\begin{equation}\label{eq-orthogonal1}
			\int_{\mathbf{R}^n}v(y)\phi(y)dy=0
		\end{equation}
		that is $0=P\phi=S_1PS_1\phi$ and $\phi\in S_2L^2(\mathbf{R}^n)$ as desired.
		
		(iii) For $j=3$, assume first that $\phi\in S_3L^2(\mathbf{R}^n)\setminus\{0\}$. Since $S_3\le S_2\le S_1$, then by Lemma \ref{S-projection}, we have
        \begin{align*}
        \psi(x)=&c\int_{\mathbf{R}^n}\bigg[\frac{1}{|x-y|^{n-2m}}-\frac{1}{(1+|x|)^{n-2m}}-\frac{C_{n-2m}^{1}}{(1+|x|)^{n-2m+1}}-\frac{\big(2C_{\frac{n-1}{2}-m}^1+1\big)x\cdot y}{(1+|x|)^{n-2m+2}}\bigg]v(y)\phi(y)dy\\
              :=&c\int_{\mathbf{R}^n}K_3(x,y)v(y)\phi(y)dy.
        \end{align*}
		We divide the kernel $K_3(x,y)$ into three parts as follows:
        \begin{align*}
        K_3(x,y)=& \frac{\sum_{l=0}^{n-2m-2}|x|^l}{(1+|x|)^{n-2m}|x-y|^{n-2m}}
				+\Bigg(\frac{C_{n-2m}^1|x|^{n-2m-1}}{(1+|x|)^{n-2m}|x-y|^{n-2m}}-\frac{C_{n-2m}^1}{(1+|x|)^{n-2m+1}}\Bigg)\\
				&\ \ +\Bigg(\frac{|x|^{n-2m}-|x-y|^{n-2m}}{(1+|x|)^{n-2m}|x-y|^{n-2m}}-\frac{\big(2C_{\frac{n-1}{2}-m}^1+1\big)x\cdot y}{(1+|x|)^{n-2m+2}}\Bigg)\\
               :=& K_3^1(x,y)+K_3^2(x,y)+K_3^3(x,y).
        \end{align*}
		For the first part $K_3^1(x,y)$, we have
		\begin{equation}\label{eq-subspace3-1}
        \Big|K_3^1(x,y)\Big|\lesssim\frac{1}{(1+|x|)^2|x-y|^{n-2m}}.
        \end{equation}
		For the second part $K_3^2(x,y)$, using \eqref{eq-subspace2-2}, we have
		\begin{equation}\label{eq-subspace3-2}
        \begin{split}
		\Big|K_3^2(x,y)\Big|\lesssim~ & \frac{|x|^{n-2m-1}}{(1+|x|)^{n-2m+1}|x-y|^{n-2m}}+\frac{1}{(1+|x|)}\Bigg|
        \frac{|x|^{n-2m}-|x-y|^{n-2m}}{(1+|x|)^{n-2m+1}|x-y|^{n-2m}}\Bigg|\\
        \lesssim~& \frac{1+|y|}{(1+|x|)^2|x-y|^{n-2m}}+\frac{(1+|y|)^{n-2m-1}}{(1+|x|)^{3}|x-y|^{n-2m-1}}.
        \end{split}
		\end{equation}
		Then for the third part $K_3^3(x,y)$, we have
		\begin{align*}
		K_3^3(x,y)= & \Bigg(\frac{|x|^{n-2m-1}\big(|x|-|x-y|\big)}{(1+|x|)^{n-2m}|x-y|^{n-2m}} -\frac{x\cdot y}{(1+|x|)^{n-2m+2}}\Bigg)\\
                    &\ \ +\Bigg(\frac{|x|^{n-2m-1}-|x-y|^{n-2m-1}}{(1+|x|)^{n-2m}|x-y|^{n-2m-1}}-\frac{2C_{\frac{n-1}{2}-m}^1x\cdot y}{(1+|x|)^{n-2m+2}}\Bigg).\\
                 := & K_3^{3,1}(x,y)+K_3^{3,2}(x,y).
		\end{align*}
		Furthermore, for $K_3^{3,1}(x,y)$ we have
		\begin{equation}\label{eq-subspace3-3-1}
		\begin{split}
		\Big|K_3^{3,1}(x,y)\Big|=~ & \Bigg|\frac{|x|^{n-2m-1}\big(|x|^2-|x-y|^2\big)}{(1+|x|)^{n-2m+2}|x-y|^{n-2m}(|x-y|+|x|)} -\frac{x\cdot y}{(1+|x|)^{n-2m+2}}\Bigg|\\
		\lesssim~ & \Bigg|\frac{2x\cdot y|x|^{n-2m-1}}{(1+|x|)^{n-2m}|x-y|^{n-2m}(|x-y|+|x|)} -\frac{x\cdot y}{(1+|x|)^{n-2m+2}}\Bigg|\\
        &~+\frac{|x|^{n-2m-1}|y|^2}{(1+|x|)^{n-2m}|x-y|^{n-2m}(|x-y|+|x|)}\\
		\lesssim~& |x\cdot y|\Bigg|\frac{2|x|^{n-2m-1}(1+|x|)^2-|x-y|^{n-2m}(|x-y|+|x|)}{(1+|x|)^{n-2m+2}|x-y|^{n-2m}(|x-y|+|x|)}\Bigg|+ \frac{|y|^2}{(1+|x|)|x-y|^{n-2m+1}}\\
        \lesssim~ &|x\cdot y|\Bigg|\frac{|x|^{n-2m+1}-|x-y|^{n-2m+1}}{(1+|x|)^{n-2m+2}|x-y|^{n-2m}(|x-y|+|x|)}\Bigg|+ \frac{1+|y|}{(1+|x|)^2|x-y|^{n-2m+1}}\\
        &+|x\cdot y|\Bigg|\frac{|x|(|x|^{n-2m}-|x-y|^{n-2m})}{(1+|x|)^{n-2m+2}|x-y|^{n-2m}(|x-y|+|x|)}\Bigg|+\frac{(1+|y|)^2}{(1+|x|)|x-y|^{n-2m+1}}\\
        \lesssim~&\frac{(1+|y|)^{n-2m+2}}{(1+|x|)^{2}|x-y|^{n-2m+1}}+\frac{(1+|y|)^{n-2m}}{(1+|x|)^2|x-y|^{n-2m}}.
		\end{split}
		\end{equation}
       And for $K_3^{3,2}(x,y),$ we obtain
       \begin{equation}\label{eq-subspace3-3-2}
       \begin{split}
       \Big|K_3^{3,2}(x,y)\Big|=~ & \Bigg|\frac{|x|^{n-2m-1}-(|x|^2-2x\cdot y+|y|^2)^{\frac{n-1}{2}-m}}{(1+|x|)^{n-2m}|x-y|^{n-2m-1}}-\frac{2C_{\frac{n-1}{2}-m}^1x\cdot y}{(1+|x|)^{n-2m+2}}\Bigg|\\
       \lesssim~&\Bigg|\frac{2C_{\frac{n-1}{2}-m}^1|x|^{n-2m-3}x\cdot y}{(1+|x|)^{n-2m}|x-y|^{n-2m-1}}-\frac{2C_{\frac{n-1}{2}-m}^1x\cdot y}{(1+|x|)^{n-2m+2}}\Bigg|+\frac{C_{\frac{n-1}{2}-m}^1|x|^{n-2m-3}|y|^2}{(1+|x|)^{n-2m}|x-y|^{n-2m-1}}\\
       &~+\Bigg|\sum_{l=2}^{\frac{n-1}{2}-m}\frac{C_{\frac{n-1}{2}-m}^l|x|^{n-2m-1-2l}(-2x\cdot y+|y|^2)^l}{(1+|x|)^{n-2m}|x-y|^{n-2m-1}}\Bigg|\\
       \lesssim~&\frac{(1+|y|)^{n-2m}}{(1+|x|)^{3}|x-y|^{n-2m-1}}.
       \end{split}
       \end{equation}
       Then by \eqref{eq-subspace3-3-1} and \eqref{eq-subspace3-3-2}, we have
		\begin{equation}\label{eq-subspace3-3}
        \begin{split}
		\Big|K_3^3(x,y)\Big|
        \lesssim& \frac{(1+|y|)^{n-2m+2}}{(1+|x|)^{2}|x-y|^{n-2m+1}}+\frac{(1+|y|)^{n-2m}}{(1+|x|)^2|x-y|^{n-2m}}
         +\frac{(1+|y|)^{n-2m}}{(1+|x|)^{3}|x-y|^{n-2m-1}}.
        \end{split}
		\end{equation}
        According to \eqref{eq-subspace3-1}, \eqref{eq-subspace3-2} and \eqref{eq-subspace3-3}. We obtain that
        \begin{equation}\label{eq-subspace3}
        \begin{split}
        \Big|K_3(x,y)\Big|\lesssim& \frac{(1+|y|)^{n-2m+2}}{(1+|x|)^{2}|x-y|^{n-2m+1}}+\frac{(1+|y|)^{n-2m}}{(1+|x|)^2|x-y|^{n-2m}} +\frac{(1+|y|)^{n-2m}}{(1+|x|)^{3}|x-y|^{n-2m-1}}.
        \end{split}
        \end{equation}
		Thus Lemma \ref{Riesz-potential-boundedness} shows that $\psi\in W_{2m-\frac{n}{2}-2}(\mathbf{R}^n)$ as desired.

       On the other hand, if $\phi=Uv\psi$ with $\psi(x)\in W_{2m-\frac{n}{2}-2}(\mathbf{R}^n)$, since $$W_{2m-\frac{n}{2}-2}(\mathbf{R}^n)\subset W_{2m-\frac{n}{2}-1}(\mathbf{R}^n),$$
       thus $\phi\in S_2L^2(\mathbf{R}^n),$ using \eqref{eq-orthogonal1}, we have
		\begin{equation*}
		\begin{split}
        \psi(x)=&c\int_{\mathbf{R}^n}K_3(x,y)v(y)\phi(y)dy
        +c\int_{\mathbf{R}^n}\frac{\big(2C_{\frac{n-1}{2}-m}^1+1\big)x\cdot y}{(1+|x|)^{n-2m+2}}v(y)\phi(y)dy\\
        \end{split}
		\end{equation*}
		The first term in the right hand and $\psi(x)$ are in $W_{2m-\frac{n}{2}-2}(\mathbf{R}^n)$, thus we must have that
		\[\frac{1}{(1+|x|)^{n-2m+2}}\int_{\mathbf{R}^n}x\cdot yv(y)\phi(y)dy=\sum_{j=1}^n\frac{x_j}{(1+|x|)^{n-2m+2}}\int_{\mathbf{R}^n}y_jv(y)\phi(y)dy\in W_{2m-\frac{n}{2}-2}(\mathbf{R}^n),\]
		this necessitates that
		\begin{equation}\label{eq-orthogonal2}
			\int_{\mathbf{R}^n}y_jv(y)\phi(y)dy=0,\ \ \ \forall 1\le j\le n.
		\end{equation}
		Hence by Lemma \ref{S-projection}, we have $\phi\in S_3L^2(\mathbf{R}^n)$ as desired.

        (iv) For $j=4$, assume first that $\phi\in S_4L^2(\mathbf{R}^n)\setminus\{0\}$. Since $S_4\le S_3\le S_2\le S_1$, then by Lemma \ref{S-projection}, we have
        \begin{align*}
        \psi(x)=
        &c\int_{\mathbf{R}^n}\Bigg[~\frac{1}{|x-y|^{n-2m}}-\frac{1}{(1+|x|)^{n-2m}}-\frac{C_{n-2m}^1}{(1+|x|)^{n-2m+1}}-\frac{C_{n-2m}^1+C_{n-2m}^2}{(1+|x|)^{n-2m+2}}\\
        &~~~\ \ \ \ \ \ \  -\frac{\big(2C_{\frac{n-1}{2}-m}^1+1\big)(2x\cdot y-|y|^2)}{2(1+|x|)^{n-2m+2}}-\frac{\big(C_{n-2m}^1+2\big)(2C_{\frac{n-1}{2}-m}^1+1\big)x\cdot y}{(1+|x|)^{n-2m+3}}\\
        &~~~\ \ \ \ \ \ \ -\frac{\Big(4\big(C_{\frac{n-1}{2}-m}^1\big)^2+C_{\frac{n+1}{2}-m}^1+C_{\frac{n-1}{2}-m}^1+1-4C_{\frac{n-1}{2}-m}^2\Big)(x\cdot y)^2}{(1+|x|)^{n-2m+4}}\Bigg]v(y)\phi(y)dy\\
        :=&~c\int_{\mathbf{R}^n} K_4(x,y)v(y)\phi(y)dy.
        \end{align*}
        Similar to the proof of $j=3,$ we divide the kernel $K_4(x,y)$ into four parts as follows:
        \begin{equation*}
			\begin{split}
			&~~ K_4(x,y)\\
           =&\frac{\sum_{l=0}^{n-2m-3}|x|^l}{(1+|x|)^{n-2m}|x-y|^{n-2m}}
				+\Bigg(\frac{C_{n-2m}^2|x|^{n-2m-2}}{(1+|x|)^{n-2m}|x-y|^{n-2m}}-\frac{C_{n-2m}^2}{(1+|x|)^{n-2m+2}}\Bigg)\\
				& +\Bigg(\frac{C_{n-2m}^1|x|^{n-2m-1}}{(1+|x|)^{n-2m}|x-y|^{n-2m}}-\frac{C_{n-2m}^1}{(1+|x|)^{n-2m+1}}-\frac{C_{n-2m}^1}{(1+|x|)^{n-2m+2}}
                     -\frac{C_{n-2m}^1\big(2C_{\frac{n-1}{2}-m}^1+1\big)x\cdot y}{(1+|x|)^{n-2m+3}}\Bigg)\\
                & +\Bigg(\frac{|x|^{n-2m}-|x-y|^{n-2m}}{(1+|x|)^{n-2m}|x-y|^{n-2m}}-\frac{\big(2C_{\frac{n-1}{2}-m}^1+1\big)(2x\cdot y-|y|^2)}{2(1+|x|)^{n-2m+2}}-\frac{\big(4C_{\frac{n-1}{2}-m}^1+2)x\cdot y}{(1+|x|)^{n-2m+3}}\\
                &\ \ \ \ \ \ \ -\frac{\Big(4\big(C_{\frac{n-1}{2}-m}^1\big)^2+C_{\frac{n+1}{2}-m}^1+C_{\frac{n-1}{2}-m}^1+1-4C_{\frac{n-1}{2}-m}^2\Big)(x\cdot y)^2}{(1+|x|)^{n-2m+4}}\Bigg)\\
				:=& K_4^1(x,y)+K_4^2(x,y)+K_4^3(x,y)+K_4^4(x,y).
			\end{split}
		\end{equation*}
        For the first part $K_4^1(x,y)$,~we have
        \begin{equation}\label{eq-subspace4-1}
        \Big|K_4^1(x,y)\Big|\lesssim\frac{1}{(1+|x|)^3|x-y|^{n-2m}}.
        \end{equation}
        For the second part $K_4^2(x,y)$,~using \eqref{eq-subspace2-2},~we have
        \begin{equation}\label{eq-subspace4-2}
        \begin{split}
        K_4^2(x,y)\lesssim~ &\frac{1}{(1+|x|)^2}\Bigg|\frac{|x|^{n-2m-2}+2|x|^{n-2m-1}+|x|^{n-2m}-|x-y|^{n-2m}}{(1+|x|)^{n-2m}|x-y|^{n-2m}}\Bigg|\\
        \lesssim~&  \frac{1}{(1+|x|)^3|x-y|^{n-2m}}+\frac{1}{(1+|x|)^2}\Bigg|\frac{|x|^{n-2m}-|x-y|^{n-2m}}{(1+|x|)^{n-2m}|x-y|^{n-2m}}\Bigg|\\
        \lesssim~&  \frac{1+|y|}{(1+|x|)^3|x-y|^{n-2m}}+\frac{(1+|y|)^{n-2m-1}}{(1+|x|)^4|x-y|^{n-2m-1}}.
        \end{split}
        \end{equation}
       For the third part $K_4^3(x,y)$,~using \eqref{eq-subspace2-2} and \eqref{eq-subspace3}, we have
        \begin{equation}\label{eq-subspace4-3}
        \begin{split}
        \Big|K_4^3(x,y)\Big|\lesssim &\frac{1}{(1+|x|)}\Big|K_3(x,y)\Big|+\frac{1}{(1+|x|)^2}\Big|K_2^2(x,y)\Big|+\frac{|x|^{n-2m-1}}{(1+|x|)^{n-2m+2}|x-y|^{n-2m}}\\
        \lesssim~& \frac{(1+|y|)^{n-2m+2}}{(1+|x|)^{2}|x-y|^{n-2m+1}}+\frac{(1+|y|)^{n-2m}}{(1+|x|)^2|x-y|^{n-2m}} +\frac{(1+|y|)^{n-2m}}{(1+|x|)^{3}|x-y|^{n-2m-1}}.
        \end{split}
        \end{equation}
        Then for the fourth part $K_4^4(x,y)$, we have
        \begin{align*}
        &K_4^4(x,y)\\
        =~& \Bigg(\frac{|x|^{n-2m-1}(|x|-|x-y|)}{(1+|x|)^{n-2m}|x-y|^{n-2m}}-\frac{2x\cdot y-|y|^2}{2(1+|x|)^{n-2m+2}}-\frac{2x\cdot y}{(1+|x|)^{n-2m+3}}-\frac{C_{\frac{n+1}{2}-m}^1(x\cdot y)^2}{(1+|x|)^{n-2m+4}}\Bigg)\\
        &+\Bigg(\frac{|x|^{n-2m-1}-|x-y|^{n-2m-1}}{(1+|x|)^{n-2m}|x-y|^{n-2m-1}}-\frac{C_{\frac{n-1}{2}-m}^1(2x\cdot y-|y|^2)}{(1+|x|)^{n-2m+2}}-\frac{4C_{\frac{n-1}{2}-m}^1x\cdot y}{(1+|x|)^{n-2m+3}}\\
        &\ \ \ \ \ \ -\frac{4\Big(\big(C_{\frac{n-1}{2}-m}^1\big)^2-C_{\frac{n-1}{2}-m}^2\Big)(x\cdot y)^2}{(1+|x|)^{n-2m+4}|x-y|^{n-2m-1}}\Bigg)\\
        :=~& K_4^{4,1}(x,y)+K_4^{4,2}(x,y).
        \end{align*}
        It is obvious that we can rewrite $K_4^{4,1}(x,y)$ as follows:
        \begin{align*}
        K_4^{4,1}(x,y)
        =&\frac{|x|^{n-2m-1}(|x|^2-|x-y|^2)}{(1+|x|)^{n-2m}|x-y|^{n-2m}(|x|+|x-y|)}-\frac{2x\cdot y-|y|^2}{2(1+|x|)^{n-2m+2}}\\
        &\ -\frac{2x\cdot y}{(1+|x|)^{n-2m+3}}-\frac{C_{\frac{n+1}{2}-m}^1(x\cdot y)^2}{(1+|x|)^{n-2m+4}}\\
        =&\Bigg(\frac{x\cdot y(|x|^{n-2m+1}-|x-y|^{n-2m+1})}{(1+|x|)^{n-2m+2}|x-y|^{n-2m}(|x|+|x-y|)}-\frac{C_{\frac{n+1}{2}-m}^1(x\cdot y)^2}{(1+|x|)^{n-2m+4}}\Bigg)\\
        &+\Bigg(\frac{x\cdot y|x|^{n-2m}(|x|-|x-y|)}{(1+|x|)^{n-2m+2}|x-y|^{n-2m}(|x|+|x-y|)}-\frac{(x\cdot y)^2}{(1+|x|)^{n-2m+4}}\Bigg)\\
        &+\Bigg(\frac{x\cdot y|x|(|x|^{n-2m-1}-|x-y|^{n-2m-1})}{(1+|x|)^{n-2m+2}|x-y|^{n-2m-1}(|x|+|x-y|)}-\frac{C_{\frac{n-1}{2}-m}(x\cdot y)^2}{(1+|x|)^{n-2m+4}}\Bigg)\\
        &+\Bigg(\frac{4x\cdot y|x|^{n-2m}}{(1+|x|)^{n-2m+2}|x-y|^{n-2m}(|x|+|x-y|)}-\frac{2x\cdot y}{(1+|x|)^{n-2m+3}}\Bigg)\\
        &+\Bigg(-\frac{|y|^2|x|^{n-2m-1}}{(1+|x|)^{n-2m}|x-y|^{n-2m}(|x|+|x-y|)}+\frac{|y|^2}{2(1+|x|)^{n-2m+2}}\Bigg).
        \end{align*}
        thus for $K_4^{4,1}(x,y),$ we have
        \begin{equation}\label{eq-subspace4-4-1}
        \begin{split}
        \Big|K_4^{4,1}(x,y)\Big|\lesssim&\frac{(1+|y|)^{n-2m+4}}{(1+|x|)^{3}|x-y|^{n-2m+2}}+\frac{(1+|y|)^{n-2m+3}}{(1+|x|)^{3}|x-y|^{n-2m+1}}+
        \frac{(1+|y|)^{n-2m+2}}{(1+|x|)^{3+2}|x-y|^{n-2m}}.
        \end{split}
        \end{equation}
        Furthermore, we have
        \begin{align*}
        K_4^{4,2}(x,y)=&\frac{-\sum_{l=3}^{\frac{n-1}{2}-m}C_{\frac{n-1}{2}-m}^l|x|^{n-2m-1-2l}(-2x\cdot y+|y|^2)^l}{(1+|x|)^{n-2m}|x-y|^{n-2m-1}}\\
        &+\Bigg(\frac{-C_{\frac{n-1}{2}-m}^2|x|^{n-2m-5}(2x\cdot y-|y|^2)^2}{(1+|x|)^{n-2m}|x-y|^{n-2m-1}}+\frac{4C_{\frac{n-1}{2}-m}^2(x\cdot y)^2}{(1+|x|)^{n-2m+4}}\Bigg)\\
        &+\Bigg(\frac{C_{\frac{n-1}{2}-m}^1|x|^{n-2m-3}(2x\cdot y-|y|^2)}{(1+|x|)^{n-2m}|x-y|^{n-2m-1}}-\frac{C_{\frac{n-1}{2}-m}^1(2x\cdot y-|y|^2)}{(1+|x|)^{n-2m+2}}-\frac{4C_{\frac{n-1}{2}-m}^1x\cdot y}{(1+|x|)^{n-2m+3}}\\
        &\ \ \ \ \ -\frac{4\Big(C_{\frac{n-1}{2}-m}^1\Big)^2(x\cdot y)^2}{(1+|x|)^{n-2m+4}|x-y|^{n-2m-1}}\Bigg).
        \end{align*}
        then we obtain that
        \begin{equation}\label{eq-subspace4-4-2}
        \begin{split}
        \Big|K_4^{4,2}(x,y)\Big|\lesssim&\frac{(1+|y|)^{n-2m+1}}{(1+|x|)^{4}|x-y|^{n-2m-1}}.
        \end{split}
        \end{equation}
        By \eqref{eq-subspace4-4-1} and \eqref{eq-subspace4-4-2}, we have
        \begin{equation}\label{eq-subspace4-4}
        \begin{split}
        \Big|K_4^4(x,y)\Big|\lesssim& \frac{(1+|y|)^{n-2m+4}}{(1+|x|)^{3}|x-y|^{n-2m+2}}+\frac{(1+|y|)^{n-2m+3}}{(1+|x|)^{3}|x-y|^{n-2m+1}}\\
        &\ +\frac{(1+|y|)^{n-2m+2}}{(1+|x|)^{3+2}|x-y|^{n-2m}}+\frac{(1+|y|)^{n-2m+1}}{(1+|x|)^{4}|x-y|^{n-2m-1}}.
        \end{split}
        \end{equation}\label{eq-subspace}
        Hence using \eqref{eq-subspace4-1}, \eqref{eq-subspace4-2}, \eqref{eq-subspace4-3} and \eqref{eq-subspace4-4}, we obtain that
        \begin{equation}\label{eq-subspace4}
        \begin{split}
        \Big|K_4(x,y)\Big|\lesssim& \frac{(1+|y|)^{n-2m+4}}{(1+|x|)^{3}|x-y|^{n-2m+2}}+\frac{(1+|y|)^{n-2m+3}}{(1+|x|)^{3}|x-y|^{n-2m+1}}\\
        &\ +\frac{(1+|y|)^{n-2m+2}}{(1+|x|)^{3+2}|x-y|^{n-2m}}+\frac{(1+|y|)^{n-2m+1}}{(1+|x|)^{4}|x-y|^{n-2m-1}}.
        \end{split}
        \end{equation}
        Thus Lemma \ref{Riesz-potential-boundedness} shows that $\psi\in W_{2m-\frac{n}{2}-3}(\mathbf{R}^n)$ as desired.

        On the other hand, if $\phi=Uv\psi$ with $\psi(x)\in W_{2m-\frac{n}{2}-3}(\mathbf{R}^n)$, since
        \[W_{2m-\frac{n}{2}-3}(\mathbf{R}^n)\subset W_{2m-\frac{n}{2}-2}(\mathbf{R}^n)\]
       which implies that $\phi\in S_3L^2(\mathbf{R}^n),$ using \eqref{eq-orthogonal1} and \eqref{eq-orthogonal2}, we have
		\begin{align*}
        \psi(x)=
        &c\int_{\mathbf{R}^n}K_4(x,y)v(y)\phi(y)dy-c\int_{\mathbf{R}^n}\frac{\big(2C_{\frac{n-1}{2}-m}^1+1\big)|y|^2}{2(1+|x|)^{n-2m+2}}v(y)\phi(y)dy\\
        &+c\int_{\mathbf{R}^n} \frac{\Big(4\big(C_{\frac{n-1}{2}-m}^1\big)^2+C_{\frac{n+1}{2}-m}^1+C_{\frac{n-1}{2}-m}^1+1-4C_{\frac{n-1}{2}-m}^2\Big)(x\cdot y)^2}{(1+|x|)^{n-2m+4}}v(y)\phi(y)dy.
        \end{align*}
		The first term in the right hand and $\psi(x)$ are in $W_{2m-\frac{n}{2}-3}(\mathbf{R}^n)$, thus we need the second and third terms are also in $W_{2m-\frac{n}{2}-3}(\mathbf{R}^n),$ similar to the case for $j=3,$ this necessitates that
		\begin{equation}\label{eq-orthogonal3}
			\int_{\mathbf{R}^n}y_iy_jv(y)\phi(y)dy=0,\ \ \ \forall 1\le i, j\le n.
		\end{equation}
		Hence by Lemma \ref{S-projection}, we obtain $\phi\in S_4L^2(\mathbf{R}^n)$ as desired.

		(v) For $5\le j\le k+1$ , assume first $\phi(x)\in S_{j}L^2(\mathbf{R}^n)\setminus\{0\}$, using Lemma \ref{S-projection} and $$S_{j}\leq S_{j-1}\leq\cdots\leq S_{1},$$ by an induction process, we have
		\begin{equation*}
		\begin{split}
		\psi(x)=&c\int_{\mathbf{R}^n}\bigg[\frac{1}{|x-y|^{n-2m}}-\frac{1}{(1+|x|)^{n-2m}}
		-\sum_{l=1}^{2(j-2)}\sum_{j_1+2j_2\leq j-2}\frac{C_{1}(l)+C_{2}(l)(x\cdot y)^{j_1}|y|^{2j_2}}{(1+|x|)^{n-2m+l}}\bigg]v(y)\phi(y)dy\\
        :=&c\int_{\mathbf{R}^n}K_j(x,y)v(y)\phi(y)dy.
		\end{split}
		\end{equation*}
		where the constants $C_{1}(l), C_{2}(l)$ which only depend on $l,j_1,j_2$ can be choosed case by case. According to the proof of the cases with $j=2,3,4,$ by an induction process, we can divided the kernel $K_j(x,y)$ into $j$ parts $K_j^k(x,y)$ with $1\le l \le j $, and we can prove that for each $1\le l\le j,$
       \begin{equation}\label{eq-subspace-j-l}
       \Big|K_j^l(x,y)\Big|\lesssim \frac{(1+|y|)^{a_l}}{(1+|x|)^{b_l}|x-y|^{c_l}}
       \end{equation}
        with all $a_l,b_l,c_l\in\mathbf{N}$ and $b_l+c_l\ge n-2m+j-1.$
        Then we can obtain that
        \begin{equation}\label{eq-subspace-j}
        \Big|K_j(x,y)\Big|\lesssim \sum_{l=1}^j\Big|K_j^l(x,y)\Big|\lesssim \frac{(1+|y|)^{\tilde{a}_j}}{(1+|x|)^{\tilde{b}_j}|x-y|^{\tilde{c}_j}}
        \end{equation}
        with all $\tilde{a}_j,\tilde{b}_j,\tilde{c}_j\in\mathbf{N}$ and $\tilde{b}_j+\tilde{c}_j\ge n-2m+j-1.$

        Thus Lemma \ref{Riesz-potential-boundedness} shows that $\psi\in W_{2m-\frac{n}{2}-(j-1)}(\mathbf{R}^n)$ with $5\le j\le k+1$ as desired. In particular, if $j=k+1$, since $n=4m+1-2k,$ and by our definition of $W_{\sigma}(\mathbf{R}^n),$ we have $\psi\in W_{2m-\frac{n}{2}-k}(\mathbf{R}^n)\subset L^2(\mathbf{R}^n).$

        On the other hand, for $5\le j\le k+1,$ if $\phi(x)=Uv(x)\psi(x)$ with $\psi(x)\in W_{2m-\frac{n}{2}-(j-1)}(\mathbf{R}^n),$ similar to the process for $j=3,$ and by Lemma \ref{S-projection}, we can obtain that $\phi(x)\in S_jL^2(\mathbf{R}^n).$
	\end{proof}

	\begin{proof}[\bf Proof of Proposition \ref{classification-2m-even} ($2m< n \le 4m$ and even)]
    Since the proof of Proposition \ref{classification-2m-even} is completely similar to the Proposition \ref{classification-2m-odd}, so we omit it.	
	\end{proof}
	
{\bf Acknowledgements:} The authors would like to express their sincere gratitude to the reviewing referee for his/her many constructive comments which helped us greatly improve the previous version. H. L. Feng is supported by the China  Postdoctoral Science Fundation, Grant No.2019M653135. A. Soffer is partially supported by NSFC grant No.11671163 and NSF grant DMS-1600749.  X. H. Yao  is partially supported by  NSFC (No.11771165) and the program for Changjiang Scholars and Innovative Research Team in University (IRT13066). 	Part of this work was done while the second author was a visiting professor at Central China Normal University (CCNU). The authors wolud like to thank Dr. S. L. Huang for interesting discussions, who also obtained a similar result about the absence of positive eigenvalue.

%\nocite{*}

%\bibliographystyle{amsalpha}
%\bibliography{referencelib}

\end{document}